\documentclass[11pt]{article}

\usepackage{lipsum}

\usepackage[utf8]{inputenc} %
\usepackage[T1]{fontenc}    %
\usepackage[colorlinks]{hyperref}       %
\usepackage{url}            %
\usepackage{booktabs}       %
\usepackage{nicefrac}       %
\usepackage{microtype}      %

\usepackage{amsthm,amsfonts,amsmath,amssymb,epsfig,color,float,graphicx,verbatim, enumitem}

\usepackage{algpseudocode,algorithm,algorithmicx}  

\usepackage{bbm}
\usepackage{caption}

\def\E{\mathbb E}
\def\P{\mathbb P}
\def\R{\mathbb R}

\def\N{\mathbb N}

\newcommand\numberthis{\addtocounter{equation}{1}\tag{\theequation}}

\newcommand{\1}{\mathbbm{1}}

\newcommand{\cN}{\mathcal{N}}
\newcommand{\cL}{\mathcal{L}}
\newcommand{\cT}{\mathcal{T}}
\newcommand{\cE}{\mathcal{E}}
\newcommand{\cS}{\mathcal{S}}
\newcommand{\cC}{\mathcal{C}}
\newcommand{\cX}{\mathcal{X}}

\def\argmin{\qopname\relax n{argmin}}

\newcommand{\trace}{\operatorname{tr}}
\newcommand{\HT}{\operatorname{HT}}
\newcommand{\cov}{\operatorname{Cov}}
\newcommand{\Var}{\operatorname{Var}}

\newcommand{\real}{\ensuremath{\mathbb{R}}}

\newcommand{\littlesum}{\mathop{\textstyle \sum}}

\newtheorem{theorem}{Theorem}[section]

\newtheorem{lemma}[theorem]{Lemma}

\newtheorem{proposition}[theorem]{Proposition}

\newtheorem{definition}{Definition}

\newtheorem{remark}[theorem]{Remark}

\theoremstyle{definition}
\newtheorem{assumption}{Assumption}

\usepackage{color}
\definecolor{Red}{rgb}{1,0,0}
\definecolor{Blue}{rgb}{0,0,1}
\definecolor{Olive}{rgb}{0.41,0.55,0.13}
\definecolor{Green}{rgb}{0,1,0}
\definecolor{MGreen}{rgb}{0,0.8,0}
\definecolor{DGreen}{rgb}{0,0.55,0}
\definecolor{Xellow}{rgb}{1,1,0}
\definecolor{Cyan}{rgb}{0,1,1}
\definecolor{Magenta}{rgb}{1,0,1}
\definecolor{Orange}{rgb}{1,.5,0}
\definecolor{Violet}{rgb}{.5,0,.5}
\definecolor{Purple}{rgb}{.75,0,.25}
\definecolor{Brown}{rgb}{.75,.5,.25}
\definecolor{Grey}{rgb}{.5,.5,.5}

\title{Robust regression with covariate filtering: \\ Heavy tails and adversarial contamination}

\author{
Ankit Pensia\\
University of Wisconsin-Madison\\
{\tt ankitp@cs.wisc.edu}\\
\and
Varun Jog\\
University of Cambridge\\
{\tt vj270@cam.ac.uk}
\and
Po-Ling Loh\\ University of Cambridge\\
{\tt pll28@cam.ac.uk}
}

\begin{document}

\maketitle

\begin{abstract}
We study the problem of linear regression where both covariates and responses are potentially (i) heavy-tailed and (ii) adversarially contaminated.
Several computationally efficient estimators have been proposed for the simpler setting where the covariates are sub-Gaussian and uncontaminated; however, these estimators may fail when the covariates are either heavy-tailed or contain outliers.
In this work, we show how to modify the Huber regression, least trimmed squares, and least absolute deviation estimators to obtain estimators which are simultaneously computationally and statistically efficient in the stronger contamination model.
Our approach is quite simple, and consists of applying a filtering algorithm to the covariates, and then applying the classical robust regression estimators to the remaining data. We show that the Huber regression estimator achieves near-optimal error rates in this setting, whereas the least trimmed squares and least absolute deviation estimators can be made to achieve near-optimal error after applying a postprocessing step.
\end{abstract}

\section{Introduction}

Robust linear regression is a well-studied topic in statistics, both from the viewpoint of theory and practice~\cite{HubRon11, HamEtal11, MarEtal19}. It has long been observed that the introduction of even a handful of outliers can massively affect the quality of a regression estimator; furthermore, high-leverage points, which are outlying in terms of their covariate values, have the potential for even more drastic consequences. Various methods have been proposed to alleviate the effect of outliers in the data, including diagnostic tests which focus on identifying and removing outliers~\cite{CooWei82}. On the other hand, such methods are mostly heuristic and few theoretical results exist in this area.

Much classical work in robust linear regression focuses on developing and analyzing estimators that are applied aggregately to an entire data set and are relatively insensitive to certain types of perturbations in the data. These estimators include different families of $M$-estimators~\cite{Hub73}, $GM$-estimators~\cite{Mal75}, $S$-estimators~\cite{RouYoh84}, and $MM$-estimators~\cite{Yoh87}, among others. Notably, most of the corresponding statistical theory has focused on analyzing i.i.d.\ data, often assumed to be drawn from a mixture distribution involving the parametric model and a (possibly heavy-tailed) contaminating distribution. Recent years have seen a flurry of activity on the somewhat different topic of adversarial contamination---spurred by advances in the theoretical computer science community and motivated by modern machine learning applications---and several approaches have subsequently been proposed for estimating the mean of a multivariate distribution~\cite{DiaKan19}. An interesting question which has remained largely unaddressed is whether simpler and seemingly more straightforward approaches such as $M$-estimation can be proven to achieve similar error guarantees as the more complicated proposals which have emerged from this line of work.

On the topic of $M$-estimation, Sasai and Fujisawa~\cite{SasFuj20} recently derived bounds for linear regression with a Huber loss when adversarial contamination may be present in the response variables. Slightly earlier analysis from Bhatia et al.~\cite{BhaJK15, BhaJKK17} provided guarantees for the popular least trimmed squares estimator~\cite{Rou84} with adversarially contaminated responses. In contrast, no analogous error bounds have been furnished for the behavior of these or other estimators when the covariates are adversarially contaminated. Rather, a series of classical results on the low breakdown point of regression estimators~\cite{Dav93} established the rather pessimistic message that adversarially contaminating even a single data point in both covariates and responses may have an unbounded effect on the accuracy of a convex $M$-estimators such as the Huber or least absolute deviation regression estimators (see, e.g., Maronna et al.~\cite{MarEtal19} and the references cited therein). Of course, the difficulty in using nonconvex loss functions is that nontrivial challenges arise in optimization.

We note, however, that the failure of simple $M$-estimation assumes that all the points are included in the estimation procedure, whereas a grossly outlying point might easily be flagged before fitting a moderately robust estimator on the remaining data. In Huber's textbook~\cite[p.\ 152]{HubRon11}, we find the following comment: ``Undoubtedly, a typical cause for breakdown in regression are gross outliers in the carrier $X$. In the robustness literature, the problem of leverage points and groups has therefore been tackled by so-called high breakdown point regression\dots. I doubt that this is the proper approach\dots. In my opinion, if there are sizable minority components, the task of the statistician is not to suppress them, but to disentangle them." However, the literature on how to perform outlier removal in a theoretically rigorous manner is fairly sparse.

Regarding heavy-tailed distributions, the ordinary least squares estimator may be shown to be highly suboptimal when the additive errors are allowed to be heavy-tailed (cf.\ Proposition~\ref{PropLowOLSMulti} in the appendix). Concretely, in a setting with $p$ parameters, $n$ data points, and noise variance $\sigma^2$, the $\ell_2$-error of the ordinary least squares estimator may increase as $\Theta\left(\sigma \sqrt{\frac{p}{n\tau}}\right)$ with probability $\tau$---in contrast to the error bound $O\left(\sigma\sqrt{\frac{p}{n}} + \sigma\sqrt{\frac{\log(1/\tau)}{n}}\right)$, which may be achieved under sub-Gaussian distributional assumptions. Starting from the seminal work of Catoni~\cite{Cat12}, the topic of heavy-tailed estimation has been an active area of research in theoretical statistics in recent years~\cite{Men15,MenZhi20,LugMen19-tour,LecLer20,Hop20,LugMen19-survey,DepLec19,HsuSab16}, and for regression, Lugosi and Mendelson~\cite{LugMen19-tour, LugMen19-survey} introduced an estimator based on a median-of-means algorithm which achieves the sub-Gaussian error rate even in heavy-tailed scenarios, provided $n = \Omega(p)$.
On the other hand, the proposed estimator has running time which is exponential in the dimension, hence is not computationally feasible for large $p$. More recently, Cherapanamjeri et al.~\cite{CheHRT20} proposed a polynomial-time estimator with the desired error rate when $n = \tilde{\Omega}\left(p \sqrt{\log(1 / \tau)}\right)$. However, the estimator requires the covariates to satisfy a stronger condition: a sum-of-squares (SOS) certifiable proof of degree $8$. The proposed algorithm uses an SOS hierarchy and involves solving a large semidefinite program which, although achievable in polynomial time, is not very practical.

\subsection{Our contributions}

In this paper, we take a cue from the literature on robust mean estimation under adversarial contamination, in which the proposed algorithms implicitly involve a filtration or screening step to identify and remove outlying data points, after which a (weighted) empirical mean is computed on the remaining data~\cite{LaiRV16, DiaKKLMS16-focs} (cf.\ Section~\ref{SubSecRobMean} below). The success of these filtering-based algorithms stems from a useful lemma which states that when the distribution of the uncontaminated data is isotropic, the empirical mean of a set of data points which have an approximately isotropic
 empirical covariance matrix will be close to the true mean. The filtering mechanism consequently operates by iteratively removing data points until the remaining set is approximately isotropic---theoretically, one can show that the proposed filters do not remove too many uncontaminated data points, while removing any adversarially introduced outliers that move the sample mean sufficiently far from the true mean. A key insight of this paper is that the condition of approximate isotropy of the empirical covariance (also known as stability) is in fact a sufficient condition for the success of classical robust regression estimators such as the Huber $M$-estimator, least trimmed squares (LTS), and least absolute deviation (LAD) estimator. Thus, an adversarially contaminated data set may first be preprocessed by applying a filter to the covariates, and then the classical estimator may be applied to the remaining data to obtain an overall estimate close to the true regression vector. A careful analysis shows that this method can be applied to data sets which possess adversarial contamination in \emph{both} the covariates and responses. Furthermore, the same method can be used to obtain error guarantees for heavy-tailed covariates and/or responses. Perhaps it is unsurprising that both adversarial contamination and heavy-tailed distributions may be treated using similar estimators, since in the latter case, ``outlying" points may be seen as occurring due to randomness naturally present in the sample rather than having been introduced adversarially.

 We will assume throughout our paper that prior to contamination, the covariates are drawn from a distribution with mean zero and identity covariance and also satisfies a property known as hypercontractivity (bounded fourth moments). We will also assume that the additive noise in the linear model is independent of the covariates and (in most cases) has finite first or second moments. Note that these assumptions are significantly less restrictive than the usual assumptions of sub-Gaussianity, and include various heavy-tailed distributions, as well. Under these assumptions, we can show that the Huber estimator after filtering achieves the optimal $\ell_2$-error rate of $O\left(\sigma \sqrt{\frac{p}{n}} + \sigma \sqrt{\frac{\log(1/\tau)}{n}}\right)$, provided the sample size satisfies $n = \Omega(p \log p)$. Furthermore, our method is computationally feasible, since we simply need to perform the iterative filtering algorithm, followed by optimization of a convex objective function. If adversarial contamination is introduced to the covariates and/or response variables, the error bound of the filtered Huber estimator becomes $O\left(\sigma \left(\sqrt{\frac{p \log p}{n}} + \sqrt{\frac{\log(1/\tau)}{n}} + \epsilon^{1-1/k}\right)\right)$, provided $n = \Omega(p \log p)$
 and the covariates satisfy an additional $k^{\text{th}}$ moment bound, for $k \ge 4$. Note that the dependence on $\epsilon$ matches the lower bound derived in Bakshi and Prasad~\cite{BakPra20}. When the covariates are drawn from a Gaussian distribution with identity covariance, the error rate of the filtered Huber estimator further improves to $O\left(\sigma \left(\sqrt{\frac{p}{n}} + \sqrt{\frac{\log(1/\tau)}{n}} + \epsilon \sqrt{\log(1/\epsilon)}\right)\right)$, provided $n = \Omega(p)$. The dependence on $p$, $n$, and $\tau$ is optimal, while the dependence on $\epsilon$ is nearly-optimal up to a $\sqrt{\log(1/\epsilon)}$ factor~\cite{CheGR16}. (This rate also shaves off the additional $\sqrt{\log(1/\epsilon)}$ factor achieved in previous works~\cite{DiaKS19,CheATJFB20}, which obtained the rate $O(\epsilon \log(1/\epsilon))$ in terms of $\epsilon$.)
Going back to the heavy-tailed setting, i.e., when the covariates are drawn from a distribution with mean zero and bounded fourth moments, we extend our analysis to the setting when the covariance matrix $\Sigma$ of the covariates is unknown but satisfies the bound $(1/2)I\preceq \Sigma \preceq 2 I$. In this setting, we show that the filtered Huber estimator achieves the error rate $O\left(\sigma \left(\sqrt{\frac{p \log p}{n}} + \sqrt{\frac{\log(1/\tau)}{n}} + \sqrt{\epsilon}\right)\right)$, provided $n = \Omega(p \log p)$. The SQ lower bound of Diakonikolas et al.~\cite{DiaKS19} suggests that such a dependence on $\epsilon$ is essentially optimal when $n = o(p^2)$.

We derive error bounds for the LTS and LAD estimators under slightly different assumptions: When the noise distribution has bounded $(k')^{\text{th}}$ moments, for some $k' \ge 2$, we  obtain an error rate of the form $O\left(\sigma \left(\frac{p \log p}{n} + \epsilon + \frac{\log(1/\tau)}{n}\right)^{1/2 - 1/k'}\right)$ for the LTS estimator, provided $n = \Omega(p \log p)$. Assuming a first moment bound of $\kappa$ on the noise distribution, we can show that the LAD estimator has $\ell_2$-error $O(\kappa)$, provided $n = \Omega(p \log p)$. Although the error bounds for the LTS and LAD estimators are somewhat weaker than the bounds we obtain for the Huber regression estimator, we note that the LTS estimator is extremely quick to compute in practice~\cite{BhaJK15,BhaJKK17},
and the LAD estimator does not involve any tuning parameters, unlike the Huber estimator (which requires a tuning parameter for the loss) and the LTS estimator (which requires a tuning parameter specifying the degree of trimming). %
Furthermore, we show that a simple postprocessing step involving applying the robust multivariate mean algorithm to a shifted data set can be used to obtain near-optimal error guarantees in terms of $\tau$ and $p$.
Lastly, we note that the LTS or LAD estimators may be practically useful for initializing a gradient descent algorithm when optimizing the Huber regression objective in order to save on computation.

\subsection{Related work}

Several recent works have highlighted significant challenges that appear in the presence of heavy-tailed responses and/or adversarial contamination in responses~\cite{LasDB09,NasTN11,NguTra13,BhaJK15,MukGJK19,SasFuj20,WanLJ07}.
In all of these works, the covariates are assumed to satisfy strong assumptions: sub-Gaussian tails and no contamination.
The preceding works can be loosely categorized into two categories: (i) regularization-based estimators and (ii) thresholding-based estimators.
In the first category, a popular choice is a penalized Lasso-type estimator that solves  the following optimization problem:
\begin{align*}
\min_{\beta,z} \left\{\frac{1}{n}\|y - X \beta - z\|_2^2 + \lambda \|z\|_1\right\},
\end{align*}
where the variable $z$ accounts for outliers in the response variables. Several works have shown that Lasso-type estimators can handle contamination or heavy-tailed noise in responses~\cite{NguTra13,SasFuj20}---indeed, Huber regression is closely related to penalized Lasso-type estimators~\cite{SheOwe11,SasFuj20}.
The idea of using the Huber loss for estimation under heavy-tailed error distributions has recently been studied in the context of mean estimation~\cite{Cat12,Min19-uni} and regression~\cite{FanLW17,SunZF20}.
Our work on Huber regression is closely related to Sun et al.~\cite{SunZF20}, and we roughly follow their proof structure. However, we establish significantly tighter results for heavy-tailed covariates (see Section~\ref{SecHuber} for more details).

Another popular convex estimator is the LAD estimator with a Lasso penalty~\cite{WanLJ07,KarPri19}.
In the dense setting, Karmalkar and Price~\cite{KarPri19} (see also Dwork et al.~\cite{DwoMT07}) studied the LAD estimator $\min_{\beta}  \|y - X \beta \|_1 $, and showed its robustness to adversarial contamination in the responses.
However, their theory imposes a deterministic condition on the covariates that can be shown to hold with high probability for sub-Gaussian distributions, but does not necessarily hold for heavy-tailed or corrupted covariates.
As opposed to convex relaxation-based estimators, several recent works have studied alternating minimization algorithms for robust regression~\cite{JaiKar17,BhaJK15,BhaJKK17,JaiTK14}. These algorithms were developed to optimize the nonconvex objective function corresponding to the LTS estimator~\cite{Rou84}.
In our paper, we critically leverage the aforementioned results on LAD~\cite{KarPri19} and LTS~\cite{BhaJK15, BhaJKK17} estimation by showing that the deterministic conditions under which the respective algorithms are guaranteed to succeed are satisfied with high probability by our preprocessed covariates.

Turning to papers which analyze the setting involving corruption in both covariates and responses, a general framework for robust convex optimization was considered in Diakonikolas et al.~\cite{DiaKKLSS19} and Prasad et al.~\cite{PraSBR20} using the robust mean estimation algorithm on gradients of the loss function.
Although these results lead to polynomial-time estimators for several tasks, the resulting rates are suboptimal for linear regression.
In the Gaussian setting, Diakonikolas et al.~\cite{DiaKS19} proposed computationally efficient estimators with near optimal-error guarantees under adversarial contamination in both covariates and responses.

In concurrent work, Zhu et al.~\cite{ZhuJS20} and Bakshi and Prasad~\cite{BakPra20} studied computationally-efficient algorithms for heavy-tailed robust regression in a more general setting, where the covariance $\Sigma$ of the covariates is unknown and the noise may not be independent,
with the goal of minimal dependence on the level of adversarial contamination $\epsilon$.
Initiated by Klivans et al.~\cite{KliKM18}, their algorithms are based in a sum-of-squares framework, and impose a \emph{certifiable} hypercontractivity assumption on covariates, which is a somewhat more restrictive than our assumption of hypercontractivity~\cite{KotSte17a, KotSte17b}.
As the goal in these works is slightly different, the resulting estimators have suboptimal dependence on sample complexity and probability of error in comparison to ours.

Recently, Cherapanamjeri et al.~\cite{CheATJFB20} and Depersin~\cite{Dep20} considered the case of covariates with bounded fourth moments, and proposed an iterative gradient based procedure for robust regression.
When $\Sigma$ is unknown and the noise is independent,
Cherapanamjeri et al.~\cite{CheATJFB20} obtained
a near-linear time estimator (when $\epsilon$ is constant) with near-optimal sample complexity, but with a constant error probability.
Depersin~\cite{Dep20} studied the case of known $\Sigma$ and possibly dependent noise, and proposed a computationally efficient estimator with a sub-Gaussian error rate and a $O(\sqrt{\epsilon})$ dependence on $\epsilon$. However, the error guarantee for the estimator does not improve when higher-order moments are bounded.

We emphasize that the focus of our work is slightly different from the aforementioned works: we seek to show that several classical estimators \emph{which are known to be robust to corruptions in the responses} can also be made robust to corruptions in the covariates after a simple outlier filtration step.
For each of the Huber, LAD, and LTS estimators, our guarantees for heavy-tailed covariates (nearly) match their corresponding known results for sub-Gaussian covariates.
In addition, we highlight the fact that our filtered Huber estimator (cf. Theorem~\ref{ThmAdvHuberReg}) is the first known polynomial-time estimator that is near-optimal in all the parameters $\epsilon$, $p$, $\tau$, and $n$ for the case of isotropic covariates and independent noise.

\subsection{Organization}
The rest of the paper is organized as follows: In Section~\ref{SecPrelim}, we explain the problem setup and connection with robust mean estimation. In Section~\ref{SecHuber}, we analyze the Huber regression estimator.
We prove our results regarding the LTS and LAD estimators in Sections~\ref{SecLTS} and~\ref{SecLAD}, respectively.
Section~\ref{SecPP} contains the details regarding a postprocessing step which can be used to improve the accuracy of the LTS and LAD estimators.
Finally, Section~\ref{SecSim} contains simulation results reporting the effect of the proposed filtering step.
 Section~\ref{sec:discussion} concludes the paper with a short discussion of open questions.

\section{Background and problem setup}
\label{SecPrelim}

We begin by listing some notation that will be used throughout the paper. For a real-valued random variable $z$,
let  $\|z\|_{\psi_2}$ denote the sub-Gaussian norm of $z$. We use $[n]$ as a shorthand for $\{1,\ldots,n\}$. For a vector $b \in \R^n$ and $m\in[n]$, we say that $b$ is $m$-sparse if at most $m$ entries of $b$ are nonzero, and we also write $\|b\|_0 = m$. For $1 \le i \le n$, we write $|b|_{(i)}$ to denote the $i^{\text{th}}$ smallest component of $b$ according to magnitude.
Let $\cS^{n-1}$ denote the unit sphere in $n$ dimensions.
For a square matrix $M$, we use $\lambda_{\max}(M)$ and $\lambda_{\min}(M)$ to denote the largest and smallest eigenvalues, respectively. We use $\|M\|_2$ to denote the spectral norm. For two matrices $M_1, M_2$, we write $M_1 \succeq M_2$ to denote the fact that $M_1 - M_2$ is positive semidefinite.

For a differentiable function $f$, we use $\nabla f$  to denote its gradient.
For a scalar $x \in \R$, we use $\text{sgn}(x)$ to denote the sign of $x$, i.e., $\text{sgn}(x) = 0$ for $x = 0$; $\text{sgn}(x) = 1$ for $x > 0$; and $\text{sgn}(x) = -1$ for $x< 0$. For two sets $A$ and $B$, let $A \setminus B$
denote the set difference and let $A \triangle B$ denote the symmetric difference.
Let $\1(A)$ denote the indicator function over a set $A$.

We use $c,C,c_1,C_1,\ldots$ to denote absolute positive constants with values that might change from line to line.
We also use the standard big-$O$ notation to simplify the expressions in two regimes:
For two nonnegative functions $f$ and $g$ with domain $D$, we say that $f = O(g)$ when one of the following is true: (i) $D = \N$, and there exists constants $C$ and $n_0$ such that $f(n) \leq C g(n)$ for all $n \geq n_0$; or (ii) $D = [0,1]$, and there exists constants $C$ and $\epsilon_0 \in (0,1)$ such that $f(\epsilon) \leq C g(\epsilon)$ for $\epsilon \leq \epsilon_0$. 
The setting will be clear from context.  
We say that $f = \Omega(g)$ if $g = O(f)$, and we say that $f = \Theta(g)$ when $f = O(g)$ and $f = \Omega(g)$. We also use $\lesssim$ and $\gtrsim$ to hide constants.

We also recall the following definitions:
\begin{definition} (Hypercontractivity)
We say that a random vector $X \in \R^p$
\emph{satisfies $(k, 2)$-hypercontractivity with parameter $\sigma_k$}
if for all unit vectors $v \in \R^p$, we have
\begin{align*}
\left(\E |v^TX|^k\right)^{1/k} \leq \sigma_k \left(\E (v^TX)^2 \right)^{1/2}. 
\end{align*}
\end{definition}

\begin{definition}(Strong convexity)
\label{DefStrongCvx}
For a convex set $\cX \subseteq \R^n$, we say that a continuously differentiable function $f: \cX \to \R$ is \emph{$\alpha$-strongly convex} if for any $x, y \in \cX$, we have 
\begin{align}
f(y) \geq f(x) + \langle \nabla f(x) , y- x\rangle + \frac{\alpha}{2}  \|y - x\|_2^2.
\end{align}
\label{EqnStrongConv}
\end{definition}

\subsection{Linear model}

Suppose we have observations drawn from the linear model
\begin{align}
\label{EqnLinModel}
y_i = x_i^T \beta^* + z_i, \qquad 1 \le i \le n,
\end{align}
where $\beta^* \in \R^p$, the $x_i$'s are sampled i.i.d.\ from a distribution over $\R^p$, and the $z_i$'s are i.i.d.\ noise. We will also use the standard statistical notation to write equation~\eqref{EqnLinModel} as $y = X\beta^* + z$, where $y, z \in \real^n$, $\beta^* \in \real^p$, and $X \in \real^{n \times p}$.
Our goal is to estimate $\beta^*$ from the data set $S= \{(x_1,y_1),\ldots,(x_n,y_n)\}$. We make the following assumption about the distribution of the covariates:

\begin{assumption}
\label{AsCov}
The covariates satisfy  $\E x_i = 0$ and $\E x_ix_i^T = I$. Moreover, the covariates satisfy $(4,2)$-hypercontractivity with parameter
$\sigma_{x,4} \leq C$, for a known constant $C$.
\end{assumption}
Note that the case of a known, non-identity covariance matrix can be reduced to the setting of identity covariance via a linear transformation. We relax the condition of an identity covariance matrix to an unknown but \emph{bounded} covariance matrix in Section~\ref{SecUnknownCov}.

We assume an identity covariance structure in Assumption~\ref{AsCov} because of the computational statistical query (SQ) lower bound from Diakonikolas et al.~\cite{DiaKS19}, stating that in the case of an unknown covariance matrix, any computationally efficient SQ algorithm requires approximately $\Omega(p^2)$ samples to achieve an error rate of $o(\sqrt{\epsilon})$ in the strong contamination model (cf.~Theorem~\ref{ThmAdvHuberReg}).
We show that the filtered Huber estimator achieves the rate $O(\sqrt{\epsilon})$ in the unknown covariance setting in Section~\ref{SecUnknownCov}.
However, even with an identity covariance matrix, the covariates could have a degenerate distribution such that, with high probability, all the sampled points have norm $0$ and all information about $\beta^*$ would be lost. As a result, we also include the hypercontractivity condition in Assumption~\ref{AsCov}, which is a standard assumption in this field. Note that under the identity covariance assumption, the hypercontractivity condition can simply be written as $(\E (v^Tx_i)^4)^{1/4} \leq C$.

\begin{remark}
Note that the assumption that an upper bound $C$ on the hypercontractivity constant $\sigma_{x,4}$ is known is necessary for running the algorithms in this paper in practice (e.g., Algorithms~\ref{AlgHubGeneralCase}, \ref{AlgLTSFiltering}, and \ref{AlgLAD_filter} below), since our theory requires the filtering parameter $\epsilon'$ to be smaller than some value which depends on $C$.
\end{remark}

We also make the following assumption about the additive noise distribution:

\begin{assumption}
\label{AsNoise}
The noise variables $\{z_i\}$ are independent of the covariates $\{x_i\}$, and $\E z_i = 0$. %
\end{assumption}
The independence assumption on the $z_i$'s and $x_i$'s is somewhat restrictive, but we leave the study of more general distributions to future work.
We will relax this assumption on noise for a subset of our results: (i) Theorems~\ref{ThmDetHuberReg} and \ref{ThmStocHuberReg} hold even if the first moment of the $z_i$'s is infinite, and (ii) Theorem~\ref{ThmLAD} holds even if the $z_i$'s are dependent on $x_i$'s and have nonzero mean.

In the sequel, we also study the robustness of our estimators when a fraction of data points are adversarially contaminated. We formally define the contamination model of the adversary below:
\begin{definition}(Strong Contamination Model)
\label{DefContModel}
We say that a set $T$ is an \emph{$\epsilon$-corrupted version of a set $S$} if $|T| = |S|$ and $|T \cap S| \geq (1 - \epsilon)|S|$.
\end{definition}
This contamination model is called the \emph{strong contamination model} in the literature, since no computational or statistical restrictions are imposed on $T$. In contrast, Huber's $\epsilon$-contamination model requires the contamination mechanism to be oblivious and additive, i.e., it can only add outliers to the uncontaminated i.i.d.\ data without looking at the inliers.

\subsection{Stability conditions}

Our technical results will rely on appropriately defined notions of stability. Recall the following stability condition from the robust mean estimation literature~\cite{DiaKKLMS16-focs,DiaKKLMSl17,SteCV18,DonHL19,DiaKan19,CheDG19,CheDGS20}:
\begin{definition}(Strong stability)
\label{DefStab}
For $\epsilon < 1/2$, we say that a multiset $S= \{x_1,\dots,x_n\}$ \emph{satisfies $(\epsilon,\delta)$-stability} for $\epsilon \leq \delta$ with respect to $\mu$ and $\sigma^2$ if for all $S' \subseteq S$ such that $|S'| \geq (1 - \epsilon) n$, we have
\begin{enumerate}
	\item $\left\| \frac{1}{|S'|} \sum_{i \in S'} x_i - \mu \right\|_2 \leq \sigma \delta$, and
	\item $\left\| \frac{1}{|S'|} \sum_{i \in S'} (x_i - \mu)(x_i - \mu)^T - \sigma^2 I \right\|_2 \leq \frac{\sigma^2 \delta^2}{\epsilon}$.
\end{enumerate}
\end{definition}

Definition~\ref{DefStab} is designed for samples from a distribution with mean $\mu$ and covariance $\Sigma \preceq \sigma^2 I$. 
Note that a set which is $(\epsilon, \delta)$-stable is also $(\epsilon', \delta')$-stable for any $\epsilon' \le \epsilon$ and $\delta' \ge \delta$.
The $(\epsilon,\delta)$-stability condition states that for every large enough subset, (i) the $\ell_2$-distance between the empirical mean and $\mu$ is at most $\sigma \delta$, and (ii) the spectral distance between the (centered) second moment matrix and $\sigma^2 I$ is at most $\frac{\sigma^2 \delta^2}{\epsilon}$.
Since our primary focus will be on distributions with $\mu = 0$ and $\sigma^2 = 1$, we will not explicitly state these parameters when they are clear from context.

Next, we mention a deterministic condition on the covariates that appeared in the analysis of least trimmed squares regression in Bhatia et al.~\cite{BhaJK15}:

\begin{definition} (Weak stability)
\label{AsDetCov}
Let $\epsilon \in (0,1)$. The set $\{x_1, \dots, x_n\}$ \emph{satisfies $(\epsilon, L, U)$-weak stability} if for every subset $S \subseteq [n]$ such that $|S| \geq (1 - \epsilon) n$, the second moment matrix of $S$ is approximately isotropic, i.e.,
	\begin{align*}
L \leq 	 \lambda_{\min}\left(\frac{1}{n}\sum_{i \in S} x_ix_i^T \right) \leq   \lambda_{\max}\left( \frac{1}{n} \sum_{i \in S} x_ix_i^T \right) \leq U.
	\end{align*}
\end{definition}
	
Bhatia et al.~\cite{BhaJK15} established the convergence of an alternating minimization algorithm under the weak stability condition for a fixed $\epsilon$, provided (i) $L = \Theta(1)$ and (ii) $U = \Theta(1)$.
We will show in Section~\ref{SecHuber} that under the same conditions, Huber regression also succeeds with high probability.
This leads to the question of whether weak stability directly holds with high probability for heavy-tailed covariates; following arguments in Koltchinskii and Mendelson~\cite{KM15}, it can be shown that condition (i) holds with high probability~\cite{DiaKP20}.
However, known concentration results suggest that condition (ii) does \emph{not} hold with high probability for heavy-tailed covariates when $S = [n]$:
The usual matrix Chernoff bounds~\cite{Tro15} would yield $U = O(1)$ with probability $1 - \tau$ if $n = \Omega(p \log(1 / \tau))$, which may be much larger than the ideal sub-Gaussian sample complexity which is \emph{additive} rather than multiplicative in $p$ and $\log(1/\tau)$.

We note the following simple lemma, which shows that strong stability implies weak stability:
\begin{lemma}
\label{LemStrongWeak}
Let $S= \{x_1,\dots,x_n\}$ be an $(\epsilon,\delta)$-stable set with respect to $\mu=0$ and $\sigma^2 = 1$, such that $\frac{\delta^2}{\epsilon} < 1$. Then $S$ is also $(\epsilon,L,U)$-weakly stable with $L = (1 - \epsilon)\left(1 - \frac{\delta^2}{\epsilon}\right)$ and $U = 1 + \frac{\delta^2}{\epsilon}$. In particular, if $\frac{\delta^2}{\epsilon} < 0.5$, we have $L = \Omega(1)$ and $U = O(1)$.
\end{lemma}

\begin{proof}
By the definition of strong stability and the triangle inequality, we clearly have
\begin{equation*}
\left\|\frac{1}{n}\sum_{i \in [n]} x_ix_i^T\right\|_2 \leq 1 + \frac{\delta^2}{\epsilon},
\end{equation*}
showing that we can take $U = 1 + \frac{\delta^2}{\epsilon}$.

For the lower bound, consider a subset $S \subseteq [n]$ such that $|S| \geq (1 - \epsilon)n$.
By the stability condition, we know that for any unit vector $v$, we have
\begin{equation*}
v^T \left(I - \frac{1}{|S|} \sum_{i \in S} x_i x_i^T\right) v \le \frac{\delta^2}{\epsilon},
\end{equation*}
implying that
\begin{equation*}
\frac{n}{|S|} \cdot v^T \left(\frac{1}{n} \sum_{i \in S} x_i x_i^T\right) v \ge 1 - \frac{\delta^2}{\epsilon}.
\end{equation*}
Hence,
\begin{equation*}
\lambda_{\min}\left(\frac{1}{n} \sum_{i\in S}x_ix_i^T\right) \geq \frac{|S|}{n}\left(1 - \frac{\delta^2}{\epsilon}\right) \geq (1 - \epsilon)\left(1 - \frac{\delta^2}{\epsilon}\right),
\end{equation*}
giving the desired result. The second result follows by noting that $\epsilon<1/2$.
\end{proof}

Bhatia et al.~\cite{BhaJK15} also defined the following notions in their analysis of LTS:
\begin{definition} (SSC and SSS)
\label{DefSS}
Let $x_1,\dots,x_n$ be $n$ points in $\R^p$. 
For $m \in [n]$, we say that the $x_i$'s satisfy the \emph{Subset Strong Convexity (SSC) property at level $m$ with parameter $\lambda_m$} if 
\begin{align*}
\lambda_m \leq \min_{S \subseteq [n]:|S| = m} \lambda_{\min}\left(\sum_{i \in S} x_ix_i^T\right).
\end{align*}
We say that the $x_i$'s satisfy the \emph{Subset Strong Smoothness (SSS) property at level $m$ with parameter $\Lambda_m$} if 
\begin{align*}
\max_{S \subseteq [n]:|S| = m} \lambda_{\max}\left(\sum_{i \in S} x_ix_i^T\right) \leq \Lambda_{m}.
\end{align*}
\end{definition}
Note that if a set satisfies $(\epsilon, L, U)$-weak stability, then it satisfies the SSC and SSS properties at level $(1-\epsilon)n$ with parameters $nL$ and $nU$, respectively. However, the results of Bhatia et al.\ (cf.\ Lemma~\ref{LemAltMin} below) require finer control of the minimum and maximum eigenvalues at different levels, in addition to the assumption of weak stability.

Our final notion of stability comes from Karmalkar and Price~\cite{KarPri19}:
\begin{definition} ($\ell_1$-stability)
\label{DefL1Stable}
We say a set of data points $\{x_1,\dots,x_n \}\subseteq \R^p$ satisfies \emph{$(m,M, \epsilon, \ell_1)$-stability} if for all subsets $S \subseteq [n]$ with $|S| \geq (1 - \epsilon) n$ and all unit vectors $v \in \R^p$, the following two conditions are satisfied: 
\begin{enumerate}
	\item $\frac{1}{n} \sum_{i \in S} |x_i^Tv| \geq M$, and
\item 
$\frac{1}{n} \sum_{i \in [n] \setminus S} |x_i^Tv| \leq m$.
\end{enumerate}
\end{definition}
Note that this definition of stability controls the $\ell_1$-norm of projections, whereas weak stability (or strong stability) is a statement about $\ell_2$-norms. This notion of stability was used by Karmalkar and Price~\cite{KarPri19} in their analysis of the LAD estimator, and will also be used in our analysis of the LAD estimator in the present paper.
As shown later (cf.\ Lemma~\ref{PropStabL1Error}), the upper bound in the definition of $\ell_1$-stability can be derived directly from strong stability.

\subsection{Iterative filtering algorithm}
\label{SubSecRobMean}

A recent line of work in the robust mean estimation literature has led to various algorithms that succeed when the stability condition holds (see Diakonikolas and Kane~\cite{DiaKan19} for a recent survey).
We choose to work with the iterative filtering algorithm with independent removal~\cite{DiaKan19}:
\begin{theorem}(Diakonikolas and Kane~\cite{DiaKan19})
\label{ThmStability}
Let $\epsilon < 1/2$, and suppose $S \subseteq \R^p$ is a multiset such that there exists a subset $S' \subseteq S$ such that (i) $|S'|\geq (1 - \epsilon)|S|$ and (ii) $S'$ is $(C\epsilon,\delta)$-stable with respect to $\mu$ and $\sigma^2$ for a large enough constant $C>1$.
Let $T$ be an $\epsilon$-corrupted version of the set $S$.
Then there exists a computationally efficient algorithm that, given $T$ and $\epsilon$ as inputs, with probability at least $1 - O(\exp(-\Omega(n \epsilon)))$, outputs a multiset $T' \subseteq T$ such that (i) $|T'| \geq (1 -  c_1\epsilon)|T|$ and  (ii) $T'$ is $( c_2C\epsilon, c_3\delta)$-stable with respect to $\mu$ and $\sigma^2$.
\end{theorem}

\begin{remark}
\label{RemStableMean}
Note that by the definition of stability, the empirical mean of an $(\epsilon, \delta)$-stable set lies within $\sigma\delta$ of $\mu$. Thus, Theorem~\ref{ThmStability} provides a high-probability error bound on the empirical mean of the filtered data points, when the original data set is an $\epsilon$-corrupted version of a data set containing a large stable subset.
\end{remark}

Stability-based algorithms use the fact that if the empirical covariance matrix has a small spectral norm, then the empirical mean is itself a good estimate of $\mu$.
The algorithm mentioned in Theorem~\ref{ThmStability} uses this insight to obtain a subset of cardinality $(1 - O(\epsilon))n$ such that the resulting empirical covariance matrix has a small spectral norm, by iteratively removing a certain fraction of points.
At a high level, in each iteration, the algorithm uses the projection of the points along the leading eigenvector of the empirical covariance matrix (of the remaining points) to define a distribution over the (remaining) points such that the probability mass over the outliers is greater than the mass over the inliers. 
This distribution is then used to remove points stochastically, so that at each iteration, the algorithm is more likely to remove outliers than inliers.
Since the number of outliers is at most $\epsilon n$, it does not remove too many inliers.
Whereas prior work has focused on using the filtering algorithm mentioned in Theorem~\ref{ThmStability} as a subroutine to find an estimate $\widehat{\mu}$ for $\mu$ (or, more generally, to robustly estimate the gradient of a function), we emphasize that our motivation in applying the filtering algorithm is to \emph{identify a subset} $T'$ that satisfies weak stability---indeed, mean estimation is unnecessary because we already know the covariate distribution is centered around 0.

The probability of success of our preprocessing step will depend on the probability of success of Theorem~\ref{ThmStability} applied to i.i.d.\ data from a distribution satisfying Assumption~\ref{AsCov}.
We will use the following recent result from Diakonikolas et al.~\cite{DiaKP20}, which provides a useful guarantee for when the condition of Theorem~\ref{ThmStability} is satisfied with high probability:

\begin{theorem}(Diakonikolas et al.~\cite{DiaKP20})
\label{ThmStabHighProb}
Let $S$ be a set of $n$ i.i.d.\ points from a distribution in $\R^p$ with mean $\mu$ and covariance $I$.
Further assume that the distribution satisfies $(k,2)$-hypercontractivity with parameter $\sigma_k$, for some $k \geq 4$.
Let $\epsilon$ and $\tau$ be such that $\epsilon' = C\left(\epsilon + \frac{\log(1/\tau)}{n}\right) = O(1)$, for a large enough constant $C$. 
Then with probability at least $1 - \tau$,  there exists a subset $S' \subseteq S$ such that $|S'| \geq (1 - \epsilon')|S|$ and $S'$ is $(C_1\epsilon',\delta)$-stable, where $C_1 > 2$ is any large constant and $\delta = O\left(\sqrt{\frac{p\log p}{n}} + \sigma_k \epsilon^{1 - \frac{1}{k}} + \sigma_4\sqrt{\frac{\log(1/ \tau)}{n}}\right)$ with prefactor depending on $C_1$.
\end{theorem}

Combining the two theorems above, we see that with probability $1 - \tau$, we can identify a large subset $S' \subseteq S$, in a computationally efficient manner, such that $S'$ is $(O(\epsilon),\delta)$-stable for an appropriate choice of $\epsilon$ and $\delta$ as specified by Theorem~\ref{ThmStabHighProb}. This rather technical conclusion is the starting point of our work.

\section{Huber regression}
\label{SecHuber}

In this section, we will study Huber's loss for regression.
The Huber loss with parameter $\gamma$ is defined as follows:
\begin{align*}
\ell_\gamma(x) = \begin{cases} \frac{x^2}{2}, & \text{ if } |x| \leq \gamma, \\
			\gamma |x| - \frac{\gamma^2}{2}, & \text{ if } |x| > \gamma.
			\end{cases}
\end{align*}
This loss function has a long history in robust statistics, starting from the seminal work of Huber~\cite{Hub64,HubRon11}.
Let $\psi_\gamma(x) = \nabla \ell_\gamma(x)$ be the gradient of Huber's loss:
\begin{align*}
\psi_\gamma(x) = \begin{cases} x, & \text{ if } |x| \leq \gamma, \\
			\gamma (\text{sgn}(x))  & \text{ if } |x| > \gamma.
			\end{cases}
\end{align*}
We now define $\cL_\gamma(\beta) := \frac{1}{n} \sum_{i \in [n]} \ell_\gamma(y_i - x_i^T \beta)$ and let Huber's $M$-estimator be defined as
\begin{equation*}
\widehat{\beta}_{H,\gamma} = \argmin_{\beta} \cL_\gamma(\beta).
\end{equation*}

Note that the Huber objective function is convex, so it is possible to (approximately) obtain the minimizer $\widehat{\beta}_{H, \gamma}$ in a computationally feasible manner. Thus, we will begin by analyzing statistical properties of the Huber regression estimator and then comment only briefly on optimization (cf.\ Section~\ref{SubSecHubRunTime}). We present our statistical analysis in increasing levels of complexity: fixed design covariates satisfying weak stability and i.i.d.\ symmetric noise (Section~\ref{SecSym}), random i.i.d.\ covariates and asymmetric noise (Section~\ref{SecHuberGeneral}), and adversarially contaminated data (Section~\ref{SecHuberAdv}).

\subsection{Fixed design and symmetric noise}
\label{SecSym}

Our main result in this subsection is the following:
\begin{theorem}
\label{ThmDetHuberReg}
Suppose we have $n$ i.i.d.\ samples from the following (fixed design) model: $y_i = x_i^T \beta^* + z_i$,
where the covariates $\{x_i\}$ satisfy weak stability with some $\epsilon$, $L$, and $U$.
Suppose the errors $\{z_i\}$ are sampled independently from a \textit{symmetric} distribution.
Let $\widehat{\beta}_{H,\gamma} \in \arg\min \cL_\gamma(\beta)$. 
Let $\tau$ be such that $\frac{\log(1/ \tau)}{n} = O(\epsilon)$.
  Then setting $\gamma$ such that $\P(|z_i| \geq \gamma/2) = O(\epsilon)$, we have, with probability at least $1 - \tau$,
\begin{align*}	
	\|\widehat{\beta}_{H,\gamma} - \beta^*\|_2 \lesssim \frac{ \gamma \sqrt{U}}{ L}\left( \sqrt{\frac{p}{n}} +  \sqrt{\frac{\log (1 / \tau)}{n}} \right), \,\, \text{ as long as } n = \Omega\left(\frac{U^2(p + \log(1/\tau))}{L^2 \epsilon^2}\right).	
	\end{align*}
Furthermore, $\cL_ \gamma(\beta)$ is $L$-strongly convex in a ball of radius $\Omega(\epsilon \gamma / \sqrt{U})$ around $\widehat{\beta}_{H,\gamma}$.
\end{theorem}

Theorem~\ref{ThmDetHuberReg} provides an error bound on the Huber regression estimator under a deterministic condition on the covariates; the probabilistic nature of the theorem comes from the randomness in the additive errors, which are assumed to be drawn from a symmetric noise distribution. In Theorems~\ref{ThmStocHuberReg} and~\ref{ThmAdvHuberReg} below, we will show that the weak stability condition holds with high probability when the covariates are drawn from possibly heavy-tailed, possibly contaminated distributions and then passed through a filtering algorithm. We will also show how to relax the assumption that the distribution of $z_i$ is symmetric via an appropriate preprocessing step.

\begin{remark}
When $\Omega(1)= L \leq U = O(1)$ and $ \epsilon = \Omega(1)$, the sample complexity reduces to $n = \Omega(p)$ (by assumption, $n = \Omega(\log(1/ \tau))$).
Also, the radius of strong convexity is then $ \Omega(\gamma)$.
\end{remark}
\begin{remark}
\label{RemarkNoiseFirstMoment}
Note that Theorem~\ref{ThmDetHuberReg} does not require the additive noise to have finite moments. 
If the noise distribution has a finite $k^{\text{th}}$ moment, however, Markov's inequality implies that we can always set $\gamma = \Omega(\epsilon^{-1/k} (\E |z_i|^k)^{1/k})$. In particular, if the $z_i$'s have a finite variance $\sigma^2$, we can take $\gamma = \Omega(\sigma/ \sqrt{\epsilon})$.
\end{remark}

The assumption that $\P(|z_i| \ge \gamma/2) = O(\epsilon)$ implies that the parameter $\gamma$ used to define the Huber loss needs to be sufficiently large in order for our theory to succeed, in a sense being calibrated to the tail behavior of the error distribution. Indeed, the heavier the tails of the $z_i$'s, the larger $\gamma$ would need to be, leading to a worse error bound.
 Since it is generally unreasonable to assume that the scale of the additive noise distribution is known in practice, we will discuss methods for adaptively choosing $\gamma$ from the data in our results below.

\begin{proof}

We will follow the proof structure of Sun et al.~\cite{SunZF20}. The proof relies on the fact that $\cL_\gamma(\beta)$ is a convex function. We first show (Lemma~\ref{LemmaGradNorm}) that the gradient at $\beta^*$ is small, and then show (Lemma~\ref{LemmaHessLower}) that the loss function is strongly convex in a sufficiently large ball around $\beta^*$. Combining these two observations, we conclude that $\beta^*$ is close to the empirical minimizer, $\widehat{\beta}_{H, \gamma}$. Our rates are substantially tighter than those of Sun et al.~\cite{SunZF20} due to the improved guarantees of Lemmas~\ref{LemmaGradNorm} and~\ref{LemmaHessLower} in comparison to the results in that paper.

We now state and prove the two supporting lemmas:

\begin{lemma}
\label{LemmaGradNorm}
Consider the setting of Theorem~\ref{ThmDetHuberReg}. With probability at least $1 - \tau$, the gradient of the loss function satisfies
\begin{align*}
\|\nabla\cL_\gamma(\beta^*)\|_2 \lesssim \gamma  \sqrt{U}\left( \sqrt{\frac{p}{n}} + \sqrt{\frac{\log 1 / \tau}{n}} \right).
\end{align*}
\end{lemma}

\begin{proof}
We first note that the gradient at $\beta^*$ has a simple structure:
\begin{align*}
\nabla\cL_\gamma(\beta^*) &=  -\frac{1}{n}\sum_{i=1}^n \psi_\gamma(y_i - x_i^T \beta^*)x_i = -\frac{1}{n}\sum_{i=1}^n \psi_ \gamma(z_i)x_i .
\end{align*}
For brevity, we define $W := \nabla\cL_\gamma(\beta^*)$ and $W_i = \psi_ \gamma(z_i)$. Note that since the $z_i$'s are symmetric, the $W_i$'s are i.i.d.\ bounded random variables and $\E(W) = 0$.

We will now show that $W$ has sub-Gaussian concentration around 0.
Let $v$ be any unit vector.
Since the $W_i$'s are bounded by $\gamma$, the sub-Gaussian norm of $v^TZ$ can be bounded using Proposition 2.6.1 of Vershynin~\cite{Ver18}:
\begin{align*}
\|v^TW\|_{\psi_2} \lesssim \frac{1}{n} \sqrt{\littlesum_{i \in [n]} \gamma^2 (v^Tx_i)^2  } \leq \gamma \sqrt{\frac{U}{n}},
\end{align*}
where the last step uses weak stability. Therefore, $W$ is an $O\left(\gamma \sqrt{\frac{U}{n}}\right)$-sub-Gaussian random variable, so again using the results of Vershynin~\cite{Ver18}, we have
\begin{align*}
\|W\|_2 = \|W - \E W\|_2 &\lesssim \gamma \sqrt{\frac{U}{n}}\left(\sqrt{p} + \sqrt{\log \frac{1}{\tau}}\right),
\end{align*}
with probability at least $1-\tau$.
\end{proof}

\begin{lemma} \label{LemmaHessLower}
Consider the setting in Theorem~\ref{ThmDetHuberReg}. Let $r$, $U$, $\tau$, and $\gamma$ be such that
\begin{equation*}
C_2 \left(\frac{r \sqrt{U}}{\gamma} + \P \left(|z_i| \ge \frac{\gamma}{2}\right)  +   \frac{\log(1/\tau)}{n} \right)  \leq \epsilon,
\end{equation*}
for a constant $C_2 > 0$.
Then with probability at least $1 - \tau$, the loss function $\cL_\gamma(\beta)$ is $L$-strongly convex in the ball $\{\beta: \|\beta - \beta^*\|_2 \leq r\}$.
\end{lemma}

\begin{proof}
First note that $\cL_\gamma(\beta)$ is a convex function. The Hessian of $\cL_\gamma$ is not defined due to the fact that the Huber loss is not twice differentiable at $\gamma$. However, if we define the matrix
\begin{align*}
H_n(\beta) := \frac{1}{n} \sum_{i=1}^n x_ix_i^T \1\left( |y_i - x_i^T \beta| < \gamma\right),
\end{align*}
it follows that the strong convexity parameter of $\cL(\beta)$ is at least $\lambda_{\min}(H_n)$ (see Lemma~\ref{PropStrongCvx}).

Let $W := \sup_{\beta: \|\beta- \beta^*\|_2 \leq r} \frac{1}{n} \sum_{i=1}^n \1\left(|y_i - x_i^T \beta| \geq \gamma\right)$ and define the event $\cE:= \{ W < \epsilon \}$. By the weak stability property, we are guaranteed that on the event $\cE$, we have $\lambda_{\min}(H_n(\beta)) \ge L$ for any $\beta$ such that $\|\beta- \beta^*\|_2 \leq r$.

In the remainder of the proof, we will show that the event $\cE$ holds with high probability.
We first note that $W$ can be bounded from above, as follows:
\begin{align*}
W &=   \sup_{\beta: \|\beta- \beta^*\|\leq r} \frac{1}{n}\sum_{i=1}^n  \1\left(|y_i - x_i^T \beta| \geq \gamma\right) \\
&\leq \sup_{\beta: \|\beta- \beta^*\|\leq r}\frac{1}{n} \sum_{i=1}^n \1\left(|x_i^T (\beta - \beta^*)| \geq \frac{\gamma}{2}\right) + \frac{1}{n}\sum_{i=1}^n \1\left(| z_i| \geq \frac{\gamma}{2} \right). 
\numberthis \label{EqnHessExp}
\end{align*}
We can deterministically bound the first term using weak stability. Using the fact that for $x \geq 0$ and $y > 0$, the inequality $\1(x \geq y ) \leq \frac{x}{y}$ holds, we obtain the following bound for all $\beta$ such that $\|\beta - \beta^*\|_2 \le r$:
\begin{align*}
 \frac{1}{n} \sum_{i=1}^n \1\left(|x_i^T (\beta - \beta^*)| \geq \frac{\gamma}{2}\right) 
&\leq  \frac{2}{\gamma}\frac{\sum_{i=1}^n |x_i^T (\beta - \beta^*)|}{n}    
\leq \frac{2}{\gamma}  \sqrt{\frac{1}{n} \sum_{i=1}^n |x_i^T (\beta - \beta^*)|^2}    \\
&\leq \frac{2}{\gamma} \sqrt{U \|\beta - \beta^*\|_2^2}  \leq  \frac{2r \sqrt{U}}{\gamma},
\end{align*}
where we also use weak stability and the Cauchy-Schwarz inequality. Altogether, we obtain
\begin{align}
\label{EqIndEventChebyshev}
W \leq \frac{2r \sqrt{U}}{\gamma} + \frac{1}{n}\sum_{i=1}^n \1\left(| z_i| \geq \frac{\gamma}{2} \right).
\end{align}
Now let $W' := \frac{1}{n} \sum_{i=1}^n \1\left(| z_i| \geq \frac{\gamma}{2}\right)$.
Note that
\begin{align*}
\E W' = \frac{1}{n}\sum_{i=1}^n \E  \1\left(| z_i| \geq \frac{\gamma}{2} \right) = \P \left(|z_i| \ge \frac{\gamma}{2}\right). 
\end{align*}
Note that $W'$ is an empirical mean of indicator random variables. Thus, applying a Chernoff bound (cf.\ Lemma~\ref{ThmChernoff}), we obtain
\begin{align*}
W' \lesssim  \E W' +\frac{\log(1 / \tau)}{n},
\end{align*}
with probability at least $1-\tau$.
Overall, we obtain the following bound on $W$: with probability at least $1 - \tau$, 
\begin{align}
W \lesssim \frac{ r \sqrt{U}}{\gamma} + \P \left(|z_1| \ge \frac{\gamma}{2}\right) + \frac{\log (1 / \tau)}{n}.
\label{EqnHessBddDiff}
\end{align}
Therefore, the event $\cE$ (and thus, the desired lower bound on $H_n$) holds with probability $1 - \tau$, as long as the right-hand side of inequality~\eqref{EqnHessBddDiff} is less than $\epsilon$.
\end{proof}

With the help of Lemmas~\ref{LemmaGradNorm} and~\ref{LemmaHessLower}, we are ready to prove the theorem. Throughout the remainder of the proof, let $\widehat{\beta} = \widehat{\beta}_{H,\gamma}$.

We first verify the conditions for Lemma~\ref{LemmaHessLower}. 
By assumption, we have $\frac{C_2 \log(1/\tau)}{n} \leq \frac{\epsilon}{3}$ and $C_2\P \left(|z_i| \ge \gamma/2\right) \leq \frac{\epsilon}{3}$.
Therefore, for all $r \leq \frac{\epsilon \gamma}{3C_2 \sqrt{U}} := r^*$, the condition of Lemma~\ref{LemmaHessLower} is satisfied, and the function $\cL_\gamma$ is $L$-strongly convex in the region $\{\beta: \|\beta- \beta^*\|_2 \leq r^*\}$.

For an $\eta \in (0,1]$, let $ \widehat{\beta}_{\eta}$ be defined as $ \widehat{\beta}_{\eta} := \beta^* + \eta (\widehat{\beta}- \beta^*)$, and let $\eta_* \in (0,1]$ be the largest $\eta$ such that  $\|\widehat{\beta}_{\eta} - \beta^*\|_2 \leq r^*$.
Using the convexity of $\cL_\gamma(\beta)$ with Lemma~\ref{LemConvexEta} and the Cauchy-Schwarz inequality, we have
\begin{align}
\langle \widehat{\beta}_{\eta^*} - \beta^*, \nabla\cL_\gamma(\widehat{\beta}_{\eta^*}) -  \nabla\cL_\gamma(\beta^*)  \rangle &\leq \eta_* \langle \widehat{\beta} - \beta^*,\nabla\cL_\gamma(\widehat{\beta}) -  \nabla\cL_\gamma(\beta^*) \rangle \nonumber\\
&\leq \eta_* \|\nabla \cL_\gamma(\beta^*)\|_2 \| \widehat{\beta} - \beta^* \|_2,\label{EqConvCaucSchw}
\end{align}
where we use the fact that $\nabla \cL_\gamma(\widehat{\beta}) = 0$. 
Using the $L$-strong convexity of $\cL_\gamma$ in the ball of radius $r^*$ (cf.\ Lemma~\ref{PropStrongCvx}) and inequality~\eqref{EqConvCaucSchw}, we obtain
\begin{align*}
\eta_* \|\nabla\cL_\gamma(\beta^*)\|_2 \| \widehat{\beta} - \beta^* \|_2 \geq \langle  \widehat{\beta}_{\eta^*} - \beta^*, \nabla\cL_\gamma(\widehat{\beta}_{\eta^*}) - \nabla \cL_\gamma(\beta^*) \rangle \geq L \|\widehat{\beta}_{\eta^*} - \beta^*\|_2^2.
\end{align*}
We now use Lemma~\ref{LemmaGradNorm} and the fact that  $\|\widehat{\beta}_{\eta^*} - \beta^*\|_2 = \eta^*  \| \widehat{\beta} - \beta^*\|_2 $ to obtain the following bound:
 \begin{align}
\label{EqUpperBdParamDis}
\|\widehat{\beta}_{\eta^*} - \beta^*\|_2 \leq \frac{1}{ L}  \|\nabla \cL_\gamma(\beta^*)\|_2 \leq C \frac{ \gamma \sqrt{U}}{  L }\left( \sqrt{\frac{p}{n}} +  \sqrt{\frac{\log (1 / \tau)}{n}} \right) := R_n.
\end{align}

Note that $\frac{R_n}{r^*} = \frac{3CC_2U}{\epsilon L}\left(\sqrt{\frac{p}{n}} +  \sqrt{\frac{\log (1 / \tau)}{n}}\right)$, so under the sample complexity assumption $n = \Omega\left(\left(p + \log\left(\frac{1}{\tau}\right)\right) \frac{U^2}{L^2 \epsilon^2}\right)$, we have $R_n \le r^*$, implying in particular that $\eta^* = 1$ and $\widehat{\beta} = \widehat{\beta}_{\eta^*}$ satisfies the stated error bound.

The statement about $L$-strong convexity follows from the triangle inequality, since for sufficiently large $n$, we have $\|\widehat{\beta} -\beta^*\|_2 \leq R_n \leq \frac{r^*}{2}$, so the function $\mathcal{L}_\gamma$ is $L$-strongly convex in a ball of radius $\frac{r^*}{2}$ around $\widehat{\beta}$.
\end{proof}

\subsection{Generalization to random design and asymmetric noise}
\label{SecHuberGeneral}

We now generalize the result of the previous section to the random design model with asymmetric noise. We proceed by reducing the case of asymmetric noise to symmetric noise: we will randomly subtract two points so that the additive noise in the new linear model has symmetric noise. Next, we will show that the iterative filtering algorithm from Diakonikolas et al.~\cite{DiaKan19,DiaKKLMS16-focs} (Theorem~\ref{ThmStability}) can be used to obtain a large subset of data points for which the covariates satisfy weak stability.
We will then use Theorem~\ref{ThmDetHuberReg} to prove the main result of this section.

\begin{algorithm}[h]  
  \caption{Huber Regression Asymmetric Noise} 
    \label{AlgHubGeneralCase}  
  \begin{algorithmic}[1]  
    \Statex  
    \Function{Huber\_Regression\_with\_Filtering}{$(x_i,y_i)_{i \in [2n]}, \gamma, \epsilon' $}  
    		\For{$i \gets 1$ to $n$}                    
        \State $(x_i',y_i')$ $\gets$ $\left(\frac{x_{i}-x_{n+i})}{\sqrt{2}}, \frac{y_i - y_{n+i}}{\sqrt{2}}\right)$
    \EndFor
            \State $S_1 \gets $ FilteredCovariates$((x'_i)_{i \in [n]},\epsilon')$    
        \State $\widehat{\beta} \gets $ HuberRegression$((x'_i,y'_i)_{i \in S_1},\gamma)$
    \State \Return $\widehat{\beta}$
    \EndFunction  
  \end{algorithmic}  
\end{algorithm}

\begin{theorem}
\label{ThmStocHuberReg}
Suppose we have $2n$ i.i.d.\ samples $\{(x_i,y_i)\}_{i=1}^{2n}$ from the following (random-design) model: $y_i = x_i^T\beta^* + z_i$, where the covariates satisfy Assumption~\ref{AsCov} and the noise distribution satisfies Assumption~\ref{AsNoise}.
Let $\tau$ be such that $\frac{\log(1/ \tau)}{n} = O(1)$. Suppose $\gamma$ is such that $\P \left(|z_1 - z_2| \ge \frac{\gamma}{\sqrt{2}}\right) \leq c^*$ for a small enough constant $c^* > 0$, and suppose $\epsilon'$ is equal to a sufficiently small constant. Then running Algorithm~\ref{AlgHubGeneralCase} with parameters $\gamma$ and $\epsilon'$ produces an estimator that, with probability at least $1 - 2\tau$, satisfies
\begin{align*}
\|  \widehat{\beta} -  \beta^*\|_2 \lesssim \gamma \left( \sqrt{\frac{p}{n}} +   \sqrt{\frac{\log(1 / \tau)}{n}} \right), \,\,\, \text{ as long as } n = \Omega(p\log p).
\end{align*}
Moreover, on the same event, the loss function is $\Omega(1)$-strongly convex in a radius of $\Omega(\gamma)$ around $\widehat{\beta}$.
\end{theorem}

\begin{proof}
We first note that by taking pairwise differences, we reduce our case to the symmetric noise setting analyzed in Section~\ref{SecSym}: 
Given $2n$ data points, Algorithm~\ref{AlgHubGeneralCase} creates a data set $\{(x'_i,y'_i)\}_{i=	1}^n$ satisfying the linear model $y'_i = (x'_i)^T \beta^* + z_i'$,
where $z'_i = \frac{z_i - z_{n+i}}{\sqrt{2}}$.
Note that the new covariates still satisfy $\E x'_i= 0$ and $\E x'_i(x'_i)^T = I$.
Importantly, the errors are now drawn from a symmetric distribution.

Let $S_1$ be the set returned by the filter algorithm with cardinality $\Omega(n)$, and define the event
\begin{align*}
\cE = \{S_1  \text{ satisfies weak stability with $\epsilon = \Omega(1)$, $L = \Omega(1)$, and  $U  = O(1)$}\}.
\end{align*}
We first give the proof of the theorem statement on the event $\cE$.
Since the noise is symmetric and independent of the covariates (thus also of $\cE$), we have $\P \left(|z'_i| \ge \gamma/2\right)  = O(\epsilon)$, so Theorem~\ref{ThmDetHuberReg} applies and gives the desired result. In the rest of the proof, we will show that $\cE$ holds with probability $1 - \exp(- \Omega(n)) \geq 1 - \tau$.

Recall by Lemma~\ref{LemStrongWeak} that if $S_1$ is $(\epsilon_1,\delta_1)$-stable, then it also satisfies weak stability with $\epsilon = \epsilon_1, L = (1-\epsilon_1) \left(1 - \frac{\delta_1^2}{\epsilon_1}\right)$, and $U = 1 + \frac{\delta_1^2}{\epsilon_1}$.
Therefore, it suffices to show that $S_1$ is $(\epsilon_1,\delta_1)$-stable such that $\epsilon_1 = \Omega(1)$ and (say) $\frac{\delta_1^2}{\epsilon_1} < 0.5$.
By Proposition~\ref{PropStabSimpleV2}, we know that 
if $\epsilon' < c_*$ and $n = \Omega\left(\frac{p \log p}{\epsilon'}\right)$,
then with probability at least $1 - O(\exp(- \Omega(n \epsilon ')))$, the set
$S_1$ is $(\epsilon_1, \delta_1)$-stable with $\frac{\delta_1^2}{\epsilon_1} < 0.2$ and  $\epsilon_1 = \Omega(\epsilon')$.
Therefore, choosing $\epsilon'$ to be a small enough constant, say $\frac{c^*}{2}$, we conclude that the event $\cE$ holds with probability $1 - O(\exp(- \Omega(n)))$.
This requires that $n = \Omega\left(\frac{p \log p}{\epsilon'}\right) = \Omega(p \log p)$, completing the proof.
\end{proof}

\begin{remark}
Similar to Remark~\ref{RemarkNoiseFirstMoment}, if the $k^{\text{th}}$ moment of the noise distribution is finite, we can set $\gamma = \Omega((\E |z_1 - z_2|^k)^{1/k})$, for any positive $k$.
\end{remark}

We now briefly discuss how to estimate an appropriate tuning parameter $\gamma$ from the data. A natural approach is to estimate the scale of the noise distribution based on residuals $y_i - x_i^T \widehat{\beta}_0$ calculated from an initial estimate $\widehat{\beta}_0$ of $\beta^*$. Indeed, the estimate $\widehat{\beta}_0$ can be quite rough, since only need to estimate the scale of the noise up to a constant factor. Based on these observations, consider the following procedure:
\begin{enumerate}
\item Split the sample into two equal parts.
\item Using the first part, compute $\widehat{\beta}_0$ via the LAD estimator (cf. Section~\ref{SecLAD} below).
\item Using the second part, compute the symmetrized data points $\{(x_i', y_i')\}_{i=1}^{\lfloor n/2 \rfloor}$ defined as in the first step of Algorithm~\ref{AlgHubGeneralCase}. Then compute the residuals $w_i' = y_i' - (x_i')^T \widehat{\beta}_0$.
\item Define $\widehat{\gamma}$ to be twice the $\left(1 - \frac{c^*}{4}\right)^{\text{th}}$ empirical quantile of the $|w_i'|$'s.
\end{enumerate}
Note that by our assumptions on the original data set, the sample-splitting step yields two sets of i.i.d.\ points. Thus, we may use Theorem~\ref{ThmLAD} below to show that $\|\widehat{\beta}_0 - \beta^*\|_2 = O(\kappa)$ if we assume that $\E|z_i| = \kappa < \infty$. Altogether, we can show that our procedure yields an estimator $\widehat{\gamma}$ such that $\P \left(|Z_1 - Z_2| \ge \widehat{\gamma}/2\right) \leq c^*$ (where $Z_1$ and $Z_2$ are fresh i.i.d.\ draws from the distribution of the $z_i$'s) and $\widehat{\gamma} = O(\E |z_i|)$, with high probability. Although other methods for choosing a rough initial estimator $\widehat{\beta}_0$ would also work, we suggest using the LAD estimator for initialization since it is tuning parameter-free. See Lemma~\ref{LemEstGamma} for more details.

\subsection{Adversarial corruption}
\label{SecHuberAdv}

We will now consider the case of adversarial corruption in both covariates and responses.
Let $S$ be the set of $n$ i.i.d.\ samples and let $T$ be an $\epsilon$-corrupted version of $S$ in the sense of Definition~\ref{DefContModel}.
One might expect Algorithm~\ref{AlgHubGeneralCase} to be robust to adversarial contamination, as Huber regression has been shown to be robust against corruption in responses~\cite{SasFuj20} and the filtering step can handle corruptions in covariates.
In this section, we will crucially use the strong stability condition, and not just weak stability, to obtain tighter control on deviations.
In fact, the following result shows that Huber regression also achieves near-optimal statistical guarantees in the adversarial setting with a slightly different choice of parameters.

\begin{theorem}
\label{ThmAdvHuberReg}
Let $S = \{(x_i, y_i)\}_{i=1}^{2n}$ be a set of i.i.d.\ samples drawn according to the same distributional assumptions as in Theorem~\ref{ThmStocHuberReg}.
Further suppose that the covariates satisfy $(k,2)$-hypercontractivity with parameter $\sigma_{x,k} = O(1)$, for some $k \geq 4$.
 Let $T$ be an $\epsilon$-corrupted version of $S$. 
 Suppose $\gamma$ is such that $\P \left(|z_1 - z_2| \ge \frac{\gamma}{\sqrt{2}}\right) \leq c^*$ for a small enough constant $c^* > 0$.
Then running Algorithm~\ref{AlgHubGeneralCase} on the set $T$ with parameters $\epsilon' = \Theta\left(\epsilon + \frac{\log(1/\tau)}{n}\right)$
produces an estimator that,
with probability at least $1 - \tau$, satisfies
\begin{align*}
\|\widehat{\beta} - \beta^*\|_2 \lesssim \gamma \left( \sqrt{\frac{p\log p}{n}} +  \sqrt{\frac{\log(1/ \tau)}{n}}  + {\epsilon}^{1 - 1/k}\right),\,\, \text{ as long as } n = \Omega(p \log p + \log(1/\tau)),
\end{align*}
and $\epsilon$ is less than a sufficiently small constant. Moreover, on the same  event,  the loss function is $\Omega(1)$-strongly convex in a radius of $\Omega(\gamma)$ around $\widehat{\beta}$.
\end{theorem}

\begin{remark}
Since the adversarial contamination mechanism might create dependencies between data points, the analysis of a sample-splitting algorithm to estimate an appropriate parameter $\gamma$ from the data, as in the previous subsection, becomes more complicated. A covering argument akin to the one employed in the proof of Theorem~\ref{PropRobMeanMain} below could be used instead, albeit at the price of a slightly worse error rate.
Another approach would be to tune the Huber parameter using Lepski's method~\cite{Lep91, Bir01}, at the expense of a slightly worse error probability due to a union bound over a grid of parameter values.
As noted in Remark~\ref{RemarkNoiseFirstMoment}, if the $(k')^{\text{th}}$ moment of the noise distribution is finite and known, Markov's inequality implies that we can set $\gamma = \Omega((\E |z_1 - z_2|^{k'})^{1/k'})$, for any positive $k'$.
\end{remark}

\begin{remark}
In order to run Algorithm~\ref{AlgHubGeneralCase} with the theoretical choice of $\epsilon'$ in Theorem~\ref{ThmAdvHuberReg}, we must assume knowledge of the level of adversarial contamination. On the other hand, note that if $T$ is an $\epsilon_1$-corrupted version of $S$, then $T$ is also an $\epsilon_2$-corrupted version of $S$, for any $\epsilon_1 \le \epsilon_2$. Thus, knowledge of an upper bound on the level of adversarial contamination is sufficient. (The same remark applies to Theorems~\ref{ThmLTS} and~\ref{ThmLAD}, and Theorems~\ref{PropPost} and~\ref{PropRobMeanMain} below.)
\end{remark}

The proof of Theorem~\ref{ThmAdvHuberReg} is rather technical and is provided in Appendix~\ref{AppThmAdvHuberReg}.
Briefly, our proof strategy is similar to the proof of Theorem~\ref{ThmDetHuberReg}:
Although the covariates and noise are not necessarily independent on the filtered set, we can establish modified versions of the structural Lemmas~\ref{LemmaGradNorm} and~\ref{LemmaHessLower}. In particular, we crucially use the stability property of the filtered set, which is stronger than the assumption of weak stability. 

\begin{remark}
We also note that Algorithm~\ref{AlgHubGeneralCase} has another favorable property when only the covariates are corrupted: Suppose $\{x_i\}_{i=1}^n$ and $\{z_i\}_{i=1}^n$ are generated from distributions satisfying Assumptions~\ref{AsCov} and~\ref{AsNoise}, respectively. Instead of observing $(X,X \beta^* + z)$, the statistician observes $(\tilde{X}, \tilde{y})$, where $\tilde{y} = \tilde{X} \beta^* + z $, and $\tilde{X}$ matches $X$ in all but $\epsilon n$ rows and is independent of $z$.
Then as long as $\epsilon$ is smaller than a fixed constant, the error guarantee of Theorem~\ref{ThmAdvHuberReg} would be of the form $O\left(\sqrt{\frac{p}{n}} + \sqrt{\frac{\log(1/\tau)}{n}}\right)$ and is independent of $\epsilon$.
Since $\tilde{X}$ and $\tilde{y}$ still follow a linear relationship and independence is maintained between the errors and covariates, the setting is essentially reduced to that of Theorem~\ref{ThmDetHuberReg}.
\end{remark}

\begin{remark}Finally, we mention a slightly stronger guarantee for Algorithm~\ref{AlgHubGeneralCase} for Gaussian covariates, i.e., $X \sim \cN(0,I)$. As can be seen in Appendix~\ref{AppLemHubAdv} in the proof of Theorem~\ref{ThmAdvHuberReg}, we could instead obtain an error bound of the form $O\left(\sqrt{\frac{p}{n}} + \sqrt{\frac{\log(1/\tau)}{n}} + \epsilon\sqrt{\log(1/\epsilon)}\right)$. This is because a set of $n$ i.i.d.\ samples from $\cN(0,I)$ is $(\epsilon,\delta)$-stable with probability $1- \tau$, where $\delta \lesssim \sqrt{\frac{p}{n}} + \sqrt{\frac{\log(1/\tau)}{n}} + \epsilon \sqrt{\log(1/\epsilon)}$~\cite{DiaKKLMS16-focs,Li18,DiaKKLMSl17}. We note that the subGaussian distributions with identity covariance and subgaussian norm $O(1)$ also achieve this rate.
\end{remark}

\subsection{Generalization to unknown covariance}
\label{SecUnknownCov}

We now discuss the case where the covariates have an unknown but bounded covariance matrix. We replace Assumption~\ref{AsCov} with the following assumption:

\begin{assumption}
\label{AsCov2}
The covariates satisfy  $\E x_i = 0$ and $\kappa_l I \preceq \E x_ix_i^T \preceq \kappa_uI$ for some $\kappa_l \in (0,1)$ and $\kappa_u \geq 1$. (For simplicity, we will assume that $\kappa_l = 1/2$ and $\kappa_u = 2$ in our arguments, but similar results hold as long as $\kappa_u = \Theta(\kappa_l)$.) Moreover, the covariates satisfy $(4,2)$-hypercontractivity with parameter
$\sigma_{x,4} \leq C$, for a known constant $C$.
\end{assumption}
We are able to generalize our result from Theorem~\ref{ThmAdvHuberReg} to the setting under Assumption~\ref{AsCov2}.
\begin{theorem}
\label{ThmAdvHubRegUnknownCov}
Suppose we have $2n$ i.i.d.\ samples $\{(x_i,y_i)\}_{i=1}^{2n}$ from the following (random-design) model: $y_i = x_i^T\beta^* + z_i$, where the covariates satisfy Assumption~\ref{AsCov2} and the noise distribution satisfies Assumption~\ref{AsNoise}.
Let $\tau$ be such that $\frac{\log(1/ \tau)}{n} = O(1)$. Suppose $\gamma$ is such that $\P \left(|z_1 - z_2| \ge \frac{\gamma}{\sqrt{2}}\right) \leq c^*$ for a small enough constant $c^* > 0$, and suppose $\epsilon'$ is equal to a sufficiently small constant.
 Let $T$ be an $\epsilon$-corrupted version of $S$. 
Then running Algorithm~\ref{AlgHubGeneralCase} on the set $T$ with parameters $\epsilon' = \Theta\left(\epsilon+ \frac{\log(1/\tau)}{n}\right)$  and $\gamma = \Omega( \sigma)$ produces an estimator that,
with probability at least $1 - \tau$, satisfies
\begin{align*}
\|\widehat{\beta} - \beta^*\|_2 \lesssim \gamma \left( \sqrt{\frac{p\log p}{n}} +  \sqrt{\frac{\log(1/ \tau)}{n}}  +  \sqrt{\epsilon}\right),\,\, \text{ as long as } n = \Omega(p \log p + \log(1/\tau)),
\end{align*}
and $\epsilon$ is less than a sufficiently small constant. Moreover, on the same  event,  the loss function is $\Omega(1)$-strongly convex in a radius of $\Omega(\gamma)$ around $\widehat{\beta}$.
\end{theorem}

The proof of Theorem~\ref{ThmAdvHubRegUnknownCov} is given in Appendix~\ref{AppHuberUnkCov}, and follows the same strategy as Theorem~\ref{ThmAdvHuberReg}, by noting that Huber regression primarily relies on $(\epsilon,L,U)$-weak stability, where $\epsilon= \Omega(1), L = \Omega(1)$, and $ U = O(1)$. 
The first two conditions are satisfied due to the small ball property, and the guarantee of the filter algorithm in the unknown covariance case is strong enough to ensure the third condition~\cite{DiaKP20}.
However, these algorithms do not adapt to higher moments of the data in the unknown covariance setting. This drawback is reflected in the worse dependence on $\epsilon$, i.e.,  $O(\sqrt{\epsilon})$ instead of $O(\epsilon^{3/4})$ under $(4,2)$-hypercontractivity. Note that the SQ lower bound of Diakonikolas et al.~\cite{DiaKS19} suggests that this $O(\sqrt{\epsilon})$ dependence is essentially optimal when $n = o(p^2)$ even when the covariates are Gaussian (with an unknown covariance).

\begin{remark}
In the absence of adversarial contamination, we can follow the same strategy as in Theorem~\ref{ThmStocHuberReg}: Under Assumption~\ref{AsCov2}, we can run the filter algorithm with $\epsilon'$ equal to a small enough constant (independent of $\tau$) to obtain a sub-Gaussian tail in the error guarantee.
\end{remark}

\subsection{Optimization}
\label{SubSecHubRunTime}

As noted above, the Huber objective function $\cL_\gamma(\beta)$ is convex in $\beta$, so optimization should in principle be easy. Taking a closer look, we see that as established in Theorems~\ref{ThmStocHuberReg} and \ref{ThmAdvHuberReg}, the loss function is \emph{strongly} convex in a ball of sufficiently large enough radius $\Omega(\gamma)$ around $\widehat{\beta}$. 
Therefore, running gradient descent yields linear convergence if the initialization is inside that ball~\cite{Bub15}.
Considering the case when we set the Huber parameter to be $ \gamma = \Theta(\sigma)$, our theory shows that we can guarantee such an initialization using the LAD estimator (cf.\ Theorem~\ref{ThmLAD}) or LTS estimator (cf.\ Theorem~\ref{ThmLTS}).

If we do not want to use a different robust regression estimator for a warm start, we can always directly apply the ellipsoid algorithm to the Huber loss. However, running the ellipsoid algorithm might be undesirable, as its running time, although polynomial, is practically slow~\cite{Bub15}.

\section{Least trimmed squares estimator}
\label{SecLTS}

In this section, we study the least trimmed squares (LTS) estimator~\cite{Rou84}:
\begin{align}
\label{EqnLTS}
\widehat{\beta}_{LS,m} = \argmin_{\beta} \min_{S \subseteq n: |S| = n - m} \sum_{i \in S} (y_i-x_i^T \beta)^2,
  \end{align}
where $m$ is the trimming parameter. We will establish conditions under which $\|\widehat{\beta}_{LS,m}-\beta^*\|_2$ is small, with very high probability.

\begin{algorithm}
  \caption{Alternating minimization algorithm} 
    \label{AlgLTSOrig}  
  \begin{algorithmic}[1]  
    \Statex  
    \Function{Alternating\_Minimization}{$(x_i,y_i)_{i \in [n]}, m, J $}  
            \State $b^{0} \gets 0$   
	\For{$j \gets 1$ to $J$}                    
        \State $b^{j} \gets \HT_m(P_X b^{j-1} + (I - P_X) y )$

    \EndFor
    \State $\widehat{\beta}_J \gets (X^TX)^{-1} X^T(y - b^{j})$
    \State \Return $\widehat{\beta}_J$
    \EndFunction  
  \end{algorithmic}  
\end{algorithm}

Unlike the Huber regression estimator, a significant drawback of the LTS estimator is that the objective function~\eqref{EqnLTS} is nonconvex. Nonetheless, various methods have been developed to efficiently obtain a local optimum of the LTS objective function, which have been shown to perform well empirically~\cite{RouVan06}. In recent work, Bhatia et al.~~\cite{BhaJK15,BhaJKK17} proved that under sufficiently nice assumptions on the covariates, the alternating minimization algorithm (Algorithm~\ref{AlgLTSOrig}) succeeds in finding a good candidate solution.
 Here, $P_X = X(X^TX)^{-1}X^T$ denotes the hat matrix, and the function $\HT_m$ is defined as follows:
\begin{definition}
For any $v \in \R^n$ and $m \in [n]$, let $S_{m,v} \subseteq [n]$ be the set of cardinality of $m$ such that for any $i \in S_{m,v}$ and $j \in [n] \setminus S_{m,v}$, we have $|v_{i}| \geq |v_{j}|$.
To ensure uniqueness, we choose the smaller indices if ties occur.
The \emph{$m$-hard thresholding operator} is the function $\HT_m: \R^n \to \R^n$ defined as follows: For any $v \in \R^n$, we have
\begin{align*}
  (\HT_m(v))_i = \begin{cases}v_i, & \text{ if } i \in S_{m,v},\\
  0, & \text{ otherwise. }\end{cases}
  \end{align*}
\end{definition}
In other words, the set $S_{m,v}$ identifies the indices of the $m$ coordinates of $v$ that are largest in magnitude, and the $\HT_m$ function returns a vector that preserves these top $m$ components and sets the rest to zero.
Note that Algorithm~\ref{AlgLTSOrig} is derived by recasting the optimization problem~\eqref{EqnLTS} as
\begin{equation*}
\min_{\beta \in \real^p, \|b\|_0 \le m} \|X \beta - (y-b)\|_2^2
\end{equation*}
and alternately minimizing over $\beta$ and $b$, where we explicitly solve for $\beta$ on each iteration (see Bhatia et al.~\cite{BhaJKK17} for more details).

We now state the following deterministic result, which is implicit in Bhatia et al.~\cite{BhaJKK17}. For completeness, we provide a proof in Appendix~\ref{AppLTSBhatia}. Recall the definitions of the SSC and SSS properties from Definition~\ref{DefSS}.

\begin{lemma}(Adapted from Lemma 5 of Bhatia et al.~\cite{BhaJKK17})
\label{LemAltMin}
Suppose $y = X \beta^* + z$, where the SSC and SSS parameters of the $x_i$'s, denoted by $\{\lambda_{k}\}$ and $\{\Lambda_{k}\}$, respectively, satisfy $\frac{\Lambda_{2m}}{\lambda_n} < \frac{1}{4}$ and $\Lambda_n = O( \lambda_n )$.
Suppose $z = w + b^*$, for some vector $w \in \R^n$ and an $m$-sparse vector $b^* \in \R^n$, and let $G$ and $H$ be numbers such that $G \geq  \sup_{S': |S'| \leq 2m} \sqrt{\sum_{i \in S'} w_i^2}$ and $H \geq \|\sum_{i=1}^n x_iw_i\|_2$. Then Algorithm~\ref{AlgLTSOrig}, after $J \succsim \log_2\left(\frac{\|b^*\|_2}{2G + 2H /\sqrt{\lambda_n}}\right)$ iterations, outputs an estimator $\widehat{\beta}$ such that 
\begin{align*}
\| \widehat{\beta} - \beta^* \|_2  \lesssim \frac{ G \sqrt{\Lambda_n} +  H}{\lambda_n}.
\end{align*}
\end{lemma}

\begin{remark}
\label{RemAltMin}
The proof of Lemma~\ref{LemAltMin} actually implies that for any error level $e \gtrsim \frac{G\sqrt{\Lambda_n} + H}{\lambda_n}$, Algorithm~\ref{AlgLTSOrig} is guaranteed to output an estimator satisfying the error bound $\|\widehat{\beta} - \beta^*\|_2 \le e$ after $J \gtrsim \log_2\left(\frac{\|b^*\|_2}{e}\right)$ iterations. This form of the result is helpful in settings such as Theorem~\ref{ThmLTS} below, where we can obtain data-driven upper bounds on $G$ and $H$, and consequently also on the term $\frac{G\sqrt{\Lambda_n} + H}{\lambda_n}$, which hold with high probability. Together with a data-driven upper bound on $\|b^*\|_2$, this provides a calculable lower bound on the number of iterations required for Algorithm~\ref{AlgLTSOrig} to succeed in outputting an estimator with small error.
\end{remark}

\begin{algorithm}
  \caption{Alternating minimization algorithm} 
    \label{AlgLTSFiltering}  
  \begin{algorithmic}[1]  
    \Statex  
    \Function{Alternating\_Minimization\_with\_Filtering}{$(x_i',y_i')_{i \in [n]}, \epsilon' , m, J $}  
            \State $T_1 \gets $FilteredCovariates$((x_i)_{i \in [n]},\epsilon')$  
    \State $\widehat{\beta}_J \gets \textsc{Alternating\_Minimization}((x_i', y_i')_{i \in T_1}, m, J)$
    \State \Return $\widehat{\beta}_J$
    \EndFunction  
  \end{algorithmic}  
\end{algorithm}

Note that the statement of Lemma~\ref{LemAltMin} is deterministic: In Bhatia et al.~\cite{BhaJKK17}, it was shown that when the covariates are i.i.d.\ Gaussian, the SSC and SSS conditions hold with high probability. Our main result in this section shows that these conditions hold with high probability for possibly heavy-tailed, adversarially contaminated covariates after applying our filtering step.

\begin{theorem}
\label{ThmLTS}
Let $S = \{(x_i, y_i)\}_{i=1}^{n}$ be a set of i.i.d.\ samples drawn according to the same distributional assumptions as in Theorem~\ref{ThmAdvHuberReg}. Let $T = \{(x_i', y_i')\}_{i=1}^n$ be an $\epsilon$-corrupted version of $S$, where $\epsilon$ is less than a sufficiently small constant.
Further suppose that the errors satisfy $(k',2)$-hypercontractivity with parameter  $\sigma_{z,k'} = O(1)$, for some $k' \geq 2$.
Let $\tau$ be such that $\frac{\log(1/ \tau)}{n} = O(1)$.
With probability at least $1 - O(\tau)$, running Algorithm~\ref{AlgLTSFiltering} on the set $T$ with parameters $m = \Theta\left(p \log p + \epsilon n + \log\left(\frac{1}{\tau}\right)\right)$ and $\epsilon' = \Theta\left(\frac{m}{n}\right)$
yields an estimator $\widehat{\beta}$ satisfying
\begin{align*}
\|\widehat{\beta} - \beta^*\|_2 \lesssim \sigma \left( \sigma_{z,k'} \left(\frac{p \log p}{n} + \epsilon +  \frac{\log(1/ \tau)}{n}\right)^{1/2 - 1/k'}\right),
 \,\, \text{ as long as } n = \Omega(p \log p),
\end{align*}
provided $J \gtrsim \log_2\left(\frac{ \|y'\|_2 + \|X'\|_2 \|\beta^*\|_2}{\alpha}\right)$, where $\alpha$ is defined to be the error bound given above.

If we further suppose that the errors satisfy $(4,2)$-hypercontractivity with $\sigma_{z,4} = O(1)$, then $J \gtrsim \log_2\left( \frac{\|y'\|_2(1 + \|X'\|_2)}{\alpha}\right)$ iterations suffice. %
\end{theorem}

\begin{remark}
Note that the error guarantee of the LTS estimator in Theorem~\ref{ThmLTS} is weaker than that of the Huber regression estimator in Theorem~\ref{ThmAdvHuberReg}.
It is not clear whether the suboptimality of the LTS error bound is intrinsic to the LTS estimator or an artifact of our analysis; we leave this question for future work. In the case of sub-Gaussian noise, it can be shown that the guarantee of Theorem~\ref{ThmLTS} matches the guarantee of Bhatia et al.~\cite{BhaJK15,BhaJKK17} (up to log factors) who assume, in addition, that the covariates are sub-Gaussian. 
\end{remark}

The complete proof of Theorem~\ref{ThmLTS} is provided in Appendix~\ref{AppLTSProb}, but we provide a proof sketch below.

\begin{proof}
To simplify the argument, assume for this proof sketch that no adversarial contamination is present in the data.
Recall that $T_1$ is the output of the filter algorithm with input $T$ and $\epsilon' = \Theta(m/n)$.
Let $n_1 = |T_1|$.
Note that if the covariates in $T_1$ satisfy $(\epsilon,\delta)$-stability, then
\begin{equation*}
n_1\left(1 - \frac{\delta^2}{\epsilon}\right) \leq \lambda_{n_1} \leq \Lambda_{n_1} \leq n_1\left(1  + \frac{\delta^2}{\epsilon}\right).
\end{equation*}
Furthermore, by Proposition~\ref{PropStabSqError}, we have $\Lambda_{\lfloor\epsilon n_1\rfloor} \leq \frac{3n_1\delta^2}{\epsilon}$.
Suppose $T_1$ is $(\epsilon_1,\delta_1)$-stable such that $|T_1|\epsilon_1 = \epsilon_1 n_1 \geq 2m$ and $n_1 \geq \frac{n}{2}$.
Thus, if $\frac{\delta_1^2}{\epsilon_1^2}$ is less than (say) $0.05$,
the condition $\frac{\Lambda_{2m}}{\lambda_n} < \frac{1}{4}$ of Lemma~\ref{LemAltMin} holds and the error bound is $O\left(\frac{H}{n} +  \frac{G}{\sqrt{n}}\right)$. 
Proposition~\ref{PropStabSimpleV2} shows that this holds if   $\frac{m}{n}$ is small enough.
We will now sketch how to bound the quantities $G$ and $H$.

To bound $G$, let $F$ be the cdf of the distribution of $|z_i|$, and let $F^{-1}$ be its inverse.
A Chernoff bound implies that with probability at least $1 - \exp(- \Omega(m))$, we have
\begin{align*}
\left|\left\{i : |z_i| >  F^{-1}\left(1 - \frac{m}{8n}\right)  \right\}\right| \leq \frac{m}{4}.
\end{align*}
The moment assumption on $z_i$ and Markov's inequality directly imply that $F^{-1}\left(\frac{m}{8n}\right) \lesssim \sigma_{z,k'}\sigma \left(\frac{m}{n}\right)^{-1/k'}$.
We will define $w_i$ to be zero if the corresponding value of $z_i$ does not satisfy this condition. Since $G$ is the maximum $\ell_2$-norm of any subvector of $w$ with $2m$ components, we have
$$G = O\left( \sigma \sigma_{z,k'}\sqrt{m} \left(\frac{m}{n}\right)^{-1/k'}\right).$$
Consequently, its contribution to the error is $ O\left(\frac{G}{\sqrt{n}}\right) = O \left(\sigma \sigma_{z,k'}\left(\frac{m}{n}\right)^{1/2 - 1/k'}\right)$.

Next, we bound $H$. Consider the random variable $ x_iz_i$, which has mean zero and covariance $\sigma^2 I$. By Theorem~\ref{ThmStabHighProb}, we know that with probability $1 - \exp(- \Omega(m))$, there exists a set $S'$ such that $|S'| \geq n - \frac{m}{2}$ and $S'$ is $ \left(\frac{Cm}{n}, \delta\right)$-stable with respect to $\mu$ and $\sigma^2$.
Finally, we will define $w_i$ to be zero if $i \notin S'$, as well.
Using the stability of $S'$, we can show that $H = O(n\sigma \delta)$.
Therefore, the overall bound is of the form $O\left( \sigma \sigma_{z,k} \epsilon^{1/2 - 1/k} + \sigma\delta\right)$, where $\sigma\delta$ is smaller than the first term.

The last step is to derive a high-probability upper bound on $\|b^*\|_2$. An application of the triangle inequality gives
\begin{equation*}
\|b^*\|_2 \le \|y'\|_2 + \|X'\|_2 \|\beta^*\|_2,
\end{equation*}
where $T = (X', y')$ is the corrupted data set. Finally, we show how to obtain a high-probability upper bound on $\|\beta^*\|_2$ which depends on known quantities, under the additional assumption that the $z_i$'s satisfy $(4,2)$-hypercontractivity. We can derive the inequality
\begin{equation*}
\E|y_i| = \E|x_i^T \beta^* + z_i| \ge \max \{\E | x_i^T \beta^* |, \E |z_i|\} \ge \max\left\{\frac{ \|\beta^*\|_2^2}{\sigma_{x,4}^4}, \frac{\sigma^2}{\sigma_{z,4}^4}\right\}.
\end{equation*}
Then the Paley-Zygmund inequality, together with a Chernoff bound, allow us to show that appropriately chosen quantiles of the $y_i$'s (and consequently also the corrupted responses) are larger than a multiple of $\|\beta^*\|_2$, with high probability.
\end{proof}

\begin{remark}
\label{RemLTSinit}
The two statements in Theorem~\ref{ThmLTS} differ in the number of iterations we require to guarantee that the output of the alternating minimization algorithm will have small $\ell_2$-error---in order to obtain a data-driven upper bound on $\|\beta^*\|_2$, we impose additional hypercontractivity assumptions on the noise distribution. As in the case of the Huber estimator (cf.\ Section~\ref{SubSecHubRunTime}), one might choose to use the LAD estimator to warm-start the algorithm and save on computation.
Theorem~\ref{ThmLAD} below guarantees that the LAD estimator satisfies $\|\widehat{\beta}_{\text{LAD}} - \beta^*\|_2 = O(\kappa)$ when $\E|z_i| = \kappa$;
the runtime of Algorithm~\ref{AlgLTSFiltering} on the shifted data $(X,y-X^T \widehat{\beta}_{\text{LAD}})$ would then scale with $\|\widehat{\beta}_{\text{LAD}}- \beta^* \|_2 = O(\kappa)$ rather than $\|\beta^*\|_2$.

As shown in the proof of Lemma~\ref{LemAltMin}, we can alternatively run Algorithm~\ref{AlgLTSFiltering} until $\|b^{j} - b^{j-1}\|_2 = O(\alpha\sqrt{n})$, where $\alpha$ is the error bound in Theorem~\ref{ThmLTS}, to obtain a data-dependent stopping criterion. Indeed, by inequality~\eqref{EqConvOfB} below, we have $\|b^{j+1} - b^*\| \leq e_0 + \frac{1}{2}\|b^j - b^*\|_2$, so by the triangle inequality,
\begin{equation*}
\|b^{j} - b^{j+1}\|_2 \geq \|b^{j} - b^*\|_2 - \|b^{j+1} - b^*\|_2 \geq \frac{1}{2}\|b^{j} - b^*\|_2 - e_0.
\end{equation*}
Thus, if the difference between successive iterates is sufficiently small, the error must be small, as well.

\end{remark}

Finally, we emphasize that although the LTS objective function is nonconvex~\eqref{EqnLTS}, our theoretical guarantees are for the output of a particular iterative algorithm which can be performed efficiently. Importantly, the validity of our theoretical analysis does not require us to assume that the alternating minimization algorithm converges to a global optimum of the LTS objective.

\section{Least absolute deviation}
\label{SecLAD}

In this section, we study the least absolute deviation (LAD) estimator:
\begin{align*}
\widehat{\beta}_{LAD} = \argmin_{\beta}  \sum_{i=1}^n |y_i-x_i^T \beta|.
  \end{align*}
Note that the LAD estimator is parameter-free. Although the error bounds we derive for the LAD estimator have suboptimal error rates compared to the other estimators, the LAD estimator is useful for initialization for tuning or optimizing the Huber estimator (cf.\ Sections~\ref{SecHuberGeneral} and~\ref{SubSecHubRunTime}), or initializing the alternating minimization algorithm for the LTS estimator (cf.\ Remark~\ref{RemLTSinit}).

Our main result relies on the following lemma from Karmalkar and Price~\cite{KarPri19}, who showed that if the covariates satisfy $(\epsilon,m,M,\ell_1)$-stability, then the LAD estimator is robust to corruption in responses.
We provide a proof for completeness:
\begin{lemma}(Karmalkar and Price~\cite{KarPri19})
\label{LemKarPriLAD}
 Suppose the covariates satisfy $(m,M, \epsilon, \ell_1)$-stability such that $M > m$. Then
\begin{align*}
\|\widehat{\beta}_{\text{LAD}} - \beta^*\|_2 = O\left( \frac{ \sum_{i=1}^{(1 - \epsilon) n} |z|_{(i)} }{n (M - m)} \right).
\end{align*}
\end{lemma}

\begin{proof}
We denote $\widehat{\beta} = \widehat{\beta}_{\text{LAD}}$ for brevity.
Let $S$ be the set of $(1-\epsilon)n$ indices with the smallest magnitudes of additive errors.
We have the following:
\begin{align*}
0 &\geq  \sum_{i \in S} |y_i - x_i^T \widehat{\beta}| 
  - \sum_{i \in S} |y_i - x_i^T \beta^*|  + \sum_{i \in S^c} |y_i - x_i^T \widehat{\beta}| - \sum_{i \in S^c} |y_i-x_i^T \beta^*|     \\
&\geq \sum_{i \in S} |x_i^T ( \widehat{\beta} - \beta^*)| - 2\sum_{i \in S} |y_i - x_i^T \beta^*| - \sum_{i\in S^c}|x_i^T (\widehat{\beta} - \beta^*)|,\\
&\geq n M \|\widehat{\beta} - \beta^*\|_2 -2 \sum_{i\in S} |z_i| - nm \|\widehat{\beta} - \beta^*\|_2, 
\end{align*}
where the first inequality follows by the optimality of $\widehat{\beta}$,
the second inequality uses the triangle inequality, and the third inequality uses the property of $(\epsilon,m, M, \ell_1)$-stability.
Rearranging the inequality and using the fact that $\sum_{i \in S} |z_i| \le \sum_{i=1}^{(1-\epsilon)n} |z|_{(i)}$, we obtain the desired result.
\end{proof}

Our main result in this section is to show that under our setting, the filtered covariates satisfy the $\ell_1$-stability condition of Definition~\ref{DefL1Stable}, from which we may derive an error bound according to Lemma~\ref{LemKarPriLAD}.

\begin{algorithm}[h]  
  \caption{LAD with filtered covariates} 
    \label{AlgLAD_filter}  
  \begin{algorithmic}[1]  
    \Statex  
    \Function{LAD\_with\_Filtering}{$(x_i',y_i')_{i \in [n]}, \epsilon'$}  
            \State $T_1 \gets $FilteredCovariates$((x_i)_{i \in [n]},\epsilon')$     
        \State $\widehat{\beta}_{\text{LAD}} \gets $ LAD$((x_i', y_i')_{i \in T_1})$
    \State \Return $\widehat{\beta}_{\text{LAD}}$
    \EndFunction  
  \end{algorithmic}  
\end{algorithm}

\begin{theorem}
\label{ThmLAD}
Let $S = \{(x_i,y_i)\}_{i=1}^n$ be i.i.d.\ samples from the linear model $y_i = x_i^T \beta^* + z_i$, where the covariates satisfy Assumption~\ref{AsCov}
and the noise satisfies $\E|z_i| = \kappa$.
For an $\epsilon < c^*$, let $T$ be an $\epsilon$-corrupted version of $S$.
Let $\widehat{\beta}$ be the output of Algorithm~\ref{AlgLAD_filter} with input $T$ and $\epsilon'$, where $\epsilon'$ is a small enough constant.
Let $\tau$ be such that $\frac{\log(1/ \tau)}{n} = O(1)$.
Then with probability at least $1 - \tau$, we have
\begin{align*}
 \|\widehat{\beta} - \beta^*\|_2 = O( \kappa), \,\,\, \text{ as long as } n = \Omega(p \log p).
 \end{align*}
\end{theorem}

\begin{proof}
The following lemma shows that the filtered covariates satisfy $(m,M, \epsilon,\ell_1)$-stability:

\begin{lemma}
\label{LemLADL1Stab}
Let $S$ be the data set described in Theorem~\ref{ThmLAD}.
For an $\epsilon_1 < c_*$, let $T$ be an $\epsilon_1$-corrupted version of set $S$.
Let $T_1$ be the output of the filter algorithm on input $T$ and $\epsilon'$, where $\epsilon' = \Theta(1)$.
Then with probability at least $1 - O(\exp(- \Omega(n)))$, the set $T_1$ satisfies $(\epsilon_2,m,M, \ell_1)$-stability with $\epsilon_2 = \Theta(1)$, $m = \Theta(1) $, $M = \Theta(1)$, and $M \geq 2m$, and these parameters do not depend on $\epsilon_1$. Moreover, $|T_1| \geq \frac{n}{2}$.
\end{lemma}

\begin{proof}
We provide a sketch of the proof here; more details may be found in Appendix~\ref{AppLADL1Stab}. We show that the lower bound (on $M$) in Definition~\ref{DefL1Stable} is satisfied due to the small-ball property~\cite{Men15}, and that the filtering algorithm removes the ``outliers'' in the data set, leading to the upper bound (on $m$). 
The proof of the lower bound is given in Lemma~\ref{LemLowL1Stab}, which follows similar calculations from previous work~\cite{KM15,DiaKP20}.
These arguments show that if $n = \Omega(p \log p)$, the $\ell_1$-stability lower bound holds with $M \geq \frac{1}{2 \sigma_4^2}$.
For the upper bound, we use the fact that the filtered set $T_1$ is $(\epsilon,\delta)$-stable. Then Proposition~\ref{PropStabL1Error} implies that for $T' \subseteq T_1$ with $|T'| \leq \epsilon|T_1|$, and any unit vector $v$, we have  $\frac{1}{|T_1|}\sum_{i \in T'}  |x_i^Tv| \leq 2\delta$, so the stability upper bound holds with $m \leq 2 \delta$.
We choose the parameter values such that $ M \geq \frac{1}{2 \sigma_4^2} \geq  4\delta \geq 2 m  = \Omega(1)$.
\end{proof}

Lemma~\ref{LemLADL1Stab} states that, with probability at least $1 - O(\exp(- \Omega( n))$, the set $T_1$ obtained by running the filtering algorithm on $T$ satisfies $(\epsilon_2,m,M,\ell_1)$-stability, where $2m \le M = \Theta(m)$ and $\epsilon_2 = \Theta(1)$.
We assume that $\epsilon$ is small enough such that $\epsilon_2 > 4 \epsilon$. %
Applying Lemma~\ref{LemKarPriLAD}, we claim that the $\ell_2$-estimation error is bounded by a constant times $\sum_{i=1}^{n - \epsilon_2 n_1} |y'  - X' \beta^*|_{(i)}$, where we denote the corrupted data set by $T = \{(x_i',y_i')\}_{i=1}^n$ and $n_1 = |T_1| = (1-\epsilon')n$. Indeed, the bound in Lemma~\ref{LemKarPriLAD} involves a sum of the $(1-\epsilon_2)n_1$ smallest residuals in the filtered data set. Each of these terms appears in the set of residuals $\{|y_i' - x_i'^T \beta^*|\}_{i=1}^n$ for $T$, so the aforementioned sum is certainly upper-bounded by the sum of all but the $\epsilon_2 n_1$ largest residuals for $T$.
Furthermore, we have
\begin{equation*}
\sum_{i  =1}^{n - \epsilon_2 n_1} |y' - X' \beta^*|_{(i)} \le \sum_{i  =1}^{n - \epsilon_2 n/2} |y' - X' \beta^*|_{(i)} \le \sum_{i  =1}^{n - \epsilon_2 n/2 + \epsilon n} |y-X \beta^*|_{(i)} \le \sum_{i  =1}^{n - \epsilon_2 n/4} |y - X\beta^*|_{(i)},
\end{equation*}
where the first inequality uses the fact that $n_1 \geq \frac{n}{2}$, the second inequality uses the fact that $T$ differs from $S$ in at most $\epsilon n$ points, and the last inequality uses the fact that $\epsilon \leq \frac{\epsilon_2}{4}$. 
Applying Lemma~\ref{LemTrimmedSumL1}, we see that the final quantity is at most $ O\left(\frac{n \kappa}{\epsilon_2}\right) $, with probability at least $1 - O(\exp(- \Omega(n \epsilon_2)))$.
Since $\epsilon_2 = \Omega(1)$, this completes the proof.
\end{proof}

\begin{remark}
\label{RemLADGen}
Note that the guarantees of Theorem~\ref{ThmLAD} hold under very general conditions. Unlike our assumptions on the noise distribution elsewhere in the paper, our theorem does not require the noise distribution to have zero mean or be independent of the covariates; all we require is that the first moment $\E|z_i|$ is finite.
Furthermore, we can generalize this result to the case of an unknown but bounded covariance of the form $\frac{1}{2} I \preceq \E xx^T \preceq 2I$ (cf.\ Section~\ref{SecUnknownCov}), as well.
\end{remark}

\section{Postprocessing }
\label{SecPP}

We now outline a one-step estimator which, given an initial estimator $\widehat{\beta}_1 $ such that $\|\widehat{\beta}_1 - \beta^*\|_2 = O(\sigma)$, returns another estimator $\widehat{\beta}_2$ that has sub-Gaussian rates. In the analysis of this section, we will assume that Assumption~\ref{AsNoise} is satisfied and the noise variance $\E(z_i^2) = \sigma^2$ is finite.
As shown in Sections~\ref{SecLTS} and~\ref{SecLAD}, the LTS or LAD estimators will then satisfy the error bound of $O(\sigma)$ with high probability and can be used for $\widehat{\beta}_1$.
We note that a similar postprocessing construction has been leveraged in earlier works~\cite{BalDLS17,DiaKS19,PraSBR20}.

We first state a version of the result for a setting where the estimate $\widehat{\beta}_1$ does not depend on the data.
This can always be achieved by splitting the samples when either (i) there is no  contamination, or (ii) the contamination mechanism does not depend on the data, e.g., in Huber's contamination model.

We first recall the median-of-means preprocessing algorithm (see Lugosi and Mendelson~\cite{LugMen19-survey} for a recent survey): Given data points $\{x_1,\dots,x_n\}$ and a parameter $k \in [n]$, construct $\{z_1,\dots,z_k\}$, as follows: Randomly bucket $\{x_1,\dots,x_n\}$ into $k$ disjoint buckets of equal size (if $k$ does not divide $n$, then remove some samples), and let $\{z_1,\dots,z_k\}$ be the empirical means of the points in these buckets. The following result from Diakonikolas et al.~\cite{DiaKP20} shows that applying the iterative filtering algorithm to the $k$ data points obtained after running the median-of-means algorithm returns a sub-Gaussian estimate of the mean of the original sample:

\begin{theorem} (Diakonikolas et al.~\cite{DiaKP20})
\label{ThmStabSubGaussian}
Let $S$ be a set of $n$ i.i.d.\ samples from a distribution with mean $\mu$ and covariance $\Sigma$.
Let $T$ be an $\epsilon$-corrupted version of $S$.
For a probability $\tau$, let $\epsilon' = \Theta\left(\epsilon + \frac{\log(1 / \tau)}{n}\right)$, where $\epsilon'$ is less than a small constant.
Let $k = \lceil \epsilon' n\rceil$.
Let $T_k := \{z_1,\dots,z_k\}$ be the set obtained by median-of-means preprocessing on the set $T$.
Then running the filtering algorithm in Theorem~\ref{ThmStability} with inputs $T_k$ and $\epsilon' = \Theta(1)$ returns a set $T'$ such that, with probability at least $ 1 - \exp(- \Omega(k))$,	
\begin{equation*}	
\|\widehat{\mu}_{T'} - \mu\|_2  = O\left(\sqrt{\frac{\trace(\Sigma)}{n}} + \sqrt{\|\Sigma\|_2\epsilon} + \sqrt{\frac{\|\Sigma\|_2\log(1/ \tau)}{n}}\right),	
\end{equation*}	
where $\widehat{\mu}_{T'}$ is the empirical mean of the set $T'$.	
\end{theorem}

Using the result of Theorem~\ref{ThmStabSubGaussian}, we can derive the following theorem:

\begin{theorem}
\label{PropPost}
Let $S$ be a set of $n$ i.i.d.\ samples from the linear model $y_i = x_i^T \beta^* + z_i$, where the covariates satisfy Assumption~\ref{AsCov} and the noise distribution satisfies Assumption~\ref{AsNoise}. Suppose $\E(z_i^2) = \sigma^2$.
Let $\widehat{\beta}_1$ be any estimator which is independent of $S$, satisfying $\|\widehat{\beta}_1 - \beta^*\|_2 = O(\sigma)$.
Let $T$ be an $\epsilon$-corrupted version of $S$, where $T$ might depend on $\widehat{\beta}_1$.
Define the set $T_1 := \left\{\widehat{\beta}_1 + (y_i' - (x_i)'^T \widehat{\beta}_1)x_i': (x_i',y_i') \in T\right\}$.
Suppose $\epsilon' = \Theta\left( \epsilon + \frac{\log(1 / \tau)}{n}\right) = O(1)$.
Then given $\epsilon$, $T_1$, and $\tau$ as inputs, the mean algorithm in Theorem~\ref{ThmStabSubGaussian} returns an output $\widehat{\beta}$ satisfying
\begin{equation*}
\|\widehat{\beta} - \beta^*\|_2 \lesssim \sigma\left(\sqrt{\frac{p}{n}} + \sqrt{\epsilon} + \sqrt{\frac{\log(1 / \tau)}{n}}\right),
\end{equation*}
with probability at least $1-\tau$.
\end{theorem}

\begin{proof}
Throughout the proof, we will condition on the value of the initial estimator $\widehat{\beta}_1$.
Let $S_1 := \left\{\widehat{\beta}_1+ (y_i - x_i^T \widehat{\beta}_1)x_i: (x_i,y_i) \in S\right\}$. Since $\widehat{\beta}_1$ is independent of $S$ by assumption, the set $S_1$ consists of i.i.d.\ samples when we condition on $\widehat{\beta}_1$.
It is easy to see that $T_1$ is an $\epsilon$-corrupted version of $S_1$ and $\E\left[\widehat{\beta}_1+ (y_i - x_i^T \widehat{\beta}_1)x_i\right] = \beta^*$.
Thus, the desired result follows from Theorem~\ref{ThmStabSubGaussian} if we can show that the set $S_1$ satisfies the stated conditions.
For simplicity, set
\begin{equation*}
w_i := \widehat{\beta}_1+ (y_i - x_i^T \widehat{\beta}_1)x_i = \widehat{\beta}_1 + x_i^Tx_i(\beta^* - \widehat{\beta}_1) + x_iz_i.
\end{equation*}

We will work conditionally on $\widehat{\beta}_1$ in the remainder of the proof.
Since $\widehat{\beta}_1$ is independent of $S$, the $w_i$'s are then conditionally i.i.d.
Set $\Delta := \widehat{\beta}_1 - \beta^*$, so $\|\Delta\|_2 \leq \sigma$ by assumption, and
observe that $w_i - \beta^* = \Delta - x_i^T x_i \Delta + x_iz_i$. Therefore, for any unit vector $v$, we have 
\begin{align}
\label{EqnCovW}
v^T \Sigma_{w_i}v = \E (v^T(w_i - \beta^*))^2 &=  \E (v^T\Delta - (v^Tx_i)(\Delta^Tx_i) + v^T x_i z_i )^2 \notag \\
&\lesssim   (v^T\Delta)^2 + \E \left((v^Tx_i)^2(\Delta^Tx_i)^2\right) + \E \left((v^T x_i)^2 z_i^2\right) \notag \\
&\lesssim  \|\Delta\|_2^2 + \sqrt{\E (v^Tx_i)^4}\sqrt{\E (\Delta^Tx_i)^4} + \sigma^2 \notag \\
&\lesssim \|\Delta\|_2^2 + \sigma_{x,4}^4 \|\Delta\|_2^2 + \sigma^2 \notag \\
&\lesssim \sigma^2.
\end{align}
Therefore, $\trace(\Sigma_w) \lesssim \sigma^2 p $ and $\|\Sigma_w\|_2 \lesssim \sigma^2$.
This completes the proof.
(Observe that if $\|\widehat{\beta}_1 - \beta^*\|_2$ were much larger than $\sigma$, this argument yields an error bound which depends on $ \sqrt{ \sigma^2 + \|\widehat{\beta} - \beta^*\|_2^2} $.)
\end{proof}

We now consider the case when $\widehat{\beta}_1$ might depend on the data.
Such a situation might arise if we were to perform sample splitting on an adversarially contaminated data set, meaning we would estimate $\widehat{\beta}_1$ from the first half of the data and use it to initialize a postprocessing step on the other half.
Since the adversary is allowed to look at the whole data set, this could lead to dependence between the two halves. 
In such a case, the argument used in the proof of Theorem~\ref{PropPost} cannot be applied because we do not necessarily have an i.i.d.\ data set when we condition on $\widehat{\beta}_1$. However, we may still obtain a looser error bound by taking a union bound over a large enough cover of $\cS^{p-1}$. We have the following result, proved in Appendix~\ref{AppPP}:

\begin{theorem}
\label{PropRobMeanMain}
Consider the setting and notation in Theorem~\ref{PropPost}, where $\widehat{\beta}_1$ might depend on $S$.
Set $\epsilon' = \Theta\left(\epsilon + \frac{\log(1/ \tau)}{n} + \frac{p\log (pn)}{n}\right)$, where $\epsilon'$ is less than a small constant.
Then running the filtering algorithm in Theorem~\ref{ThmStability} with inputs $T_1 := \left\{\widehat{\beta}_1 + (y_i' - (x_i)'^T \widehat{\beta}_1)x_i': (x_i',y_i') \in T\right\}$ and $\epsilon'$ returns a set $T'$ such that, with probability at least $1-2\tau$,
\begin{align*}
\|\widehat{\beta} - \beta^*\|_2 \lesssim \sigma \left( \sqrt{ \frac{p \log (pn)}{n}} + \sqrt{\epsilon} + \sqrt{ \frac{\log (1 / \tau)}{n} }\right),
\end{align*}
where $\widehat{\beta}$ is the empirical mean of the vectors in $T'$.
\end{theorem}

\begin{remark}
Compared to the error bound in Theorem~\ref{PropPost}, the error bound in Theorem~\ref{PropRobMeanMain} contains an extra factor of $\sqrt{\log(pn)}$ in the first term. This arises due to the covering argument we employ, since we cannot simply condition on $\widehat{\beta}_1$ and argue that we still have i.i.d.\ data.
\end{remark}

\begin{remark}
Cherapanamjeri et al.~\cite{CheATJFB20} show that when both the covariate and noise distributions are sub-Gaussian,
running the post-processing step once more to the output achieved by the procedure in Theorem~\ref{PropRobMeanMain} can improve the error dependence on $\epsilon$ from $O(\sigma\sqrt{\epsilon})$ to $O(\epsilon \log(1/\epsilon))$.
This is because when $\|\widehat{\beta}_1 - \beta^*\|_2 \lesssim \sigma\sqrt{\epsilon}$, the covariance matrix of $\widehat{\beta}_1 + (y_i' - (x_i)'^T \widehat{\beta}_1)x_i'$ is $O(\sigma^2 \epsilon)$-close to the spherical matrix $\sigma^2I$.
When covariate and noise distributions satisfy $(k,2)$-hypercontractivity, the same argument shows that the error dependence on $\epsilon$ would improve from $O(\sigma \sqrt{\epsilon})$ to $O(\sigma\epsilon^{1-1/k})$.
In comparison, the filtered Huber regression algorithm (cf.~Theorem~\ref{ThmAdvHuberReg}) provably achieves an error of the form $O(\sigma\epsilon^{1-1/k})$ under only a $k^{\text{th}}$ moment assumption on the covariate distribution.
\end{remark}

\section{Simulations}
\label{SecSim}
We now present the results of the simulations on synthetic data to validate our theoretical findings.
We demonstrate that covariate filtering improves estimation accuracy for both (i) heavy-tailed i.i.d.\ data (Section~\ref{Sec:SimHeavy}) and (ii) heavy-tailed data with adversarial corruption (Section~\ref{Sec:SimAdv}). 

For our simulations, we take $n= 200$ and $p = 40$, which roughly corresponds to the linear-data regime $n= O(p)$.
We measure the error in the usual $\ell_2$-norm, i.e., $\|\widehat{\beta}  - \beta^*\|_2$. 
For each plot, we conduct our experiments $T=50,000$ times, and report how the empirical quantiles of the $\ell_2$-error
increase with the failure probability $\tau$. 
The main goal of the plots is to demonstrate the effect of covariate filtering on Huber regression and LTS.

We first discuss the implementation details of these estimators, which were implemented on NumPy~\cite{harris2020array}.
For Huber regression, we ran gradient descent algorithm with a line-search procedure. 
For LTS, we ran our algorithm (Algorithm~\ref{AlgLTSFiltering}) for a fixed number of $100$ steps.
We found that both of these estimators converged with these choices of parameters.
In each experiment, we sample $\beta^*$ independently from a sphere of unit norm. We initialized all of our estimators at the same point, which is also sampled independently from a sphere of unit norm, and hence its $\ell_2$-distance from $\beta^*$ is at most $2$. 
We implemented the filter so that it removed a single point at every step, which corresponds to the version in Prasad et al.~\cite{PraBR19}.

We now discuss the data-generating mechanism in our experiments.
We use the family of (symmetrized) Pareto distributions for the choice of heavy-tailed distributions for both covariates $\{x_i\}$ and additive noise $\{z_i\}$.
For an $\alpha>0$, we say that a real-valued random variable $X$ follows an $\alpha$-symmetrized-Pareto distribution if the probability density function $f_X(x)$, has polynomial tails, i.e., for all $x \in \R$, $f_X(x) \propto \left(\frac{1}{|x| + 1}\right)^{1 + \alpha}$.
It can be seen that the $k^{\text{th}}$ moment of $X$ exists if and only if $k < \alpha$.
We say that a multivariate random variable $X$ follows an $\alpha $-symmetrized-Pareto distribution if each coordinate of $X$ is i.i.d.\ with an $\alpha$-symmetrized-Pareto distribution.

\subsection{Heavy-tailed regression}
\label{Sec:SimHeavy}

In this setting, we sample the data in an  i.i.d.\ fashion from a heavy-tailed distribution without any corruption. 
As mentioned earlier, we set $n=200$ and $p = 40$, $\|\beta^*\| = 1$, and ran our experiments $50,000$ times to calculate the empirical quantiles of various estimators as a function of $\tau$.
For our experiments, we sampled covariates and additive noise from symmetrized-Pareto distributions with parameter $2$.
Note that this choice of heavy-tailed distributions does not exactly satisfy our hypercontractivity assumption (Assumption~\ref{AsCov}), because the fourth moment is infinite. 

 \begin{figure}[!ht]
 \centering
    \begin{minipage}{0.5\textwidth}
        \centering
		\includegraphics[width=\textwidth]{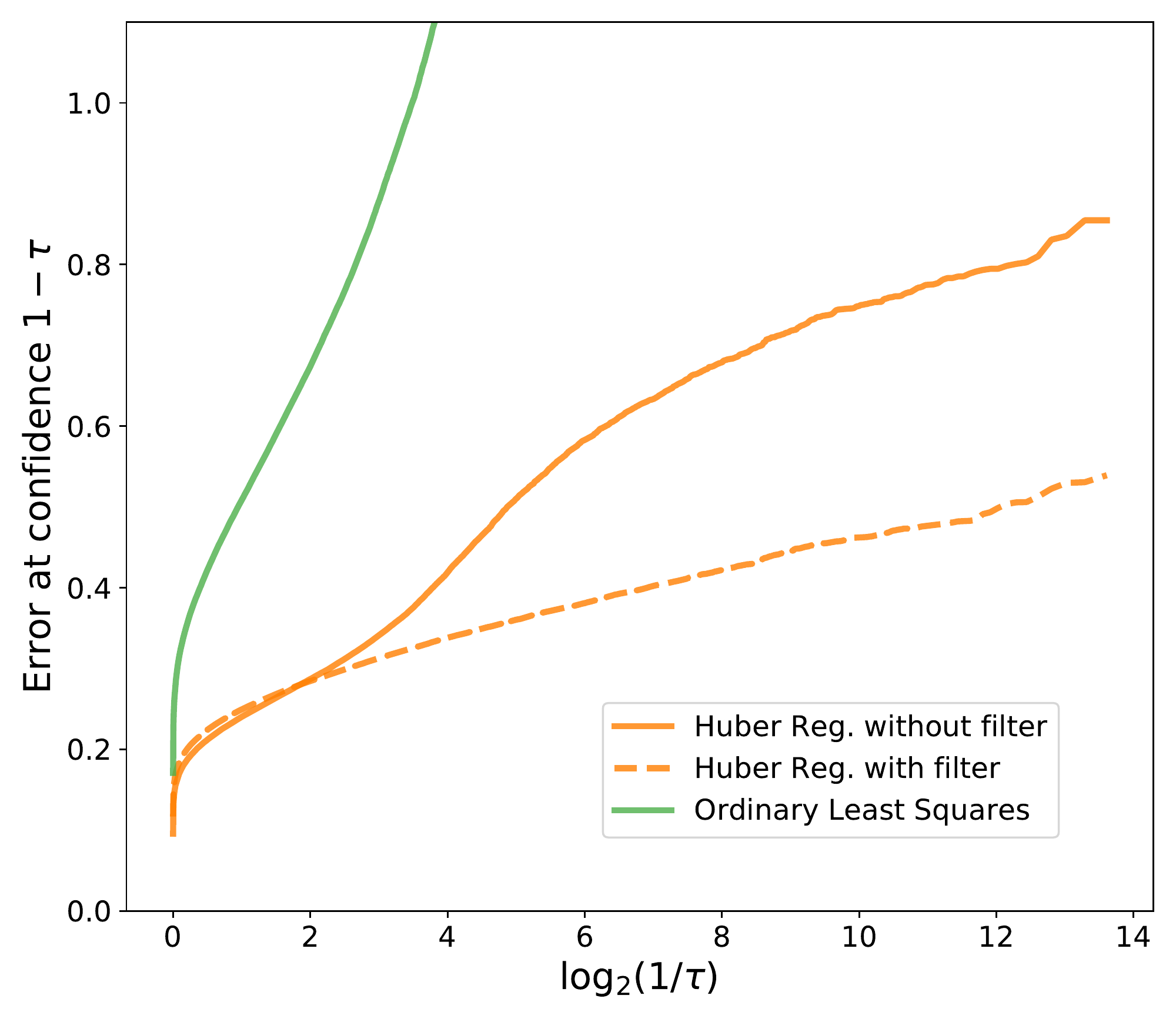}       \caption*{(a)} 
    \end{minipage}%
    \begin{minipage}{0.5\textwidth}
        \centering
		\includegraphics[width=\textwidth]{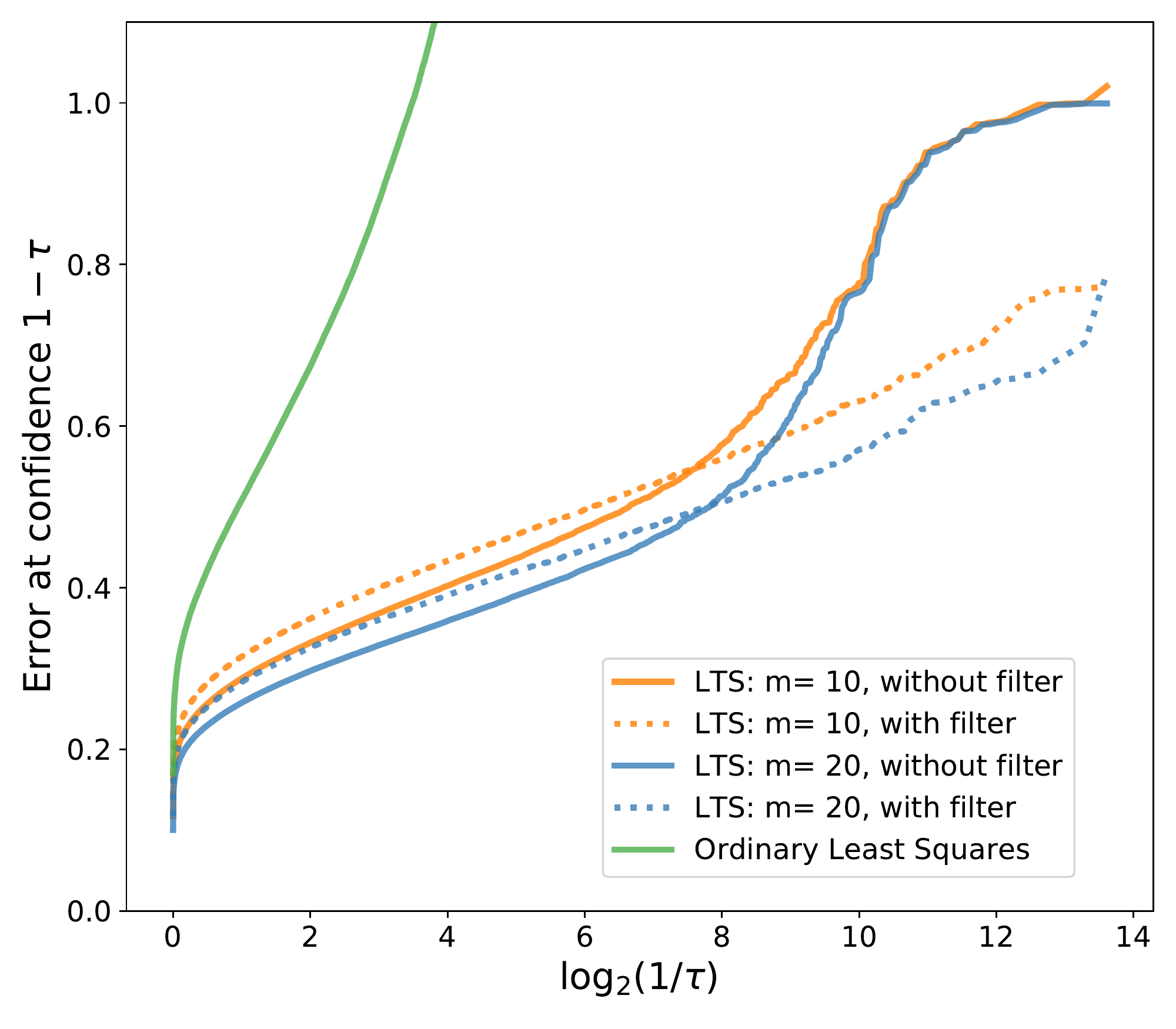}
		\caption*{(b)}        
    \end{minipage}%
\caption{Plots showing the effect of covariate filtering on (a) Huber regression and (b) LTS with heavy-tailed data $(n=200, p = 40)$.
For Huber regression, we set the Huber parameter $\gamma$ to be $0.5$. In plot (b), $m$ corresponds to the thresholding/trimming parameter in Algorithm~\ref{AlgLTSFiltering}.
The error is measured in terms of $\ell_2$-error, i.e., $\|\widehat{\beta} - \beta^*\|_2$.
Solid lines corresponds to ``vanilla'' version of the estimators (no filtering step), and dashed lines correspond to filtered versions, where the filtering step removes $10$ points out of $200$ points. 
We truncate the plots at $y=1.1$ to show the effect of filtering, but the maximum error of OLS is approximately $37$.
 }
\label{fig:heavy}
 \end{figure}

Figure~\ref{fig:heavy} shows that covariate filtering improves the performance of Huber and LTS significantly, especially in the high-confidence regime when $\tau \to 0$.
Figure~\ref{fig:heavy} demonstrates that even removing $10$ points out of $200$ points can boost the accuracy of both Huber regression and LTS, where the Huber parameter is set to be $0.5$.  Between Huber regression and LTS with filtering step, we find that Huber regression has better performance than LTS.
Additional plots showing the effect of filtering as $\gamma$ changes in Huber regression and as $m$ changes in LTS are included in Appendix~\ref{AppSims} (cf.~Figures~\ref{fig:hub_app} and~\ref{fig:lts_app}). We find that the same phenomenon as in Figure~\ref{fig:heavy} is demonstrated across a wide range of $\gamma$ and $m$.

\subsection{Adversarial corruption}
\label{Sec:SimAdv}

 \begin{figure}[!ht]
 \centering
    \begin{minipage}{0.6\textwidth}
        \centering
		\includegraphics[width=\textwidth]{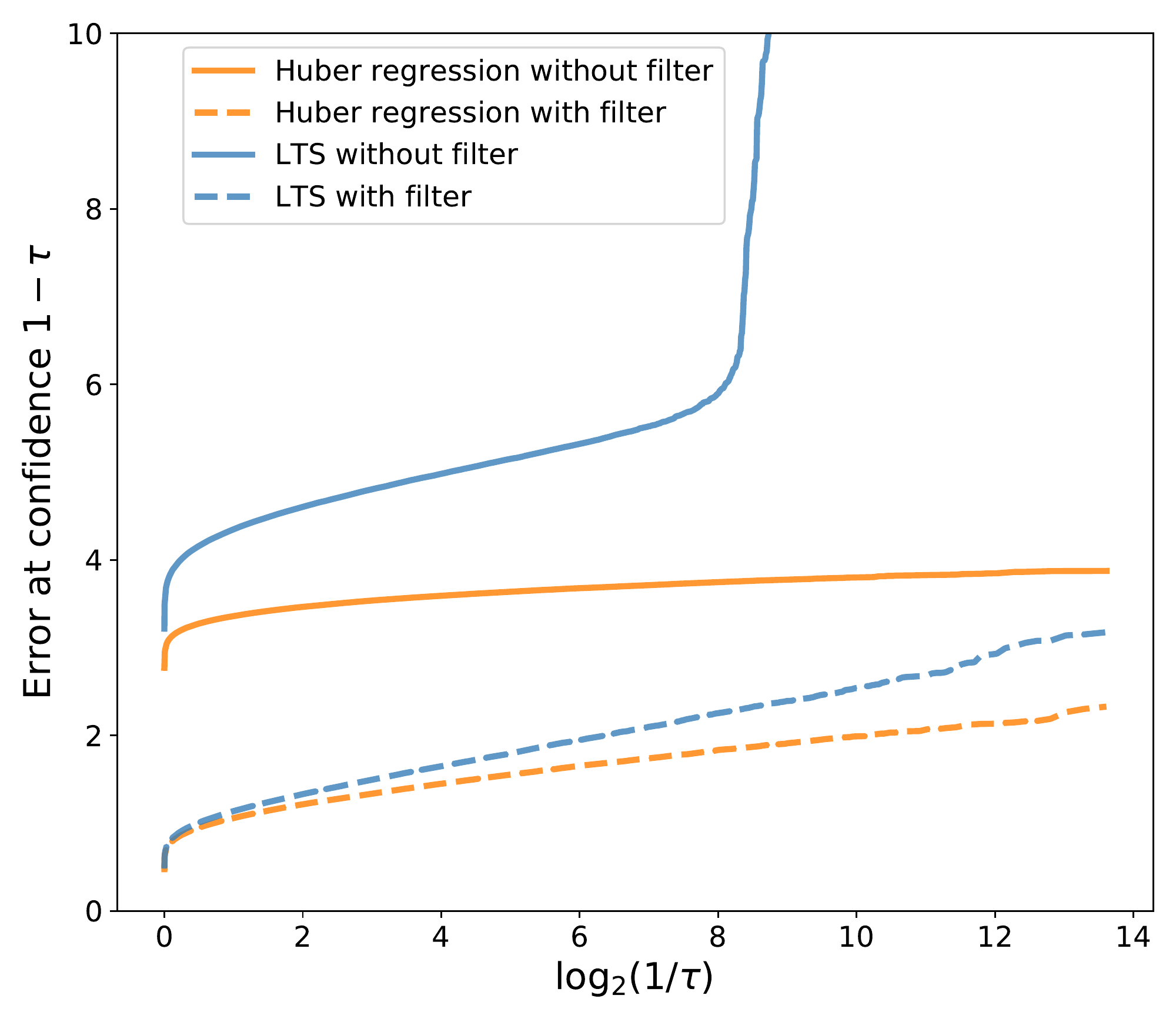}
    \end{minipage}%
\caption{ Plot showing the effect of covariate filtering on Huber regression and LTS when the data are sampled from a heavy-tailed distribution and contain adversarial corruption.
The plot corresponds to $n=200$, $p=40$, and $\epsilon=0.1$. 
The error is measured in terms of the $\ell_2$-error, i.e., $\|\widehat{\beta} - \beta^*\|_2$.
In the plot, solid lines corresponds to ``vanilla'' versions of the estimators (no filtering step), and dashed lines correspond to versions with filtering, where the filtering step removes $1.5 \epsilon n = 30$ points out of $200$ points. 
We see that the filtering step significantly improves the performance of both Huber regression and LTS.
 For ease of visualization, we do not show the error of the OLS estimator, whose minimum error is $18$ and maximum error is $150$. The maximum error of LTS without filtering is $16$.
 }
\label{fig:adv_lin}
 \end{figure}
We now explain our setup for adversarial corruption. Once again, we set $n=200$ and $p= 40$.
We sampled covariates and responses from symmetrized-Pareto distributions with parameters $4$ and $2$, respectively.
We consider the case $\epsilon = 0.1$, so $\epsilon n = 20$ points are corrupted in the following manner:
\begin{enumerate}
	\item We replace the covariates $\{x_i\}$ of 10 random points by the deterministic point $10w$, where $w$ is the vector with each coordinate equal to $1$.
	\item We replace the responses $\{y_i\}$ of 20 points, including the $10$ points selected in the previous step, by a deterministic value $200$.
\end{enumerate}

We do not corrupt the covariates of all $20$ points, because such a corruption scheme gives an advantage to the filtering step: if the filtering step perfectly removed all points with corrupted covariates, the data would effectively be clean in the responses, as well.
We run the filter so that it removes $1.5 \epsilon n = 30$ points from the data. For Huber regression, we again set the Huber parameter to be $0.5$. For LTS, we set $m=1.5 \epsilon n = 30$ to handle $\epsilon n$ corruption in responses.
Figure~\ref{fig:adv_lin} shows that the filtering step can significantly improve the performance of both Huber regression and LTS.

\section{Discussion} %
\label{sec:discussion}

In this paper, we have presented several estimators that are simultaneously robust to heavy-tailed distributions and adversarial contamination.
The main theme of our work is to show that a simple preprocessing step applied to the covariates can be used to make classical estimators such as the Huber regression, LTS, and LAD estimators robust to contamination in both covariates and responses.
Our preprocessing step leverages recent advances in algorithms for robust mean estimation, in which a filtering procedure was introduced to remove a small fraction of covariates to make the sample covariance matrix of the remaining points have a small spectral norm.
In particular, the modified Huber regression estimator achieves a near-optimal error guarantee in this setting, whereas the LTS and LAD estimators can be used for initialization and/or parameter tuning, or augmented with a preprocessing step to achieve near-optimal error rates.

Aside from the filtering method analyzed in this paper, we note that other algorithms have been proposed, which---instead of returning a subset $T'$ of the input data set $T$---return a distribution on $T$ such that the weight at any point is at most $\frac{1}{(1 - O(\epsilon))|T|}$~\cite{DiaKKLMS16-focs,SteCV18,DonHL19,CheDGS20,ZhuJS20}.
Although we have not pursued such algorithms here, one might prove analogous results for robust regression using these alternative methods for preprocessing via one of the following two approaches: (i) discretize the distribution to obtain a set $T'$ satisfying the conclusion in Theorem~\ref{ThmStability}; or (ii) study a weighted form of regression estimators (Huber regression, LAD, or LTS), where the loss at each point is weighted by the output of these algorithms. We leave a careful analysis of such algorithms to future work.

Thinking more broadly, it would be interesting to see which other common regression estimators might benefit from covariate filtering as a preprocessing step. Another important line of future work is to extend this methodology to settings where $\beta^*$ satisfies some structural assumptions, such as sparsity---this might involve proposing and analyzing a filtering step which would, with high probability, produce covariates which satisfy a restricted eigenvalue condition. Finally, we have assumed throughout the paper that the covariates and noise variables are independent, and the covariates are approximately isotropic; the question of whether our proposed algorithms could be analyzed under a more general dependency structure and unknown covariance which is not approximately isotropic remains open.

\section*{Acknowledgments}

AP and PL acknowledge support from NSF grant DMS-1749857. AP was also funded in part by the UW-Madison Institute for Foundations of Data Science (IFDS), NSF grant CCF-1740707. VJ acknowledges support from NSF grants CCF-1841190, CCF-1907786, and CCF-1942134.

\bibliography{allrefs}
\bibliographystyle{plain}

\appendix

\section{Auxiliary results}

We recall the Chernoff bound below~\cite{Ver18,BouLM13}:
\begin{lemma}
\label{ThmChernoff}
 Let $X_1,\dots,X_n$ be independent $\{0,1\}$-valued random variables.
Let $\widehat{\mu} = \frac{1}{n}\sum_{i=1}^n X_i$ be the empirical mean
and let $\mu$ denote its expectation, i.e., $\mu = \frac{1}{n} \sum_{i=1}^n \E X_i$.
Then with probability at least $1 - \tau$, we have
\begin{align*}
\widehat{\mu} \lesssim \mu + \frac{\log(1/\tau)}{n}.
\end{align*}
In particular, for $\kappa \geq 1$, we have $\widehat{\mu}  \leq 2 \kappa \mu$, with probability at least $1 - \exp(-  c \kappa n  \mu)$.
\end{lemma}

We will use the following version of Talagrand's concentration inequality regarding bounded empirical processes~\cite{Tal96}:
\begin{lemma}(Theorem 12.5 of Boucheron et al.~\cite{BouLM13}) 
\label{ThmTalagrand}
Let $X_1,\dots,X_n$ be $n$ i.i.d.\ vectors such that for each $s \in \cT$, we have $\E X_{i,s} = 0$ and $X_{i,s} \leq L$.
Define $Z := \sup_{s \in \cT} \sum_{i=1}^n X_{i,s}$, and
define $\sigma^2$ (the wimpy variance) to be $\sigma^2 := \sup_{s \in \cT} \E \sum_{i=1}^n X_{i,s}^2$.
Then with probability at least $1 - \tau$, we have
\begin{align*}
Z \lesssim \E Z + \sigma \sqrt{\log(1 / \tau)} + L \log(1 / \tau).
\end{align*}
\end{lemma}

We recall the following lemma from Lugosi and Mendelson~\cite{LugMen19-trim}:

\begin{lemma}(Lugosi and Mendelson~\cite{LugMen19-trim})
\label{LemTruncLin}
Let $X_1,\dots, X_n$ be $n$ i.i.d.\ points from a distribution over $\R^p$ with mean zero and covariance $\Sigma$.
For an $\epsilon > 0 $ such that $\epsilon = O(1)$, let $Q := C \left(\sqrt{\frac{\|\Sigma\|_{2}}{\epsilon}} + \frac{1}{\epsilon} \sqrt{\frac{\trace(\Sigma)}{n}}\right)$ for a large enough constant $C$.
For a unit vector $v$, define the set $S_v  := \left\{ i: |X_i^Tv| \geq Q\right\}$.
Let $\cE$ be the event $\cE = \{ \sup_{v} |S_v| \leq \epsilon n\} $.
Then with probability at least $ 1 - \exp(- n c\epsilon) $, the event $\cE$ holds.
\end{lemma}

We will also require the following generalization of the result above from Diakonikolas et al.~\cite[Lemma C.1]{DiaKP20}:

\begin{lemma}(Diakonikolas et al.~\cite{DiaKP20})
\label{LemTruncLinHigherMoment}
Let $X_1,\dots, X_n$ be $n$ i.i.d.\ points from a distribution over $\R^p$ with mean zero and covariance $\Sigma$. Suppose that for some $k \geq 2$, the inequality $\E\left((v^TX_i)^k\right)^{1/k} \leq \sigma_{x,k} \E\left((v^TX_i)^2\right)^{1/2}$ holds for all $v \in \cS^{p-1}$.
For some $\epsilon > 0 $ such that $\epsilon = O(1)$, define $Q := C \left(\sigma_{x,k} \sqrt{\|\Sigma\|_2} \epsilon^{-1/k} + \frac{1}{\epsilon} \sqrt{\frac{\trace(\Sigma)}{n}}\right)$ for a large enough constant $C$.
For a unit vector $v$, define the set $S_v  := \left\{ i: |X_i^Tv| \geq Q\right\}$.
Let $\cE$ be the event $\cE = \{ \sup_{v } |S_v| \leq \epsilon n\} $.
Then with probability at least $ 1 - \exp(- n c\epsilon) $, the event $\cE$ holds.
\end{lemma}
We also need the following version of the matrix Bernstein inequality:
\begin{lemma} (Corollary 7.3.2 of Tropp~\cite{Tro15})
\label{LemMatrixBernstein}
 Let $S_1,\dots,S_n$ be $n$ independent symmetric matrices such that $\E[S_i]=0$ and $\|S_i\|_2 \leq L$ a.s., for each index $i$.
Let $Z = \sum_{i=1}^nS_i$, and let $V$ be any positive semidefinite matrix such that $\sum_{i=1}^n\E[S_iS_i^T] \preceq V$. Let $\nu = \|V\|_2$ and $r= \mathrm{rank}(V)$. Then
\begin{align*}
\E [\|Z\|_2] \lesssim \sqrt{\nu \log r} + L \log r.
\end{align*}
In particular, if $S_i = \xi_i x_ix_i^T$, where $\xi_i$ is a Rademacher random variable and $x_i$ is sampled independently from a distribution with $\E[x_ix_i^T] = \Sigma$ and bounded support $\sqrt{L}$, i.e., $\|x_i\|_2 \leq \sqrt{L}$ a.s.\ for each index $i$, we have $\E[\|Z\|_2] \lesssim \sqrt{nL \|\Sigma\|_2\log(\mathrm{rank}(\Sigma))} + L \log(\mathrm{rank}(\Sigma))$.
\end{lemma} 

We will also use the following results:

\begin{lemma}(Lemma 6.1.2 of Vershynin~\cite{Ver18})
\label{LemConvex}
Let $Y$ and $Z$ be independent random variables such that $\E(Z) = 0$. Then for every convex function $f$, one has
\begin{equation*}
\E(f(Y)) \le \E(f(Y+Z)).
\end{equation*}
\end{lemma}

\begin{lemma}
\label{PropSymmetricQuantile}
Let $W$ and $Z$ be two independent symmetric random variables.
Let $Y:= W + Z$.
Then for any $r \geq 0$, we have $\P(|Z| \geq r) \le 2\P(|Y| \geq r)$.
\end{lemma}

\begin{proof}

Note that
\begin{equation*}
\{Z \geq r, W \geq 0\} \cup \{Z \leq -r, W \leq 0\} \subseteq \{|Y| \geq r\}.
\end{equation*}
Thus, by the independence of $W$ and $Z$ and the symmetry of $Z$, we have
\begin{align*}
\P(|Y| \geq r) &\geq \P(Z \geq r, W \geq 0) + \P (Z \leq -r, W \leq 0) \\
			&= \P(Z\geq r) \P(W \geq 0) + \P(Z\leq -r) \P(W \leq 0) \\
			&= \P(Z \geq r) \big(\P(W \geq 0) + \P(W \leq 0)\big) \\
			& \geq  \P(Z \geq r) \\
			& = \frac{1}{2}\P(|Z| \geq r),
\end{align*}
completing the proof.			
\end{proof}

We also recall the following result on convex functions from Sun et al.~\cite{SunZF20}:
\begin{lemma}
\label{LemConvexEta}
Let $\cL(\beta): \real^p \rightarrow \real$ be a convex function and let $\beta_1 \in \R^p$.
For some $\eta \in (0,1]$ and $\beta_2 \in \R^p$, let $\beta_\eta = \beta_1 + \eta(\beta_2 - \beta_1)$. Then we have
\begin{align*}
\langle \nabla \cL(\beta_\eta) - \nabla \cL(\beta_1), \beta_{\eta} - \beta_1 \rangle \leq \eta \langle \nabla \cL(\beta_2) - \nabla \cL(\beta_1), \beta_2 - \beta_1  \rangle.
\end{align*}
\end{lemma}

We will use the following standard properties regarding convexity and strong convexity~\cite{Nes04,BoydVand04}:
\begin{lemma}
\label{PropStrongCvx}
For a convex set $\cX \subseteq \R^n$, let $f$ be a continuously differentiable function $f: \cX \to \R$. Then the following statements hold:
\begin{enumerate}
  \item If $f$ is $\alpha$-strongly convex and continuously differentiable, then for any two points $x,y \in \cX$, we have
\begin{align*}
\langle \nabla f(y) - \nabla f(x), y -x \rangle \geq \alpha\|y-x\|_2^2.
\end{align*}
\item If $f$ is twice continuously differentiable, then $\nabla^2 f \succeq \alpha I$.
\item If $f$ is $\alpha_1$-strongly convex and $g$ is $\alpha_2$-strongly convex, then $f+g$ is $(\alpha_1 + \alpha_2)$-strongly convex.
\end{enumerate}
\end{lemma}

\section{Lower bounds for OLS and multivariate sample mean}

In this appendix, we derive a lower bound on the $\ell_2$-error of the OLS estimator by first proving a lower bound on the estimation error of the empirical mean.

\subsection{Lower bound for mean estimation}

We prove the following result regarding the estimation error of the sample mean. This result generalizes an analogous univariate result of Catoni~\cite[Proposition 6.2]{Cat12}.

\begin{proposition}
\label{PropMultiMean}
For any variance of $ \sigma^2 > 0$, dimension $p$, sample size $n$, and probability $\tau \le \frac{1}{4}$, there exists a multivariate distribution with mean $\mu \in \R^p$ and covariance $\sigma^2 I$ such that the sample mean $\widehat{\mu}$ on $n$ i.i.d.\ samples satisfies the bound
\begin{align*}
\left\| \widehat{\mu} - \mu\right\|_2^2 = \Omega\left(\frac{p\sigma^2}{n \tau}\right),
\end{align*}
with probability at least $\tau$. Moreover, the distribution of the random variable $\mu + ZX$ satisfies the bound, where $X$ is uniform on $\{-1,1\}^p$ and $Z$ is a univariate random variable supported on $\left\{-\sigma\sqrt{\frac{n}{2\tau}}, 0, \sigma\sqrt{\frac{n}{2\tau}}\right\},$ with
\begin{equation*}
\P\left(Z = -\sigma \sqrt{\frac{n}{2\tau}}\right) = \P\left(Z = \sigma \sqrt{\frac{n}{2\tau}}\right) = \frac{\tau}{n},
\end{equation*}
and $X$ and $Z$ are independent.
\end{proposition}
\begin{proof}

Without loss of generality, we will assume that $\mu = 0$.
Let $\epsilon =  \frac{1}{\sqrt{2n \tau}}$,
so
\begin{equation*}
\P(Z = - \sigma n \epsilon) = \P(Z = \sigma n \epsilon) = \frac{1}{2 n^2 \epsilon^2}
\end{equation*}
and $\P(Z = 0) = 1 - \frac{1}{n^2 \epsilon^2} $. Note that $\cov(ZX) = \E(Z^2) I = \sigma^2 I$.

Let $(X_1,\ldots,X_n)$ and $(Z_1,\ldots,Z_n)$ be independent pairs of $n$ i.i.d.\ random samples drawn from the distributions of $X$ and $Z$, respectively.
Let $W_i := Z_iX_i$, so the $W_i$'s are i.i.d.\ and $\widehat{\mu} = \frac{1}{n} \sum_{i=1}^n W_i$. Now note that for all $i$, we have $\|X_i\|_2= \sqrt{p}$. Hence, we can write
\begin{align*}
\P\left(\left\|\widehat{\mu} - \mu\right\|_2\geq \sigma \sqrt{p} \epsilon\right) & = \P\left(\left\|\frac{1}{n}\sum_{i=1}^n W_i\right\|_2\geq \sigma \sqrt{p} \epsilon\right) \\
& \geq \P\left(\exists i: \|W_i\|_2  \geq \sigma n \sqrt{p} \epsilon \text{ and } \forall j \neq i, \|W_j\|_2 = 0\right) \\
&=\P\left(\exists i: \|X_i\|_2|Z_i|  \geq \sigma n \sqrt{p} \epsilon \text{ and } \forall j \neq i, \|Z_jX_j\|_2 = 0\right)\\
&=\P\left(\exists i: |Z_i|  \geq \sigma n \epsilon \text{ and } \forall j \neq i, Z_j = 0 \right)\\
&= n  \cdot \frac{1}{n^2 \epsilon^2} \left( 1 - \frac{1}{n^2 \epsilon^2}\right)^{n-1} \\
& \geq \frac{1}{n \epsilon^2} \left( 1 - \frac{1}{n^2 \epsilon^2}\right)^{n}.
\end{align*}
We now simplify the last term using two simple observations: (i) $(1 + x)^ r \geq 1 + r x$, for $x \geq -1 $ and $r \geq 1$; and (ii) $ \frac{1}{n \epsilon^2} = 2\tau \leq \frac{1}{2}$:
\begin{align*}
\frac{1}{n \epsilon^2} \left( 1 - \frac{1}{n^2 \epsilon^2}\right)^{n} \geq \frac{1}{n \epsilon^2} \left( 1 - \frac{1}{n \epsilon^2}\right) \geq \frac{1}{2 n \epsilon^2} = \tau.
\end{align*}
Thus, we conclude that
\begin{align*}
\|\widehat{\mu} - \mu\|_2 \geq \sigma \sqrt{d } \epsilon = \sigma \sqrt{\frac{p}{2 n \tau}},
\end{align*}
with probability at least $\tau$.
\end{proof}

\subsection{Lower bound for OLS}
\label{AppOLS}

In this section, we state a lower bound for the OLS estimator using reductions to the sample mean.
We consider the following linear model:
\begin{align*}
y_i = x_i^T \beta^* + z_i, \qquad 1 \le i \le n,
\end{align*}
where $x_i$ and $z_i$ are independent. We also assume that $\E(z_i^2) = \sigma^2$.

\begin{proposition}(Lower bound for OLS for multivariate distributions)
\label{PropLowOLSMulti}
For every dimension $p$, sample size $n = \Omega(p)$, and probability $\tau \le \frac{1}{4}$ such that $\frac{\log(1/ \tau)}{n} = O(1)$, there exist covariate and error distributions satisfying Assumptions~\ref{AsCov}~and~\ref{AsNoise}, such that the OLS estimator $\widehat{\beta}_{\text{OLS}}$ satisfies the bound
\begin{align*}
\|\widehat{\beta}_{\text{OLS}} - \beta^*\|_2^2 = \Omega\left(\frac{p\sigma^2}{n \tau}\right),
\end{align*}
with probability at least $\frac{\tau}{2}$. Moreover, the bound is satisfied when the distribution of the covariates is uniform on $\{-1,1\}^p$, and the distribution of the noise is defined as in Proposition~\ref{PropMultiMean}.
\end{proposition}

\begin{proof}
Suppose the covariates and noise are sampled according to the stated distributions; we will show that the lower bound holds. Let the corresponding sampled points be denoted by $\{(x_i,y_i)\}_{i=1}^n$.

Note that the distribution of the covariates is $O(1)$-sub-Gaussian; i.e., for any unit vector $v$, we have $\|v^Tx\|_{\psi_2} = O(1)$. Thus, Assumption~\ref{AsCov} holds. Furthermore, the covariance matrix of the covariates has exponential concentration near the true covariance $I$, so if we denote $\Sigma_n = \frac{1}{n}\sum_{i=1}^n x_ix_i^T$ and define the event
\begin{align*}
\cE_1 := \left\{x_1,\ldots,x_n : \|\Sigma_n^{-1} - I\|_2 \leq 0.1 \right\},
\end{align*}
then $\P(\cE_1) \geq 1 - \exp(-cn)$ when $n = \Omega(p)$ (cf.\ Exercise 4.7.3 of Vershynin~\cite{Ver18}).

Define $\widehat{W} := \frac{1}{n}\sum_{i=1}^n x_iz_i $, and note that the OLS estimator satisfies $\widehat{\beta} - \beta^* = \Sigma_n^{-1} \widehat{W}$. Thus,
\begin{equation*}
\|\widehat{\beta}_{OLS} - \beta^*\|_2 \ge \|\widehat{W}\|_2 - \|(\Sigma_n^{-1} - I) \widehat{W}\|_2 \ge \|\widehat{W}\|_2 - \|\Sigma_n^{-1} - I\|_2 \|\widehat{W}\|_2.
\end{equation*}
Let $\cE_2$ be the event
\begin{align*}
\cE_2 := \left\{ \|\widehat{W}\|_2 = \Omega\left( \sqrt{\frac{p \sigma^2}{ n \tau}}\right)\right\}.
\end{align*}
Then on the event $\cE_1 \cap \cE_2$, we have
\begin{align*}
\|\widehat{\beta}_{OLS} - \beta^*\|_2 \geq 0.9 \|\widehat{W}\|_2 = \Omega\left( \sqrt{\frac{p \sigma^2}{ n \tau}} \right).
\end{align*}
Finally, note that $\P(\cE_2) \geq \tau$ by Proposition~\ref{PropMultiMean}, so $\P(\cE_1 \cap \cE_2) \geq \tau - \exp(- cn)  \geq \frac{\tau}{2}$, and the desired result follows.
\end{proof}

\section{Results regarding stability}
\label{AppStability}

In this appendix, we state and prove several results stemming from our notions of stability.

{ 
\begin{proposition}
\label{PropStabSimplified}
Let $S = \{x_1,\dots,x_n\}$ be a set of $n$ i.i.d. points in $\R^p$ from a distribution $P$ with mean $0$ and covariance $\Sigma$.
Suppose the following holds:
\begin{enumerate}
 	\item $\kappa_l I \preceq \Sigma \preceq \kappa_u I$, where $\kappa_l\in (0,1]$ and $\kappa_u \geq 1$ are constants.
 	\item The distribution $P$ satisfies $(4,2)$-hypercontractivity with parameter $\sigma_{x,4}$.
 \end{enumerate}
Let $\epsilon < c^*$, where $c^*$ is a small enough constant depending on $\sigma_{x,4}$ and $\frac{\kappa_l}{\kappa_u}$.
Suppose $n \gtrsim \frac{\kappa_u^2}{\kappa_l^2}\cdot \frac{(p \log p) \sigma_{x,4}^2}{\sqrt{\epsilon}} + \frac{\kappa_u}{\kappa_l}\cdot \frac{p}{\epsilon}$.
Then with probability at least $1 - O(\exp(- \Omega(n \epsilon)))$,
for every subset $S' \subseteq S$ such that $|S'| \ge (1 - \epsilon)n$, we have $\lambda_{\min}\left(\frac{1}{n} \sum_{i \in S'} x_ix_i^T\right) \geq 0.8\kappa_l$.
\end{proposition}
\begin{proof}
The proof follows the same principle as the references~\cite{KM15,DiaKP20}. In particular, the proof is similar to Diakonikolas et al.~\cite[Lemma 4.3]{DiaKP20} who consider the case when $\kappa_l = \kappa_u = 1$. For completeness, we provide a full proof here for the general case.

Let $r\geq 2$ denote a large enough constant to be specified later. First, we only consider distributions which are supported on a ball of radius at most $r \sigma_{x,4} \sqrt{\kappa_u}  \epsilon^{-1/4}\sqrt{p}$. (This is because a standard argument shows that we can simply ignore the points that do not satisfy this condition, since $(\E\|X\|_2^4)^{1/4} \leq \sigma_{x,4} \sqrt{\kappa_u}  \epsilon^{-1/4}\sqrt{p}$ for $X \sim P$,
as outlined at the end of the proof.) We will allow $P$ to have a nonzero mean $\mu$, as long as $\|\mu\|_2 \leq \sigma_{x,4} \sqrt{\kappa_u} \epsilon^{-1/4}$.

We will now apply Lemma~\ref{LemTruncLinHigherMoment}, which establishes a bound for an $(1- \epsilon)$-fraction of points when projected along any unit vector.
 Let $Q=C\left(\sigma_{x,4}\sqrt{\kappa_u} \epsilon^{-1/4} + \frac{1}{\epsilon} \sqrt{\frac{p \kappa_u}{n}}\right) + \|\mu\|_2$, which is greater than the threshold from Lemma~\ref{LemTruncLinHigherMoment} applied to the recentered distribution $P$.
Using the bound on $\|\mu\|_2$, we have $Q \lesssim \left(\sigma_{x,4}\sqrt{\kappa_u} \epsilon^{-1/4} + \frac{1}{\epsilon} \sqrt{\frac{p \kappa_u}{n}}\right)$.
Let $\cE$ denote the event from Lemma~\ref{LemTruncLinHigherMoment}, stating that for any unit vector $v$, we have $\left|\{i:|x_i^Tv|\geq Q\}\right| \leq \epsilon n$.
By Lemma~\ref{LemTruncLinHigherMoment}, we know that $\P(\cE) \ge 1 - \exp(-c n \epsilon)$.

We will now assume that the event $\cE$ holds and incur an additional failure probability of $\exp(-cn \epsilon)$ by a union bound.
Define the function $f: \R_+ \to \R_+$, as follows:
\begin{align*}
f(x) = \begin{cases} x, & \text{ if } x \in [0,Q^2],\\
Q^2, & \text{ otherwise,}\end{cases},
\end{align*}
and let $g(x) = -f(x)$. For any $v \in \cS^{p-1}$, on the event $\cE$, we have the following bound:
\begin{align*}
\min_{S': |S'| \geq (1 - \epsilon)n} \sum_{i \in S'} (x_i^Tv)^2 &\geq  \sum_{i=1}^n f((x_i^Tv)^2) - \epsilon Q^2 n \\
 &= - \left(\sum_{i=1}^n g((x_i^Tv)^2) - \E g((x_i^Tv)^2)\right) + n\E f((x_i^Tv)^2) - \epsilon Q^2 n.
\end{align*}
Taking an infimum over $v \in \cS^{p-1}$, we then have
\begin{multline}
\label{EqnTeddy2}
\inf_{v \in \cS^{p-1}} \min_{S': |S'| \geq (1 - \epsilon)n} \sum_{i \in S'} (x_i^Tv)^2
\ge   - \epsilon Q^2 n  - \sup_{v \in \cS^{p-1}} \left(\sum_{i=1}^n g((x_i^Tv)^2) - \E g((x_i^Tv)^2)\right) \\
+ n\left(\inf_{v \in \cS^{p-1}} \E f((x_i^Tv)^2) \right).
\end{multline}
Now define the random variable
\begin{align*}
N := \sup_{v \in \cS^{p-1}} \sum_{i=1}^n g((x_i^Tv)^2) - \E g((x_i^Tv)^2).
\end{align*}
Let $\xi_1,\dots,\xi_n$ be $n$ i.i.d.\ Rademacher random variables. 
We first bound the expectation of $N$ using symmetrization and contraction of Rademacher averages~\cite{LedTal91,BouLM13}:
\begin{align*}
\E N & \leq 2 \E \sup_{v \in \cS^{p-1}} \left|\sum_{i=1}^n \xi_i g((x_i^Tv)^2) \right| \leq 4 \E \sup_{v \in \cS^{p-1}} \left|\sum_{i=1}^n \xi_i (x_i^Tv)^2\right| \\
& \le 4 \E \left(\left\|\sum_{i=1}^n \xi_i x_ix_i^T\right\|_2 \right)\\
&\lesssim \frac{r^2\sigma_{x,4}^2 \kappa_u p \log p}{ \sqrt{\epsilon} } + \sqrt{\frac{n r^2\sigma_{x,4}^2 \kappa_u^2 p\log p}{ \sqrt{\epsilon}}} ,
\end{align*}
 where the last step uses the matrix Bernstein inequality (Lemma~\ref{LemMatrixBernstein}) with $L = (r \sigma_{x,4} \sqrt{\kappa_u}\epsilon^{-1/4}\sqrt{p})^2$ and $\nu = nL \kappa_u$, because
 $\|x_i\|_2 \leq r \sigma_{x,4} \sqrt{\kappa_u}\epsilon^{-1/4}\sqrt{p}$ and $\E x_ix_i^T \preceq  \kappa_u I$.
 We now bound the following term (which is usually called the \textit{wimpy variance}~\cite{BouLM13}):
\begin{align*}
\sigma^2 := \sup_{v \in \cS^{p-1}} n\Var( g((x_i^Tv)^2)) \leq \sup_{v \in \cS^{p-1}} n \E ((x_i^Tv)^2)^2 \leq n \sigma_{x,4}^4 (v^T \Sigma v)^2 \leq n \sigma_{x,4}^4 \kappa_u^2.
\end{align*}
Using Talagrand's inequality for bounded empirical processes (cf.\ Lemma~\ref{ThmTalagrand}), we therefore have that with probability at least $ 1 - \exp(- n \epsilon)$,
\begin{align*}
\frac{N}{n} &\lesssim \frac{r^2\sigma_{x,4}^2 \kappa_u p\log p}{n \sqrt{\epsilon} } + \sqrt{\frac{r^2\sigma_{x,4}^2 \kappa_u^2 p \log p}{n \sqrt{\epsilon}	}} + \sigma_{x,4}^2 \kappa_u \sqrt{\epsilon} + \epsilon Q^2 \\
 &\lesssim \frac{r^2\sigma_{x,4}^2 \kappa_u p \log p}{n \sqrt{\epsilon}	} + \sqrt{\frac{r^2\sigma_{x,4}^2 \kappa_u^2 p \log p}{n \sqrt{\epsilon}	}} + \sigma_{x,4}^2\kappa_u \sqrt{\epsilon} + \sigma_{x,4}^2 \kappa_u \sqrt{\epsilon} + \frac{p \kappa_u}{\epsilon n}\\
 &\lesssim \frac{r^2\sigma_{x,4}^2 \kappa_u p \log p}{n \sqrt{\epsilon}	} + \sqrt{\frac{r^2\sigma_{x,4}^2 \kappa_u^2 p \log p}{n \sqrt{\epsilon}	}} + \sigma_{x,4}^2 \kappa_u \sqrt{\epsilon} + \frac{p \kappa_u}{\epsilon n},
\end{align*}
where we use the definition of $Q$.
By taking $\epsilon \lesssim \left(\frac{\kappa_l}{\kappa_u}\right)^2\left(\frac{1}{\sigma_{x,4}}\right)^4$ and $n \gtrsim \frac{\kappa_u^2}{\kappa_l^2} \cdot \frac{r^2 \sigma_{x,4}^2 (p \log p)}{\sqrt{\epsilon}} + \frac{\kappa_u}{\kappa_l} \cdot \frac{p}{\epsilon}$, we can make the expression above less than $0.05\kappa_l$.
These calculations also show that we can upper-bound $\epsilon Q^2$ by $0.05 \kappa_l$. Thus, we have the following:
\begin{align}
\label{EqnBear2}
\max\left\{\frac{N}{n}, \epsilon Q^2\right\} \leq 0.05 \kappa_l.
\end{align}
Finally, note that for any $v \in \cS^{p-1}$, the Cauchy-Schwarz inequality gives 
\begin{align*}
\E\left| f((x_i^Tv)^2) - (x_i^Tv)^2\right| &= \E \left((x_i^Tv)^2 \1\{(x_i^Tv)^2 > Q^2\}\right) \leq \sqrt{\E(x_i^Tv)^4} \sqrt{\P(|x_i^Tv| > Q)} \\
&\leq \frac{\E[|x_i^Tv|^4]}{Q^2} \lesssim \frac{\sigma_{x,4}^4 \kappa_u^2}{\kappa_u \sigma_{x,4}^2 \epsilon^{-1/2}} = \sqrt{\epsilon} \sigma_{x,4}^2 \kappa_u,
\end{align*}
implying that there exists a constant $c > 0$ such that
\begin{align*}
\E f( (x_i^Tv)^2) \geq \E (x_i^Tv)^2 - c\sigma_{x,4}^2 \sqrt{\epsilon} \kappa_u \ge \kappa_l - c\sigma_{x,4}^2 \sqrt{\epsilon} \kappa_u.
\end{align*}
Taking $\epsilon \lesssim \left(\frac{\kappa_l}{\kappa_u}\right)^2\left(\frac{1}{\sigma_{x,4}}\right)^4$, we have 
\begin{align}
\label{EqnDuck2}
\E f( (x_i^Tv)^2) \geq 0.95 \kappa_l.
\end{align}
Combining inequalities~\eqref{EqnTeddy2}, \eqref{EqnBear2}, and~\eqref{EqnDuck2}, we then obtain the bound
\begin{align*}
\frac{1}{n} \inf_{v \in \cS^{p-1}} \min_{S': |S'| \geq (1 - \epsilon)n} \sum_{i \in S'} (x_i^Tv)^2
&\geq \inf_{v\in \cS^{p-1}} \E f((x_i^Tv)^2) - \epsilon Q^2 - \frac{N}{n} \\
&\geq 0.95\kappa_l - 0.05 \kappa_l - 0.05 \kappa_l \geq 0.85 \kappa_l.
\end{align*}
This completes the proof.

\paragraph{Unbounded support:} We now outline a general argument for the case when the support of the distribution is unbounded.
Let $X \sim P$. By Jensen's inequality and $(4,2)$-hypercontractivity, we have
\begin{align*}
\E \|X\|_2^4 = p^2 \E \left[\left(\sum_{j=1}^p \frac{1}{p} X_j^2 \right)^2\right] \leq p^2 \E \left[\sum_{j=1}^p\frac{1}{p}\left(X_j^2 \right)^2\right] = p\E \left[\sum_{j=1}^p X_j^4 \right] \leq \sigma_{x,4}^4 p^2\kappa_u^2,
\end{align*}
since for each $j$, we have $\E[X_j^4] = \E[(e_j^TX)^4] \leq \sigma_{x,4}^4 \|\Sigma\|_2^2$, where $e_j$ is the canonical basis vector. Applying Markov's inequality, we then obtain
\begin{align*}
\P\{\|X\|_2 > r  \sigma_{x,4} \sqrt{\kappa_u} \epsilon^{-1/4} \sqrt{p} \} \leq \frac{\E \|X\|_2^4}{r^4  \sigma_{x,4}^4 \kappa_u^2 \epsilon^{-1} p^2} \leq \frac{\epsilon}{r^4},
\end{align*}
where $r \ge 2$ is the constant to be specified below. Let $\cE_r = \{x: \|x\|_2 \leq r \sigma_{x,4} \sqrt{\kappa_u} \epsilon^{-1/4} \sqrt{p}\}$.
Applying a Chernoff bound, we see that with probability at least $1 - \exp(- c n \epsilon)$, at most $\frac{n \epsilon}{2}$ points lie outside $\cE_r$, where we take $r$ to be a sufficiently large constant.
Let $P_r$ be the distribution of $P$ conditioned on $\cE_r$.
Simply ignoring the points that lie outside $\cE_r$, we will only focus on points that come from the distribution $P_r$ and incur an additional failure probability of $\exp(-cn \epsilon)$.

Let $y_1,\dots,y_m$ be $m$ i.i.d.\ points from $P_r$, where $m \geq n\left(1 - \frac{\epsilon}{2}\right)$. It suffices to show that any subset of $\{y_1,\dots,y_m\}$ of size at least $\left(1- \frac{\epsilon}{2}\right)m$ satisfies the desired conclusion.
This is exactly what was considered in the first part of the proof, up to constant factors; thus, it remains to show that the distribution $P_r$ satisfies $(4,2)$-hypercontractivity and has an appropriately bounded second moment matrix.

Let $Z_r \sim P_r$ and $X \sim P$. For any $v \in \cS^{p-1}$, we have $\E(v^TZ)^2 \leq \E (v^TX)^2$.
We now look at the lower bound:
\begin{align*}
\P(X \in \cE_r) \E[(v^TZ)^2]  &=  \E \left[(v^TX)^2 \1_{X \in \cE_r}\right]\\
  & = \E (v^TX)^2 - \E[(v^TX)^2\1_{X \in \cE_r^c}]\\
	&\geq \E (v^TX)^2 - \sqrt{\E[(v^TX)^4]} \sqrt{\P(X \not\in\cE_r)}\\
	&\geq \E (v^TX)^2 - \sigma_{x,4}^2 \E (v^TX)^2 \sqrt{\epsilon r^{-4}}\\
	&\geq \E[(v^TX)^2] (1 - \sigma_{x,4}^2 \sqrt{\epsilon} r^{-2}).
\end{align*}
This shows that $\E (v^TZ)^2 \geq 0.99\kappa_l$, when $\epsilon\lesssim \kappa_l^2 r^4 \sigma_{x,4}^{-4}$.
It also shows that $P_r$ satisfies $(4,2)$-hypercontractivity, as follows:
\begin{align*}
\left(\E(v^TZ_r)^4\right)^{1/4} \leq \left(\E(v^TX)^4\right)^{1/4} \leq \sigma_{x,4} \left(\E(v^TX)^2\right)^{1/2} \leq \frac{\sigma_{x,4}}{\left(1 - \sigma_{x,4}^2 \sqrt{\epsilon}r^{-2}\right)^{1/2}} \left(\E(v^TZ)^2\right)^{1/2}.
\end{align*}
Thus, when $\epsilon \lesssim r^4\sigma_{x,4}^{-4}$, we see that $P_r$ satisfies ($4,2$)-hypercontractivity with $\sigma_{x,4}' \leq 2 \sigma_{x,4}$.  
Finally, we note that $P_r$ might not be centered, but the means of $P_r$ and $P$ differ by at most $\sigma_{x,4} \sqrt{\kappa_u} \epsilon^{3/4}$ in the Euclidean norm: for any unit vector $v \in \cS^{p-1}$, we have
\begin{align*}
|\E[v^TZ]| & \leq |2\P(X \in \cE_r) \E[v^TZ]| \\
& =  2\left|\E \left[v^TX\1_{X \in \cE_r}\right]\right|\\
  & = 2\left|\E [v^TX] - \E[(v^TX)\1_{X \in \cE_r^c}]\right|\\
  & = 2\left|\E[(v^TX)\1_{X \in \cE_r^c}]\right|\\
	&\leq 2\left(\E[(v^TX)^4]\right)^{1/4} \left(\P(X \not\in\cE_r)\right)^{3/4}\\
	&\leq 2\sigma_{x,4} \sqrt{\kappa_u} \epsilon^{3/4} r^{-3},
\end{align*}
using the facts that $\P\{X \in \cE_r\} \geq \frac{1}{2}$ and $\P\{X \not\in \cE_r\} \leq \frac{\epsilon}{r^4}$.
The proof now follows from the bounded support setting considered above, which allows the norm of the mean to be as large as $\sigma_{x,4} \sqrt{\kappa_u}\epsilon^{-1/4}$.
\end{proof}
}
\begin{proposition}
\label{PropStabSimpleV2}
Consider the setting of Theorem~\ref{ThmStabHighProb} with $k=4$.
Let $\epsilon < c^*$, where $c^*$ is a small enough constant.
Let $C$ be any large constant.
Suppose $n = \Omega\left(\frac{p \log p}{\epsilon}\right)$.
Then for any $\tau = O(\exp( - \Omega(n \epsilon)))$, 
with probability at least $1 - \tau$,
there exists a set $S_1 \subseteq S$ such that
\begin{itemize}
\item[(i)] $|S_1| \geq (1 - \epsilon)n$,
\item[(ii)] $S_1$ is $ (\epsilon_1, \delta_1)$-stable, where $\epsilon_1 = C\epsilon$ and $\delta_1 = O\left( \sqrt{\frac{p \log p}{n}} + \sigma_{x,4} \epsilon^{3/4} + \sigma_{x,4} \sqrt{\frac{\log(1/\tau)}{n}}\right)$, and
\item[(iii)] $\frac{\delta_1^2}{\epsilon_1} < 0.01$.
\end{itemize}

Moreover, let $T$ be an $\epsilon'$-corrupted set version of $S$, where $\epsilon' \le \epsilon$. Let $T_1$ be the output of the filter algorithm with input $T$ and $\epsilon$. Then with probability at least $1 - 2 \tau$, the set $T_1$ satisfies
\begin{itemize}
\item[(i)] $|T_1| \geq (1 - c_1 \epsilon)n$,
\item[(ii)] $T_1$ is $ (\epsilon_2, \delta_2)$-stable, where $\epsilon_2 = c_2C \epsilon$ and $\delta = O\left(\sqrt{\frac{p \log p}{n}} + \sigma_{x,4} \epsilon^{3/4} + \sigma_{x,4} \sqrt{\frac{\log(1/\tau)}{n}}\right)$, and
\item[(iii)] $\frac{\delta_2^2}{\epsilon_2} < 0.05$.
\end{itemize}
\end{proposition}

\begin{proof}
We will show that these statements are consequences of Theorems~\ref{ThmStability} and~\ref{ThmStabHighProb}.

Fix the constant $C$, the desired premultiplier in the stability results. 
Let $\epsilon_3  > 0$ be a value to be decided later, and let $\tau$ be such that $\frac{\log(1/\tau)}{n} \leq c_1 \epsilon_3$.
Suppose $\epsilon_3$ is such that $\epsilon := C_1\left(\epsilon_3 + \frac{\log(1/\tau)}{n}\right)$ is the parameter in Theorem~\ref{ThmStabHighProb}.
Applying Theorem~\ref{ThmStabHighProb}, we see that with probability $1 - \tau$,
there exists a $(C \epsilon, \delta_1)$-stable set $S' \subseteq S$, with $|S'| \geq (1 - \epsilon) |S|$ and $\delta_1 = O\left(\sqrt{\frac{p \log p}{n}}  + \sigma_{x,4} \epsilon_3^{3/4} + \sigma_{x,4} \sqrt{\frac{\log(1/\tau)}{n}}\right)$, where the premultiplier depends on $C$.

Note that
\begin{align*}
\frac{\delta_1^2}{ \epsilon_1} &\lesssim \frac{p \log p }{n \epsilon} + \sigma_{x,4}^2 \epsilon_3^{1/2} + \sigma_{x,4}^2 \frac{\log(1/\tau)}{n \epsilon} \\
&\lesssim \frac{p \log p}{n \epsilon} + \sigma_{x,4}^2 \sqrt{\epsilon} + \sigma_{x,4}^2\frac{c_1}{ C_1}.
\end{align*}
The last expression can be made less than $0.01$ by choosing $n = \Omega\left(\frac{p \log p}{\epsilon}\right)$, restricting $\epsilon$ (and thus $\epsilon_3$) to be less than a  small enough constant $c^*$, and choosing $c_1$ to be small enough.
The last condition yields that the failure probability can be made as small as $\exp(- \Omega(n \epsilon))$.
This completes the proof of the first statement.
Moreover, the bound $0.01$ was arbitrary and can be made as small as required under qualitatively similar constraints.

For the second part, we assume that the constant $C$ is large enough for Theorem~\ref{ThmStability} to succeed.
By the first part, we know that with probability at least $1 - \exp(- \Omega(n \epsilon))$, there exist $S_1 \subseteq S$ such that $|S_1| \geq (1 - \epsilon)|S|$ and $S_1$ is $(C \epsilon, \delta_1)$-stable.
Theorem~\ref{ThmStability} then implies that with probability at least $1 - O(\exp(- \Omega(n \epsilon)))$, the output of the filter algorithm $T_1$ satisfies
$|T_1| \geq (1 - c_1 \epsilon) n$ and is $(\epsilon_2, \delta_2)$-stable, where $\epsilon_2 = c_2 C \epsilon$ and $\delta_2 = c_3 \delta_1$.
It remains to check that $\frac{\delta_2^2}{\epsilon_2} < 0.05$. Note that $\frac{\delta_2^2}{\epsilon_2} = \frac{c_3^2}{c_2 C} \cdot \frac{\delta^2}{\epsilon}$.
Since $c_3$, $c_2$ and $C$ are constants, we can make $\frac{\delta_2^2}{\epsilon_2} < 0.05$ by taking $\frac{\delta_1^2}{\epsilon} < 0.05 \cdot \frac{c_2C}{c_3^2}$ in the first part.
\end{proof}

\begin{proposition}
\label{PropStabSqError}
Let $ \{x_1,\dots,x_n\}$ be an $(\epsilon,\delta)$-stable set with respect to $\mu$ and $\sigma^2$.
Then for any unit vector $v$ and any $S' \subseteq [n]$ such that $|S'| \leq \epsilon n$, we have
\begin{align}
\frac{1}{n}\sum_{i \in S'}  ((x_i - \mu)^Tv)^2 \leq  \frac{3 \sigma^2\delta^2}{\epsilon}.
\label{EqStabSqError}
\end{align}
\end{proposition}
\begin{proof}
Without loss of generality, we assume that $\mu =0$ and $\sigma^2 = 1$.
By the stability assumption, we have the inequality
\begin{align*}
\frac{1}{n}\sum_{i \in [n]}  (x_i^Tv)^2 \leq  1 + \frac{\delta^2}{\epsilon}.
\end{align*}
Furthermore, using the lower bound on eigenvalues over the set $[n]\setminus S'$, we have
\begin{align*}
\frac{1}{|[n]\setminus S'|}\sum_{i \in [n]\setminus S'}  (x_i^Tv)^2 \geq  1 - \frac{\delta^2}{\epsilon}.
\end{align*}
Combining the inequalities, we obtain
\begin{align*}
\frac{1}{n}\sum_{i \in S'}  (x_i^Tv)^2  &= \frac{1}{n}\sum_{i \in [n]}  (x_i^Tv)^2 - \frac{|[n]\setminus S'|}{n}  \frac{1}{|[n]\setminus S'|}\sum_{i \in [n]\setminus S'}  (x_i^Tv)^2\\
&\leq \left(1 + \frac{\delta^2}{\epsilon}\right) - (1- \epsilon)\left(1 - \frac{\delta^2}{\epsilon}\right) \\
&= \frac{2 \delta^2}{\epsilon} + \epsilon - \delta^2 \leq \frac{3 \delta^2}{\epsilon},
\end{align*}
where we use the fact that $\epsilon \leq \delta$.
\end{proof}

\begin{proposition}
\label{PropStabL1Error}
Let $ \{x_1,\dots,x_n\}$ be an $(\epsilon,\delta)$-stable set with respect to $\mu$ and $\sigma^2$.
Then for any unit vector $v $ and any $S' \subseteq [n]$ such that $|S'| \leq \epsilon n$, we have
\begin{align}
\frac{1}{n}\sum_{i \in S'} {|(x_i - \mu)^Tv|} \leq  2 \sigma \delta.
\label{EqStabL1Error}
\end{align}
\end{proposition}
\begin{proof}
Without loss of generality, we assume that $\mu = 0$ and $\sigma^2=1$.
By Proposition~\ref{PropStabSqError}, we have
\begin{align*}
\frac{1}{n} \sum_{i \in S'} (x_i^Tv)^2 \leq \frac{4\delta^2}{\epsilon}.
\end{align*}
Applying the Cauchy-Schwarz inequality, we then have
\begin{align*}
\frac{1}{|S'|}\sum_{i \in S'}  |x_i^Tv| \leq \sqrt{\frac{1}{|S'|}\sum_{i \in S'}  |x_i^Tv|^2 } \leq \sqrt{\frac{n}{|S'|} \frac{4 \delta^2 }{\epsilon} }.
\end{align*}
Hence, we obtain
\begin{align*}
\frac{1}{n}\sum_{i \in S'}  |x_i^Tv| = \frac{|S'|}{n} \frac{1}{|S'|}\sum_{i \in S'}  |x_i^Tv| \leq \frac{|S'|}{n}  \sqrt{\frac{n}{|S'|} \frac{4 \delta^2 }{\epsilon} } =   \sqrt{\frac{|S'|}{n} \frac{4\delta^2}{\epsilon}   } \leq 2\delta.
\end{align*}
\end{proof}

\begin{proposition}
\label{PropStabL1Mean}
Let $ \{x_1,\dots,x_n\}$ be an $(\epsilon,\delta)$-stable set with respect to $\mu$ and $\sigma^2$.
Let $a_1,\dots,a_n$ be scalars and suppose $\max_{1 \le i \le n} |a_i| \leq a$.
Then for any $S' \subseteq [n]$ such that $|S'| \leq \epsilon n$, we have
\begin{align}
\left\|\frac{1}{n} \sum_{i \in S'} {a_i (x_i - \mu)}\right\|_2 \leq  2a \sigma \delta.
\label{EqL1Const}
\end{align}
\end{proposition}

\begin{proof}
Without loss of generality, we assume that $\mu = 0$ and $\sigma^2 = 1$.
We have
\begin{align}
\left\|\frac{1}{n} \sum_{i \in S'} {a_i x_i}\right\|_2 = \frac{1}{n} \sup_{v \in \cS^{p-1}} \sum_{i \in S'} a_i x_i^T v \le \frac{1}{n} \sup_{v \in \cS^{p-1}}
\sum_{i \in S'} {|a_i||x_i^Tv|} \leq \frac{a}{n} \sup_{v \in \cS^{p-1}}
\sum_{i \in S'} {|x_i^Tv|} \leq 2a\delta,
\end{align}
where the last step uses Proposition~\ref{PropStabL1Error}.
\end{proof}

\section{Huber regression}
\label{AppHuber}

In this appendix, we provide additional proof details for the results in Section~\ref{SecHuber}.

\subsection{Estimation of $\gamma$}

In this section, we prove that the sample-splitting procedure outlined in Section~\ref{SecHuberGeneral} succeeds with high probability. We use the result of Theorem~\ref{ThmLAD}, as well as the following lemma, where we denote $\epsilon = c^*$ for notational brevity.
\begin{lemma}
\label{LemEstGamma}
Let $S= \{(x_1,y_1), \dots, (x_{2n},y_{2n})\}_{i=1}^{2n}$ be i.i.d. points from the linear model $y_i = x_i ^T \beta^* + z_i$, where the covariates are centered and isotropic, and the noise is independent of the covariates and satisfies $\E |z_i| = \kappa < \infty$.
Let $\widehat{\beta}_0$ be an estimator independent of $S$ such that $\|\widehat{\beta}_0 - \beta^*\|_2 = O(\kappa)$.
Then the sample-splitting estimator $\widehat{\gamma}$ with $\epsilon = c^*$ satisfies
\begin{itemize}
\item[(i)] $\P\left(|Z_1 - Z_2| \geq \frac{\widehat{\gamma}}{\sqrt{2}}\right) < \epsilon$, and
\item[(ii)] $|\widehat{\gamma}| = O\left(\frac{\kappa}{\epsilon}\right)$,
\end{itemize}
with probability at least $1 - 2\exp(- \Omega(n \epsilon^2))$.
\end{lemma}

\begin{proof}
Let $\beta_1 = \beta^* - \widehat{\beta}_0$. Note that conditioned on $\widehat{\beta}_0$, the pairs $\{(x'_i,w_i')\}_{i=1}^{\lfloor n/2 \rfloor}$ are i.i.d.\ draws from the linear model
\begin{equation}
\label{EqnLinModelShift}
w_i' = (x_i')^T \beta_1 + z_i',
\end{equation}
where $z_i' \stackrel{d}{=} \frac{z_1 - z_2}{\sqrt{2}}$ is the symmetrized version of the error variables.

Let $x'$, $w'$, and $z'$ denote generic random variables with the same distributions as $x_i'$, $w_i'$, and $z'_i$, respectively. Note that $x'$ is centered and isotropic, and $z'$ is symmetric with $\E |z'| \leq \sqrt{2} \kappa$.
By the triangle inequality, we therefore have
\begin{equation*}
\E |w'| \leq \E|(x')^T \beta_1| + \E |z'| \le \sqrt{\E\left((x')^T \beta_1\right)^2} + \E|z'| \leq \|\beta_1\|_2 + \sqrt{2} \kappa = O(\kappa),
\end{equation*}
using the fact that $x'$ is isotropic and $\|\beta_1\|_2 = O(\kappa)$ by assumption.

Now let $F_n$ denote the empirical cdf of the $|w'_i|$'s, so $F_n(t) = \frac{1}{n}\sum_{i=1}^n \1(|w'_i| \le t)$.
Define the event
\begin{align*}
\cE := \left\{ \sup_{t \in \R}|F_n(t) - \P(|w'| \leq t)| \leq \frac{\epsilon}{8}\right\}.
\end{align*}
By the Dvoretzky-Kiefer-Wolfowitz inequality~\cite{Mas90}, we know that $\P(\cE) \geq 1 - 2\exp( -n \epsilon^2/32)$.
Note that by definition, we have $\frac{\widehat{\gamma}}{2} = \inf\left\{t : F_n(t) \geq 1 - \frac{\epsilon}{4} \right\}$. On the event $\cE$, we therefore have
\begin{equation}
\label{EqnWprime}
\P\left(|w'| \ge \frac{\widehat{\gamma}}{2}\right) \leq \frac{3\epsilon}{8}.
\end{equation}
Furthermore, since both $z'$ and $(x')^T \beta_1$ are symmetric random variables, Lemma~\ref{PropSymmetricQuantile} applied to the linear model~\eqref{EqnLinModelShift} gives us
\begin{equation*}
\P\left(|z'| \geq \frac{\widehat{\gamma}}{2}\right) \leq 2\P\left(|w'| \geq \frac{\widehat{\gamma}}{2}\right) \leq \frac{3\epsilon}{4} < \epsilon,
\end{equation*}
which is part (i).

We now show that $|\widehat{\gamma}| \le \frac{8 \E |w'| }{\epsilon}$ on the event $\cE$.
Suppose the contrary.
By Markov's inequality, we would have
\begin{equation*}
\P\left(|w'| \geq \frac{\widehat{\gamma}}{2}\right)  \leq \P\left(|w'| \geq \frac{4\E |w'|}{\epsilon}\right) \leq \frac{\epsilon}{4},
\end{equation*}
which contradicts inequality~\eqref{EqnWprime}. Therefore, we must have $\widehat{\gamma} = O\left(\frac{\E |w'|}{\epsilon}\right) = O\left(\frac{\kappa}{\epsilon}\right)$, as wanted.
\end{proof}

\subsection{Proof of Theorem~\ref{ThmAdvHuberReg}}
\label{AppThmAdvHuberReg}

In the course of this proof, we will need to refer to set functions that take a finite set as the argument and return a value in $\real$. The sets we consider will be of the form $S = \{(u_1, v_1), (u_2, v_2), \dots, (u_n, v_n)\}$, where $u_i \in \real^p$, $v_i \in \real$, and $n \geq 1$. The set functions will be of the following form:
\begin{align*}
F(S) := \sum_{ i = 1}^n f(u_i, v_i),
\end{align*}
for some $f: \real^p \times \real \to \real$. For ease of notation, we will use the following convention:
\begin{align*}
F(S) = \sum_{(x, y) \in S} f(x, y).
\end{align*}
This simplifies notation by avoiding explicit indexing of the elements in the sets being considered. For example, if $S' \subseteq S$, we may express $F(S') = \sum_{(x, y) \in S'} f(x, y)$. 

For ease of presentation, we also redefine the algorithm with different notation, as reflected in Algorithm~\ref{AlgHubAdvCont}.

\begin{algorithm}[h]  
  \caption{Huber Regression - Adversarial Corruption} 
    \label{AlgHubAdvCont}  
  \begin{algorithmic}[1]  
    \Statex  
    \Function{Huber\_Regression\_With\_Filtering}{$T = \{x'_i,y'_i: i \in [2n] \}, \gamma, \tilde{\epsilon} $}  
        \For{$i \gets 1$ to $n$}                    
        \State $(\tilde{x}_i,\tilde{y}_i)$ $\gets$ $\left(\frac{x_{i}'-x_{n+i}'}{\sqrt{2}}, \frac{y_i' - y_{n+i}'}{\sqrt{2}}\right)$
    \EndFor
\State $T_1 \gets \{  (\tilde{x}_i, \tilde{y}_i )\}_{i=1}^n$
            \State $T_2 \gets $ FilteredCovariates$(T_1,\epsilon_1')$%
        \State $\widehat{\beta} \gets $ HuberRegression$(T_2,\gamma)$
    \State \Return $\widehat{\beta}$
    \EndFunction  
  \end{algorithmic}  
\end{algorithm}

We state the following technical lemma, which is proved in Appendix~\ref{AppLemHubAdv}:

\begin{lemma}
\label{LemHubAdv}
Under the setting of Theorem~\ref{ThmAdvHuberReg}, 
with probability at least $ 1 - 2\tau$, we have the following statements:

\begin{itemize}
\item[(i)] The filtered set of covariates $T_2$ satisfies weak stability with parameters $\epsilon_1 = \Omega(1),L= \Omega(1)$, and $U= O(1)$.

\item[(ii)]  The gradient of the loss function satisfies $\|\nabla \cL _ \gamma(\beta^*)\|_2 \lesssim \gamma\left( \sqrt{\frac{p \log p}{n}} +  \epsilon^{1 - 1/k} +   \sqrt{\frac{\log(1 / \tau)}{n}}\right)$.

\item[(iii)] For  $r \gtrsim \frac{\epsilon_1\gamma}{\sqrt{U}}$, $\gamma \gtrsim  \frac{\sigma}{\sqrt{\epsilon_1}}$, and $\frac{\log(1/ \tau)}{n} \lesssim \epsilon_1 $,
the function $\cL_\gamma$ is $L$-strongly convex in a ball of radius $r$ around $\beta^*$.
\end{itemize}
 \end{lemma}
 
Note that we can then  follow the proof of Theorem~\ref{ThmDetHuberReg} exactly, where we replace Lemmas~\ref{LemmaGradNorm} and~\ref{LemmaHessLower} with statements (ii) and (iii) of Lemma~\ref{LemHubAdv} and impose the condition that $\epsilon$ is less than a small enough constant.

\subsection{Proof of Lemma~\ref{LemHubAdv}}
\label{AppLemHubAdv}

\paragraph{\textbf{Proof of (i):}}

Recall that $T_1$ is a set of cardinality $n$, where we subtract pairs of points in the corrupted data set (and rescale by $\sqrt{2}$).
Analogously, we define the set $S_1$, where we perform pairwise subtraction on the uncorrupted data set $S$. It can be shown that $T_1$ is an (at most) $2 \epsilon$-corrupted version of set $S_1$, and $S_1$ is a set of $n$ i.i.d.\ data points from a linear model, where (i) the covariates are drawn from a centered isotropic distribution with
$k^{\text{th}}$ moment bounded by $c\sigma_{x,k}$; and (ii) the additive noise is zero-mean, symmetric, independent of the covariates, and of variance $\sigma^2$ (see Theorem 3.3 in Diakonikolas et al.~\cite{DiaKP20}).

By Theorem~\ref{ThmStabHighProb}, we know that with probability $1 - \tau$, there exists a set $S_2 \subseteq S_1$ such that $|S_2| \geq (1 -  \epsilon_1')n$ and $S_2$ is $(\epsilon_2	, \delta_2)$-stable, where $\epsilon_2 = C\epsilon_1'$ and $\delta_2 \lesssim \sqrt{\frac{p \log p}{n}} + \sigma_{x,k}  {\epsilon_1}	^{1 - 1/k} + \sigma_{x,4} \sqrt{\frac{\log(1/ \tau)}{n}}$.
Here, we take $\epsilon_1 = \frac{p\log p}{n} + 2\epsilon$ and define 
 $ \epsilon_1'	 = C\left( \epsilon_1 + \frac{\log(1 / \tau)}{n}\right)$, and note that $ \epsilon_1, \epsilon_1' = O(1)$ by our assumptions.
Recall that $T_2$ is the output of the filter algorithm on the set $T_1$ with parameter $\epsilon_1' \ge 2 \epsilon$.
Since $T_1$ is an (at most) $2\epsilon$-corrupted version of $S_1$, the existence of the stable set $S_2$, in conjunction with Theorem~\ref{ThmStability}, implies that with probability $1 - \tau$: (i) $T_2$ has cardinality at least $(1 - c_2  \epsilon_1'	)n$, and (ii) $T_2$ is $(\epsilon_3,\delta_3)$-stable,
where $\epsilon_3 = c_2 \epsilon_2	$ and $\delta_3 = c_4\delta_2$.

Moreover, by Proposition~\ref{PropStabSimplified}, we know that for $\epsilon_5$ a small enough constant, with probability at least $1 - O(\exp(- \Omega(n \epsilon_5) ))$,
every $S_3 \subseteq S_1$ with cardinality at least $(1 - \epsilon_5)n$ satisfies the inequality $\lambda_{\min}\left(\frac{1}{n} \sum_{(x,y) \in S_3} xx^T\right) \geq 0.8$.
Since the amount of corruption is sufficiently small, we will be able to translate this guarantee to the filtered set $T_2$.

We now choose $\epsilon_5 \le 1$ to be a small enough constant and choose $\epsilon_1'$ sufficiently small (note that the latter is possible for a small enough choice of $\epsilon$ and large enough choice of $n$),
so that the following are satisfied simultaneously:
\begin{enumerate}
	\item Both $\frac{\delta_2^2}{\epsilon_2} = O(1)$ and $\frac{\delta_3^2}{\epsilon_3} = O(1)$: note that
\begin{align*}
\frac{\delta_2^2}{ \epsilon_2} \lesssim \frac{ \frac{p \log p }{n} + \sigma_{x,k}^2 \epsilon_1^{2 - 2/k} + \sigma_{x,4}^2 \frac{\log(1 / \tau)}{n} }{ \epsilon_1 + \frac{\log(1/ \tau)}{n}} \lesssim 1.
\end{align*}

\item The cardinality of $S_2$ satisfies $|S_2| \geq (1 - \epsilon_1') n \geq  \left(1 - \frac{\epsilon_5}{20}\right)n \geq \frac{n}{2}$.
	\item The cardinality of $T_2$ satisfies $|T_2| \geq (1 - c_2 \epsilon_1 ')n \geq \left(1 - \frac{\epsilon_5}{20}\right)n \geq \frac{n}{2}$.

	\item The inequality $4 \epsilon < 4\epsilon_1' \leq \frac{\epsilon_5}{10}$ holds.
\end{enumerate}

We now show that the covariates in $T_2$ satisfy weak stability with $\epsilon_6 = \frac{\epsilon_5}{3} = \Omega(1)$,  $L = \Omega(1) $, and $U = O(1)$.
Suppose $T_2' \subseteq T_2$ is such that $|T_2'| \geq (1 - \epsilon_6) |T_2| $. Then
\begin{align*}
\frac{1}{|T_2|}\lambda_{\min}\left(\sum_{(x,y) \in T_2'}xx^T\right)  \le \frac{1}{|T_2|}\lambda_{\min}\left(\sum_{(x,y) \in T_2}xx^T\right) \leq 1 + \frac{\delta_3^2}{\epsilon_3} = O(1),
\end{align*}
using the $(\epsilon_3, \delta_3)$-stability of $T_2$, giving the upper bound $U = O(1)$. To obtain the lower bound, note that
\begin{align*}
|T_2' \cap S_1| &\geq |T_2'| - |T_2 \triangle S_1| \\
&\geq |T_2|\left(1 -  \frac{\epsilon_5}{3}\right)  - 2 \epsilon n \\
&\geq n\left(1 - \frac{\epsilon_5}{3}\right)\left(1 - \frac{\epsilon_5}{20}\right) - \frac{ \epsilon_5n}{20} \geq (1 - \epsilon_5)n.
\end{align*}
Therefore, $T_2'\cap S_1$ is a subset of $S_1$ with cardinality at least $(1 - \epsilon_5)n$,  %
and we conclude that
\begin{align*}
\frac{1}{|T_2|}\lambda_{\min}\left(\sum_{(x,y) \in T_2'}xx^T\right) \geq \frac{1}{n}\lambda_{\min}\left(\sum_{(x,y) \in T_2'\cap S_1}xx^T\right) \geq 0.8.
\end{align*}
This gives the desired lower bound $L = \Omega(1)$.

\paragraph{\textbf{Proof of (ii):}}

Using the same strategy as in previous step, we can show that weak stability also holds on $S_2$ with parameters $\epsilon_6$,  $L = \Omega(1) $, and $U = O(1)$.
We will use this fact to prove concentration results analogous to Lemmas~\ref{LemmaGradNorm} and~\ref{LemmaHessLower}.

In fact, in the proof of Lemma~\ref{LemmaGradNorm}, the only property of the covariates that we leveraged was the fact that they satisfy weak stability with $U = O(1)$. Thus, we can analogously argue that
\begin{align}
\label{EqAdvGradNorm}
 \left\|\frac{1}{|S_2|}\sum_{(x,y) \in S_2} \nabla \ell_\gamma(y - x^T \beta^*) \right\|_2 \lesssim \gamma\left(\sqrt{\frac{p}{n}} + \sqrt{\frac{\log(1 / \tau)}{n}}\right),
 \end{align}
with probability at least $1-\tau$.

We will now translate this result back to $T_2$ using the fact that both $S_2$ and $T_2$ are stable. Let $\cL_\gamma$ denote the Huber loss function with parameter $\gamma$ applied to the set $T_2$:
\begin{equation*}
\cL_\gamma(\beta) = \frac{1}{|T_2|} \sum_{(x,y) \in T_2} \ell_\gamma(y - x^T \beta).
\end{equation*}

Using the triangle inequality together with the bound~\eqref{EqAdvGradNorm} and the notation $z = y - x^T \beta^*$, we then obtain
\begin{align*}
 \|\nabla \cL_\gamma( \beta^*) \|_2 &= 
 \left\|\frac{1}{|T_2|}\sum_{(x,y) \in T_2} x \psi_\gamma(z) \right\|_2 
 \\
 &\leq  \left\|\frac{1}{|T_2|}\sum_{(x,y) \in S_2} x \psi_\gamma(z) \right\|_2 + \left\|\frac{1}{|T_2|}\sum_{(x,y) \in S_2 \setminus T_2 } x \psi_\gamma(z) \right\|_2 +  \left\|\frac{1}{|T_2|}\sum_{(x,y) \in T_2 \setminus S_2 } x \psi_\gamma(z) \right\|_2 \\ 
 &\lesssim \left\|\frac{1}{|S_2|}\sum_{(x,y) \in S_2} x \psi_\gamma(z)\right\|_2 + \left\|\frac{1}{|S_2|}\sum_{(x,y) \in S_2 \setminus T_2 } x \psi_\gamma(z) \right\|_2 
 + \left\|\frac{1}{|T_2|}\sum_{(x,y) \in T_2 \setminus S_2} x \psi_\gamma(z) \right\|_2 \\
&\lesssim \gamma \left( \sqrt{\frac{p}{n}} + \sqrt{\frac{\log(1 / \tau)}{n}} +  \delta_2 + \delta_3\right),
\end{align*}
with probability at least $1-\tau$, where the last step uses Proposition~\ref{PropStabL1Mean} and the stability of $S_2$ and $T_2$.
Using the bounds on $\delta_2$ and $\delta_3$ completes the proof.

\paragraph{\textbf{Proof of (iii):}}

We have shown that with probability at least $1-2\tau$, the sets $S_2$ and $T_2$ both satisfy weak stability with $\epsilon_6$, $L = \Omega(1)$, and $U = O(1)$; in addition, statements (1)--(4) hold in the proof of part (i) above. We denote this high-probability event by $\cE$, and show that under the additional assumptions, the desired strong convexity statement holds on the event $\cE$.

By the same argument used in the proof of Lemma~\ref{LemmaHessLower}, we know that on event $\cE$, if $r, \gamma$, and $\tau$ satisfy the inequality
\begin{align*}
\frac{ r \sqrt{U}}{\gamma} + \frac{ \sigma^2	}{\gamma^2} + \frac{\log (1 / \tau)}{n} \lesssim \epsilon_6,
\end{align*}
then
\begin{align}
\sup_{\beta: \|\beta- \beta^*\|_2 \leq r} \frac{1}{|S_2|} \sum_{(x,y) \in S_2}  \1\left(|y - x^T \beta| \geq \gamma\right) \leq \frac{\epsilon_6}{10}.
\label{EqAdvHuberCvx}
\end{align}
Crucially, we use the fact that conditioned on the event $\cE$ (which is entirely defined in terms of the covariates), the noise random variables $\{z_i = y_i - x_i^T \beta^*: (x_i, y_i) \in S_2\}$ remain i.i.d.

Now let $W := \sup_{\beta: \|\beta- \beta^*\|_2 \leq r} \frac{1}{|T_2|} \sum_{(x,y) \in T_2} \sum_{i=1}^n \1\left(|y - x^T \beta| \geq \gamma\right)$. 
Note that
\begin{align}
\label{EqnWbound}
W &\leq \frac{|T_2 \setminus S_2 |}{|T_2|} +   \sup_{\beta: \|\beta- \beta^*\|_2 \leq r} \frac{1}{|T_2|} \sum_{(x,y) \in S_2}  \1\left(|y - x^T \beta| \geq \gamma\right).
\end{align}
On the event $\cE$, we can bound the first term by
\begin{align*}
	\frac{|T_2 \setminus S_2|}{|T_2|}
	  \leq \frac{|T_1 \setminus S_2|}{n/2} \le \frac{2}{n} \left( |T_1 \setminus S_1| + |S_1 \setminus S_2| \right) \leq \frac{2}{n} \left( 2\epsilon n + \frac{\epsilon_6 n}{20} \right) \leq \frac{2 \epsilon_6}{5},
	\end{align*}
where the third inequality uses the fact that $|S_2| \geq (1 - \epsilon_5/20)n$, and the last inequality uses the bound $4 \epsilon \leq \frac{\epsilon_5}{10} = \frac{3\epsilon_6}{10}$.
The second term of inequality~\eqref{EqnWbound} can be bounded by
\begin{align*}
\sup_{\beta: \|\beta- \beta^*\|_2\leq r} \frac{1}{|T_2|} \sum_{(x,y) \in S_2}  \1\left(|y - x^T \beta| \geq \gamma\right) & = \frac{|S_2|}{|T_2|} \cdot \sup_{\beta: \|\beta- \beta^*\|_2\leq r} \frac{1}{|S_2|} \sum_{(x,y) \in S_2}  \1\left(|y - x^T \beta| \geq \gamma\right) \\
& \le \frac{n}{n/2} \cdot \frac{\epsilon_6}{10} = \frac{\epsilon_6}{5},
\end{align*}
using inequality~\eqref{EqAdvHuberCvx}. Thus,
\begin{align*}
W &\leq \frac{2\epsilon_6}{5} + \frac{\epsilon_6}{5} < \epsilon_6.
\end{align*}

Now define the matrix
\begin{align*}
H_n(\beta) := \frac{1}{|T_2|} \sum_{(x,y) \in T_2} xx^T \1\left( |y - x^T \beta| < \gamma\right).
\end{align*}
It follows that the strong convexity parameter of $\cL_ \gamma(\beta)$ is at least $\lambda_{\min}(H_n)$. Using the fact that $T_2$ satisfies weak stability and $W \le \epsilon_6$, we conclude that on the event $\cE$, we have $\lambda_{\min}(H_n(\beta)) \ge L$ for any $\beta$ such that $\|\beta- \beta^*\|_2 \leq r$, as wanted.

\subsection{Proof of Theorem~\ref{ThmAdvHubRegUnknownCov}}
\label{AppHuberUnkCov}

We will show that conditions analogous to the ones stated in Lemma~\ref{LemHubAdv} hold in this setting. As the proof is very similar to the proof in Section~\ref{AppLemHubAdv}, we only highlight several arguments which need to be adapted. We use the same notation defined in the previous section.

\paragraph{Condition (i):} Since the distribution of the covariates has a bounded covariance,
Theorem~\ref{ThmStabHighProbCovariance} implies that, with probability at least $1- \tau$, the set $S_2$ is $(\epsilon_2, \delta_2)$-stable, where $\delta_2 \lesssim \sqrt{\frac{p \log p}{n}} + \sqrt{\epsilon} + \sqrt{\frac{\log(1/\tau)}{n}}$. Recall that we needed $\frac{\delta_2^2}{\epsilon_2} = O(1)$. This is still satisfied, since $n \gtrsim p\log p$ and $\epsilon + \frac{\log(1/\tau)}{n} < c$, for a sufficiently small positive constant $c$.

It remains to establish ($\epsilon,L,U)$-weak stability of $T_2$ with $\epsilon= \Omega(1)$, $L= \Omega(1)$, and $U = O(1)$.
Similar to the proof of Lemma~\ref{LemHubAdv}, the lower bounds on $\epsilon$ and $L$ follow from the properties of $S_2$ which hold by the small ball property of the covariates, as shown in Proposition~\ref{PropStabSimplified}.

\paragraph{Condition (ii):} As shown in the proof of Lemma~\ref{LemHubAdv}, the norm of the gradient is bounded as $\|\nabla \cL_\gamma(\beta^*)\|_2 \lesssim \gamma\left(\sqrt{\frac{p}{n}} + \sqrt{\frac{\log(1/\tau)}{n}}\right) + \delta_2 + \delta_3$.
Since $\delta_3 = O(\delta_2)$, the bound on $\delta_2$ established in the previous paragraph suffices. 

\paragraph{Condition (iii):} This is exactly same as before, because we only used weak stability of the sets $S_2$ and $T_2$ to show this result.

\section{Least trimmed squares}
\label{AppLTS}

In this appendix, we provide additional proof details for the results in Section~\ref{SecLTS}.

\subsection{Proof of Lemma~\ref{LemAltMin}}
\label{AppLTSBhatia}

In this appendix, we reproduce the proof of the convergence guarantee for alternating minimization from Bhatia et al.~\cite{BhaJKK17}.

We begin by introducing some additional notation: For a vector $a \in \R^n$ and a set $S \subseteq [n]$, we will use $a_S$ to denote the vector $q \in \R^n$ such that (i) for $i \in S$, $q_i = v_i$; and (ii) for $i \not\in S$, $q_i = 0$. 
Similarly, for a matrix $A \in \R^{n \times p}$ and a set $S \subseteq [n]$, we will use $A_{S}$ to denote the matrix $Q \in \R^{n \times p}$ such that (i) for $i \in S$, the $i^{\text{th}}$ row of $Q$ is the same as the $i^{\text{th}}$ row of $A$; and (ii) for $i \not \in S$, all entries in the $i^{\text{th}}$ row of $Q$ are $0$.

\begin{lemma}
\label{ClaimLTSHT}
Suppose $a \in \R^n$. Let $b = \HT_{r}(a)$, let $S_1 = \text{supp}(b)$, and let $S \subseteq [n] $ be such that $S_1 \subseteq S$.
Then for any $r$-sparse vector $c$, we have $\|b - a_S\|_2 \leq \|c - a_S \|_2$.
\end{lemma}

\begin{proof}
Without loss of generality, let $a$ be such that $|a_1|\geq |a_2|\geq \cdots \ge |a_n|$.
Then $S_1 = [r]$.
Note that for any vector $c$, we have
\begin{align*}
\|c - a_S \|_2^2 &\geq \|c_S - a_S \|_2^2 = \sum_{i\in S}(c_i - a_i)^2 .
\end{align*}
It is not hard to see that the right-hand expression is minimized over $r$-sparse vectors when $c_i = a_i$ for $i \in S_1$ and $c_i = 0$ for $i \in S \setminus S_1$. This yields the expression $\|b - a_S\|_2^2$, completing the proof.
\end{proof}

Using the notation from Bhatia et al.~\cite{BhaJKK17}, let $X \in \R^{d \times n}$ denote the matrix of covariates, let $Y \in \R^n$ denote the vector of responses, and let $Z:= Y - X^T \beta^*$. (Note that the matrix $X$ is now defined to be the transpose of the design matrix that we denote by $X$ elsewhere in the paper.)
Recall that the model is $Y = X^T \beta^* + w + b^*$, where the idea is that $w$ has small entries and is nearly orthogonal to $X$, whereas $b^*$ is $m$-sparse.

Recall that $b^j$ was defined iteratively in the algorithm, and further define
\begin{align*}
\lambda^j & := (XX^T)^{-1}X(b^j - b^*), \\
g & := (I-P_X)w.
\end{align*}

Note that the update step can be written as follows:
\begin{align*}
b^{j+1} =  \HT_m\left(P_Xb^j + (I-P_X)(X^T\beta^* + w + b^*)\right) = \HT_m(b^* + X^T \lambda^j + g),
\end{align*}
using the fact that $X^T = P_XX^T$.
Denote $I_{j} := \text{supp}(b^j) \cup \text{supp}(b^*)$.
Applying Lemma~\ref{ClaimLTSHT} with $a = b^* + X^T\lambda^j + g$ and $S = I_{j+1}$, we have
  \begin{align*}
\|b^{j+1} - (b^* + X^T\lambda^j + g)_{I_{j+1}}\|_2 &\leq \|b^* - (b^* + X^T\lambda^j + g)_{I_{j+1}}\|_2 \\
&= \|b^* - b^* - X_{I_{j+1}}^T\lambda^j -g_{I_{j+1}}\|_2 = \|X_{I_{j+1}}^T\lambda^j + g_{I_{j+1}}\|_2,
\end{align*}
where we use the fact that $\text{supp}(b^*) \subseteq I_{j+1}$.
By the triangle inequality, we then have
\begin{align*}
\|b^{j+1} - b^*\|_2 &\leq \|b^{j+1} - b^*- X_{I_{j+1}}^T \lambda^j - g_{I_{j+1}} \|_2 + \|X_{I_{j+1}}^T \lambda^j + g_{I_{j+1}}\|_2 \\
&\leq 2 \|X_{I_{j+1}}^T \lambda^j + g_{I_{j+1}}\|_2 \leq 2 \|X_{I_{j+1}}^T \lambda^j\|_2 + 2\|g_{I_{j+1}}\|_2.
\end{align*}
We bound each of the latter two terms separately.
For the first term, we use the definition of $\lambda^j$ and the eigenvalue bounds on the covariates to write the following:
\begin{align*}
\|X_{I_{j+1}}^T \lambda^j\|_2 = \|X_{I_{j+1}}^T (XX^T)^{-1}X_{I_{j+1}} (b^j - b^*)\|_2 \leq  \frac{ \Lambda_{2m} }{\lambda_n} \|b^j - b^*\|_2.
\end{align*}

We now focus on the second term. By the triangle inequality, we have
\begin{align*}
\|g_{I_{j+1}}\|_2 &= \|W_{I_{j+1}} - X^T_{I_{j+1}}(XX^T)^{-1} X W\|_2 \\
&\leq \|W_{I_{j+1}}\|_2 + \|X^T_{I_{j+1}} (XX^T)^{-1}X W\|_2 \\
& \leq G + \frac{H}{\sqrt{\lambda_n}},
\end{align*}
using the fact that $W_{I_{j+1}}$ is at most $2m$-sparse and the bound
\begin{align*}
\|X^T_{I_{j+1}} (XX^T)^{-1}X W\|_2 \leq \frac{\sqrt{\Lambda_{2m}} H}{\lambda_n} \leq \frac{H}{\sqrt{\lambda_n}}.
\end{align*}
Combining the inequalities yields the bound
\begin{align}
\label{EqConvOfB}
\|b^{j+1} - b^*\|_2 \leq  \frac{2 \Lambda_{2m}}{ \lambda_n} \|b^j - b^*\|_2 + e_0 \le \frac{1}{2} \|b^j - b^*\|_2 + e_0,
\end{align}
where $e_0 := 2G + 2\frac{H}{\sqrt{\lambda_n}}$ and we have used the assumption that $\frac{2 \Lambda_{2m}}{ \lambda_n} \leq \frac{1}{2}$. Iterating the bound, we see that $\|b^j - b^* \| \leq 3e_0$ whenever $j \geq \log_2\left(\frac{\|b^0 - b^*\|_2}{e_0}\right)$.

To bound the final error between $\beta^j$ and $\beta^*$, we note that $\beta^j - \beta^* = (XX^T)^{-1}X(W + b^* - b^j)$. Using the definitions of $G$ and $H$, we have
\begin{align*}
\|\beta^j - \beta^*\|_2 &= \|(XX^T)^{-1}X(W + b^* - b^j)\|_2 \leq \frac{\|  X ( W +  (b^* - b^j))\|_2}{\lambda_n} \\
&\leq \frac{\|  X W\|_2 +    \|X(b^* - b^j)\|_2}{\lambda_n} 
\lesssim \left(\frac{H +   \sqrt{\Lambda_n}\left(G + \frac{H}{\sqrt{\lambda_n}}\right) }{\lambda_n} \right)\\
&\lesssim \frac{H + G\sqrt{\Lambda_n}}{\lambda_n},
\end{align*}
completing the proof.

\subsection{Proof of Theorem~\ref{ThmLTS}}
\label{AppLTSProb}

We will use the notation $z_i := y_i - x_i^T \beta^*$ and $z_i' := y_i' - (x_i')^T \beta^*$.

Let $m = C_1\left(p \log p + \epsilon n + \log\left(\frac{1}{\tau}\right)\right)$, for a large enough constant $C_1>6$ to be chosen later.
We will now apply Proposition~\ref{PropStabSimpleV2} with $\epsilon_1 =  \frac{C_2 m}{n}$, for a constant $C_2 \geq 1$ to be decided later.
In order for Proposition~\ref{PropStabSimpleV2} to be applicable, we need $\epsilon_1 < c^*$ and $n = \Omega\left(\frac{p \log p}{\epsilon_1}\right)$: For any $C_2$, the latter condition can be satisfied by choosing $C_1$ sufficiently large, and then the former condition can be satisfied by restricting $\epsilon, \frac{\log(1/ \tau)}{n}$, and $\frac{p \log p}{n}$ to be less than sufficiently small constants.
Let $T_1 \subseteq T$ be the set of data points corresponding to covariates which survive the filter algorithm, and let $n_1 := |T_1|$. Proposition~\ref{PropStabSimpleV2} guarantees that
with probability at least $1-2\tau$, we have
\begin{itemize}
	\item  $|T_1| \geq (1 - c_1\epsilon_1)n \geq \frac{n}{2}$,
	\item the covariates of the points in $T_1$ are $(\epsilon_2, \delta_2)$-stable, where
 \begin{equation*}
 \delta_2 = O\left(\sqrt{\frac{p \log p}{n}} + \sigma_{x,4} \epsilon_1^{3/4} + \sigma_4 \sqrt{\frac{\log(1/ \tau)}{n}}\right)
 \end{equation*}
 and $\epsilon_2 = \Theta (\epsilon_1)$, and
	\item  $\frac{\delta_2^2}{\epsilon_2} < 0.05$.
 \end{itemize}

We will now choose $C_2$ sufficiently large such that  $\epsilon_2 n = \Theta( \epsilon_1 n) = \Theta( C_2 m) > 4m$.
From here on, we will also assume that $\epsilon$, $\frac{\log(1/\tau)}{n}$, and $\frac{p \log p}{n}$ are bounded such that $4m \leq n$.

We now show that the SSC and SSS parameters of the covariates in $T_1$ are well-behaved, so that Lemma~\ref{LemAltMin} applies. We will apply the lemma to the model
\begin{equation}
\label{EqnLinCorrupt}
y_i' = (x_i')^T \beta^* + w_i + b_i^*, \qquad 1 \le i \le n_1,
\end{equation}
where for a set $T_2 \subseteq T_1$ to be defined later, we define the vector $w \in \real^{n_1}$ according to
\begin{equation*}
w_i :=
\begin{cases}
z_i, & \text{if } (x_i, y_i) \in T_2, \\
0, & \text{otherwise},
\end{cases}
\end{equation*}
and then simply define $b^* := y_i' - (x_i')^T \beta^* - w$.
Let the SSC and SSS parameters of $T_1$ be denoted by $\{\lambda_k\}$ and  $\{\Lambda_k\}$, respectively.
Note that
\begin{equation*}
\Lambda_{2m} \leq \Lambda_{ \epsilon_2 n/ 2} \leq \Lambda_{\epsilon_2 n_1} \leq \frac{3 n_1 \delta_2^2}{\epsilon_2} \leq 0.15n_1,
\end{equation*}
where we have used Proposition~\ref{PropStabSqError} in the third inequality.
By the $(\epsilon_2,\delta_2)$-stability of $T_1$, we have $\lambda_{n_1} \geq n_1\left(1 - \frac{\delta_2^2}{\epsilon_2}\right) \geq 0.9n_1$.
Therefore, $\frac{\Lambda_{2m}}{\lambda_{n_1}} \leq \frac{1}{4}$. Since
\begin{equation*}
\Lambda_{n_1} \le n_1\left(1+\frac{\delta_2^2}{\epsilon_2}\right) \le 1.05n_1,
\end{equation*}
we also have $\Lambda_{n_1} = O(\lambda_{n_1})$. Thus, the eigenvalue conditions of Lemma~\ref{LemAltMin} are indeed satisfied.

We now turn to the definition of $T_2$ and show that with this definition, $b^*$ is $m$-sparse. Let $S_2 \subseteq S$ be the set of $n - \frac{m}{4}$ uncontaminated data points with the smallest values of $|z_i|$. Let $F$ be the cumulative distribution function of $|z_i|$ and let $F^{-1}$ be its generalized inverse, i.e., $F^{-1}(p) = \inf_{t} \P(|z| \leq t) \geq p $.
Note that by a Chernoff bound, we have
\begin{align}
\label{EqnResid}
\left|\left\{i \in [n]: |z_i| > F^{-1}\left(1 - \frac{m}{8n}\right)\right\}\right| \leq \frac{m}{4},
\end{align}
with probability at least $1 - \exp(- \Omega(m))$.
Let $S_2' := S_2 \cap T$ denote the corresponding set of data points that are preserved after corruption.

Next, let $q_i := x_iz_i$, for $1 \le i \le n$, and note that the $q_i$'s are i.i.d.\  random variables with mean zero and covariance $\sigma^2I$. Applying Theorem~\ref{ThmStabHighProb} with $\epsilon_3 =  \frac{m}{3n}$ on the set $S' := \{q_1,\dots,q_n\}$, we see that, with probability except $O(\exp(- \Omega(m)))$,  there exists a set $S_3 \subseteq S'$ such that (i) $|S_3| \geq (1 - \epsilon_3)n$, and (ii) $S_3$ is $(C_4\epsilon_3, \delta_3)$-stable with respect to $\sigma^2$, where $C_4 = c_1 C_2 + 1$  and $\delta_3 = O\left(\sqrt{\frac{p \log p}{n}} + \sigma \sqrt{\frac{m}{n}}\right)$. Let $S_3' := \{(x_i, y_i): x_i z_i \in S_3\} \cap T$ denote the corresponding set of $(x,y)$ pairs that are also preserved after corruption.

Finally, we define the set
\begin{equation*}
T_2 := T_1\cap S_2' \cap S_3'.
\end{equation*}
Note that
\begin{equation*}
|T_1\setminus T_2| \leq \left(|S \setminus S_2| + |T \setminus S| \right) + \left(|S' \setminus S_3| + |T \setminus S| \right) \leq 2\epsilon n + \frac{m}{4} +  \frac{m}{3} \leq m,
\end{equation*}
where we use the fact that $m \geq 6 \epsilon n$ (since $C_1 > 3$). Thus, the vector $b^* \in \real^{n_1}$ is indeed $m$-sparse, and Lemma~\ref{LemAltMin} implies an error bound of order $\frac{G}{\sqrt{n_1}}  + \frac{H}{n_1}  = O\left(\frac{G}{\sqrt{n}} + \frac{H}{n}\right)$.
It remains to control the parameters $G$ and $H$.

Recall that with high probability, inequality~\eqref{EqnResid} holds, in which case the nonzero entries of $w_i$ have magnitude at most $F^{-1}\left(1 - \frac{m}{8n}\right)$.
Thus, we have
\begin{equation*}
 \sup_{S': |S'| \leq 2m} \sqrt{\sum_{i \in S'} w_i^2} \le \sqrt{2m}F^{-1}\left(1 - \frac{m}{8n}\right) \lesssim \sqrt{m} \left(\frac{m}{n}\right)^{- 1/k'},
\end{equation*}
where the second inequality follows from the $(k')^{\text{th}}$ moment condition on $z_i$. 
Thus, we may take $G = O\left(\sqrt{m} \left(\frac{m}{n}\right)^{- 1/k'}\right)$.

Turning to $H$, note that with high probability, we have
\begin{align*}
\frac{|T_2|}{|S_3|} \geq  \frac{|T_1| - m}{n} \geq 1 - c_1 \epsilon_1 - \frac{m}{n} = 1 - ( c_1 C_2 + 1)\frac{m}{n} = 1 - C_4\epsilon_3. 
\end{align*}
Hence, the $(C_4\epsilon_3, \delta_3)$-stability of $S_3$ implies that 
\begin{align*}
 \left\|\sum_{i = 1}^{n_1} x_i'w_i \right\|_2 = \left\|\sum_{(x,y) \in T_2} x_i z_i \right\|_2 \leq |T_2| \sigma \delta_3 \leq n \sigma \delta_3,
\end{align*}
where we employ the notation used in the proof of Theorem~\ref{ThmAdvHuberReg} in the second expression. Therefore, $H \leq n\sigma \delta_3$.

Altogether, we arrive at the error bound
\begin{align*}
\|\widehat{\beta} - \beta^*\|_2 \leq \frac{G}{\sqrt{n}} + \frac{H}{n} \lesssim \sigma \left( \delta_3  + \sigma_{z,k'} \left(\frac{m}{n}\right)^{\frac{1}{2} - \frac{1}{k'}}\right) \lesssim \sigma \sigma_{z,k'}\left(  \frac{p \log p  }{n }  + \epsilon +  \frac{\log(1 / \tau)}{n} \right)^{\frac{1}{2} - \frac{1}{k'}},
\end{align*}
where we use the value of $m$ and the fact that $ \delta_3  \lesssim \sigma_{z,k'}\left(\frac{m}{n}\right)^{1/2 - 1/k'}$.
Moreover, the probability of error is at most $O(\exp(- \Omega(m)))$.
Lastly, we choose $C_1$ large enough so that the error probability is at most $O(\tau)$.

Finally, we bound the number of iterations of the alternating minimization algorithm required to guarantee the desired accuracy bound. In light of Remark~\ref{RemAltMin}, it suffices to obtain a high-probability upper bound on $\|b^*\|_2$ that can be computed from the data.
Recall the notation $S = (X,y)$ and $T = (X', y')$ for the i.i.d.\ and corrupted data sets, respectively, and recall that $T_1 \subseteq T$ denotes the filtered data set. Abusing notation slightly, we write the model~\eqref{EqnLinCorrupt} in matrix/vector form as $y'_{T_1} = X'_{T_1} \beta^* + w_{T_1} + b^*_{T_1}$. We claim that
\begin{equation}
\label{EqnBstarBd}
\|b^*_{T_1}\|_2 = O\left(\|y'\|_2(1 + \|X'\|_2)\right),
\end{equation}
with probability at least $1 - O(\exp(-\Omega(n)))$.

Recall that by construction, either $b^*_i = 0$ or $w_i = 0$ for each $i$ in the model~\eqref{EqnLinCorrupt}. Thus, by the triangle inequality, we have
\begin{align*}
\|b^*_{T_1}\|_2 \leq \|y'_{T_1}\|_2 + \|X'_{T_1} \beta^*\|_2 \leq \|y'\|_2 + \|X' \beta^*\|_2 \le \|y'\|_2 + \|X'\|_2 \|\beta^*\|_2.
\end{align*}
We now use concentration properties of the i.i.d.\ points in $S$ to obtain a data-driven upper bound on $\|\beta^*\|_2$. Note that $\E(y_i^2) = \|\beta^*\|_2^2 + \sigma^2$.
Furthermore, by Lemma~\ref{LemConvex} and the convexity of the absolute value function, we have
\begin{align*}
\E |y_i| = \E |x_i^T \beta^* + z_i| \geq  \max \{\E | x_i^T \beta^* |, \E |z_i|\}.
\end{align*}
Furthermore, we can lower-bound both $\E | x_i^T \beta^* |$ and $\E |z_i|$ using Proposition~\ref{PropHolder} and Assumption~\ref{AsCov}:
\begin{align*}
\E |x_i^T \beta^*| & \geq \frac{\|\beta^*\|_2}{\sigma_{x,4}^2}, \\
\E |z_i| & \geq \frac{\sigma}{\sigma_{z,4}^2},
\end{align*}
using the assumption that $(\E |z_i|^{4})^{1/4} \leq \sigma_{z,4}\sigma$ by $(4,2)$-hypercontractivity.

By the Paley-Zygmund inequality (e.g., see Exercise 2.4 of Boucheron et al.~\cite{BouLM13}), we have
\begin{align*}
\P\left( |y_i| \geq \frac{\E |y_i|}{2} \right) & \geq \frac{(\E |y_i|)^2}{4\E y_i^2} \\
& \geq \frac{\max\left\{\frac{ \|\beta^*\|_2^2}{\sigma_{x,4}^4}, \frac{\sigma^2}{\sigma_{z,4}^4}\right\}}{4(\|\beta^*\|_2^2 + \sigma^2)} \\
& \geq \frac{1}{\max\{\sigma_{x,4}^4, \sigma_{z,4}^4\}} \cdot \frac{\frac{1}{2}\left(\|\beta^*\|_2^2 + \sigma^2\right)}{4(\|\beta^*\|_2^2 + \sigma^2)} \\
& = \frac{1}{8 \max(\sigma_{x,4}^4, \sigma_{z,4}^4)}.
\end{align*}
Thus,
\begin{equation*}
\P\left( |y_i| \geq \frac{\|\beta^*\|_{2}}{2\sigma_{x,4}^2}\right) \geq \frac{1}{8 \max\{\sigma_{x,4}^4, \sigma_{z,4}^4\}}.
\end{equation*}
Let $\gamma = 16\max\{\sigma_{x,4}^4, \sigma_{z,4}^4\}$, which is assumed to be $O(1)$.
Let $W$ be the  $\lceil \left(1 - 1/ \gamma\right)n	\rceil^{\text{th}}$ largest $|y_i|$.
Then by a Chernoff bound, we have
\begin{align*}
\P \left( W  <  \frac{\|\beta^*\|_{2}}{2\sigma_{x,4}^2}\right) \leq \exp\left( -  \Omega\left(\frac{n}{\alpha}\right)\right).
\end{align*}
Finally, for $\epsilon < \frac{1}{2 \gamma}$, we have $\max_i | y'_i| \geq W $.
Therefore, with high probability,
\begin{equation*}
\|\beta^*\|_2 \leq 2 \sigma_{x,4}^2 \max_i |y_i'| = O(\|y'\|_2).
\end{equation*}
This completes the proof.

\section{Least absolute deviation }
\label{AppLAD}

In this appendix, we provide additional proof details for the results in Section~\ref{SecLAD}.

\subsection{Auxiliary results}

\begin{proposition}
\label{PropHolder}
Suppose $Z$ satisfies $\E Z^2 = 1$ and $ \E Z^4 < \infty$.
Then $ \E |Z| > 1/ \sqrt{\E |Z|^4}$.
\end{proposition}
\begin{proof}
We apply H\"{o}lder's inequality, which states that
\begin{align*}
\E |XY| \leq (\E |X|^p)^{1/p} (\E |Y|^q )^{1/q},
\end{align*}
for $p \in (1 , \infty)$ and $q =  \frac{p}{p-1}$. Taking $X= Z^{4/3}$, $ Y =Z^{2/3}$, and $p = 3$, we have
\begin{align*}
1 = \E Z^2 \leq (\E (|Z|^{4/3})^3)^{1/3} (\E (|Z|^{2/3})^{3/2})^{2/3} = (\E |Z|^{4})^{1/3} (\E |Z|)^{2/3}.
\end{align*}
\end{proof}

\begin{lemma}
\label{LemTrimmedSumL1}
Let $X_1, \dots, X_n$ be i.i.d.\ nonnegative random variables and let $\epsilon \in (0,1)$.
Then with probability $1 - 2\exp(- c n \epsilon)$, the trimmed sum satisfies
\begin{equation*}
\sum_{i=1}^{(1 - \epsilon)n} X_{(i)} = O\left(\frac{n \E X_i}{\epsilon}\right),
\end{equation*}
where $\{X_{(i)}\}_{i=1}^n$ are order statistics.
\end{lemma}

\begin{proof}
Let $F$ be the cdf of the $X_i$'s, and let $F^{-1}$ be its inverse, so $F^{-1}(1 - \epsilon) = \inf\{t: \P (X_i > t) \leq \epsilon   \}$ for $\epsilon \in [0,1]$.
Let $a := F^{-1}\left(1 - \frac{\epsilon}{3}\right)$ and define $Z_i = \min(X_i,a)$. Note that $\sum_{i=1}^n Z_i \le an$.

Now let $Y_i = \1\{X_i > a\}$ and define the event
\begin{align*}
\cE := \left\{ \sum_{i=1}^n Y_i < \epsilon n  \right\}.
\end{align*}
We have
\begin{equation*}
\E Y_i = \P (X_i > a) = \P\left(X_i > F^{-1}\left(1 - \frac{\epsilon}{3}\right)\right) \leq \frac{\epsilon}{3}.
\end{equation*}
Applying a Chernoff bound, we therefore have
\begin{align*}
\sum_{i=1}^n Y_i \leq \frac{2 \epsilon n}{3},
\end{align*}
with probability at least $1 - \exp(- c n \epsilon)$, implying that $\P(\cE) \ge 1 - \exp(-c_2n\epsilon)$.

Finally, note that on the event $\cE$, we have
\begin{align*}
\sum_{i=1}^{(1 - \epsilon)n} X_{(i)} \leq \sum_{i=1}^n Z_i \le an.
\end{align*}
Applying Markov's inequality, we have $\P\left( X_i \ge  \frac{4\E X_i}{\epsilon}\right) \leq \frac{\epsilon}{4} < \frac{\epsilon}{3}$.
Therefore, $a \leq \frac{4 \E X_i}{\epsilon}$, completing the proof.
\end{proof}

\begin{lemma}
\label{LemLowL1Stab}
 Suppose the covariates $x_1,\dots, x_n$ are sampled i.i.d.\ from a distribution satisfying Assumption~\ref{AsCov}.
With probability $1 - 2\exp(- c n \epsilon)$, we have that for any unit vector $v$ and any $S\subseteq [n]$ with $|S| \geq (1 - \epsilon)n$, the following holds:
\begin{align*}
\frac{1}{n}\sum_{i \in S}  |x_i^Tv|  \geq  \frac{1}{ \sigma_{x,4}^2} - O\left(\sqrt{\epsilon} + \sqrt{ \frac{p}{n}}\right).
\end{align*}
\end{lemma}
\begin{proof}
Let $Q$ be the threshold $C\left(\sqrt{\frac{1}{\epsilon}} + \frac{1}{\epsilon} \sqrt{\frac{p}{n}}\right)$ from Lemma~\ref{LemTruncLin}. Let $\cE$ denote the event from Lemma~\ref{LemTruncLin}, stating that for any unit vector $v$, we have $\left|\{i:|x_i^Tv|\geq Q\}\right| \leq \epsilon n$.
By the lemma, we know that $\P(\cE) \ge 1 - \exp(-c n \epsilon)$.

We will now assume that the event $\cE$ holds and incur an additional failure probability of $\exp(-cn \epsilon)$ by a union bound.
Define the function $f: \R_+ \to \R_+$, as follows:
\begin{align*}
f(x) = \begin{cases} x, & \text{ if } x \in [0,Q],\\
Q, & \text{ otherwise,}\end{cases},
\end{align*}
and let $g(x) = -f(x)$. For any $v \in \cS^{p-1}$, on the event $\cE$, we have the following bound:
\begin{align*}
\min_{S: |S| \geq (1 - \epsilon)n} \sum_{i \in S} |x_i^Tv| &\geq  \sum_{i=1}^n f(|x_i^Tv|) - \epsilon Q n \\
&= - \left(\sum_{i=1}^n g(|x_i^Tv|) - \E g(|x_i^Tv|)\right) + n\E f(|x_i^Tv|) - \epsilon Q n.
\end{align*}
Taking an infimum over $v$, we then have
\begin{equation}
\label{EqnTeddy}
\inf_{v \in \cS^{p-1}} \min_{S: |S| \geq (1 - \epsilon)n} \sum_{i \in S} |x_i^Tv|
\ge   - \epsilon Q n  - \sup_{v \in \cS^{p-1}} \left(\sum_{i=1}^n g(|x_i^Tv|) - \E g(|x_i^Tv|)\right) + n\left(\inf_{v \in \cS^{p-1}} \E f(|x_i^Tv|) \right).
\end{equation}
Now define the random variable
\begin{align*}
N := \sup_{v \in \cS^{p-1}} \sum_{i=1}^n g(|x_i^Tv|) - \E g(|x_i^Tv|).
\end{align*}
We first bound the expectation of $N$ using symmetrization and contraction of Rademacher averages~\cite{LedTal91,BouLM13}:
\begin{align*}
\E N & \leq 2 \E \sup_{v \in \cS^{p-1}} \left|\sum_{i=1}^n \xi_i g(|x_i^Tv|) \right| \leq 4 \E \sup_{v \in \cS^{p-1}} \left|\sum_{i=1}^n \xi_i x_i^Tv\right| \\
& \le 4 \E \left(\left\|\sum_{i=1}^n \xi_i x_i\right\|_2 \sup_{v \in \cS^{p-1}} \|v\|_2\right) \le 4 \sqrt{\E\left(\left\|\sum_{i=1}^n \xi_i x_i\right\|_2^2\right)} \\
& = 4 \sqrt{\E\left(\sum_{i=1}^n x_i^T x_i\right)} = 4 \sqrt{\sum_{i=1}^n \E\left(\trace(x_i^T x_i)\right)} = 4 \sqrt{\sum_{i=1}^n \E\left(\trace (x_i x_i^T)\right)} \\
& = 4 \sqrt{\sum_{i=1}^n\trace\left(\E\left(x_i x_i^T\right)\right)} \\
& = 4 \sqrt{pn},
\end{align*}
where the $\xi_i$'s are i.i.d.\ Rademacher random variables. We now bound the following term (which is usually called the \textit{wimpy variance}~\cite{BouLM13}):
\begin{align*}
\sigma^2 := \sup_v n\Var( g(|x_i^Tv|)) \leq \sup_v n \E |x_i^Tv|^2 = n.
\end{align*}
Using Talagrand's inequality for bounded empirical processes (cf.\ Lemma~\ref{ThmTalagrand}), we therefore have
\begin{equation}
\label{EqnBear}
N = O( \sqrt{pn} +  \sqrt{n}\sqrt{n \epsilon} + Q n \epsilon) = O( \sqrt{pn} + n \sqrt{\epsilon} + n \sqrt{\epsilon} + \sqrt{pn} ) = O(\sqrt{pn} + n \sqrt{\epsilon}),
\end{equation}
with probability at least $ 1 - \exp( - c'n \epsilon)$.

Finally, note that for any $v \in \cS^{p-1}$, the Cauchy-Schwarz inequality gives 
\begin{align*}
\E\left| f(|x_i^Tv|) - |x_i^Tv|\right| & \le \E \left(|x_i^Tv| \1\{|x_i^Tv| > Q\}\right) \\
& \leq \sqrt{\E (x_i^Tv)^2} \sqrt{\P (|x_i^Tv| \geq Q)} \leq \sqrt{\frac{\E (x_i^Tv)^2}{Q^2}} \\
& =  O(\sqrt{\epsilon}),
\end{align*}
where the last two steps use Markov's inequality and the fact that $Q = \Omega(1/ \sqrt{\epsilon})$.
Thus,
\begin{equation}
\label{EqnDuck}
\E f(|x_i^T v|) \geq \E|x_i^Tv| - O(\sqrt{\epsilon}) \ge \frac{1}{\sigma_{x,4}^2} - O(\sqrt{\epsilon}),
\end{equation}
where the second inequality follows from Proposition~\ref{PropHolder}.

Combining inequalities~\eqref{EqnTeddy}, \eqref{EqnBear}, and~\eqref{EqnDuck}, we obtain the bound
\begin{align*}
\frac{1}{n} \inf_{v \in \cS^{p-1}} \min_{S: |S| \geq (1 - \epsilon)n} \sum_{i \in S} |x_i^Tv|
&\geq \inf_v \E f(|x_i^Tv|) - \epsilon Q - \frac{N}{n} \\
&\geq \frac{1}{\sigma_{x,4}^2} - O\left(\sqrt{\frac{p}{n}} + \sqrt{\epsilon}\right).
\end{align*}
This completes the proof.
\end{proof}

\subsection{Proof of Lemma~\ref{LemLADL1Stab}}
\label{AppLADL1Stab}

We follow the proof strategy from Koltchinskii and Mendelson~\cite{KM15} and Diakonikolas et al.~\cite{DiaKP20}.

Let $T_1 = \{(x_i', y_i')\}_{i=1}^{n_1}$ be the output of the filter algorithm with inputs $T$ and $\epsilon'$, where $\epsilon_1 < \epsilon'$.
By Proposition~\ref{PropStabSimpleV2},
with probability at least $1 - 2\exp( - n\epsilon' )$, the set $T_1$ is $(\epsilon_2, \delta_2)$-stable, where $\epsilon_2 = \Theta(\epsilon')$ and $\delta_2 = O\left(\sqrt{\frac{p \log p}{n}} + \sqrt{\epsilon'}\right)$, and $T_1$ has cardinality $n_1 \ge (1 - c_1 \epsilon')n$. Furthermore, we choose $\epsilon_1$ and $\epsilon'$ sufficiently small to guarantee that $n_1 \ge \frac{n}{2}$.
Therefore, for any $T' \subseteq T_1$ such that $|T'| \leq \epsilon_2 |T_1|$, Proposition~\ref{PropStabL1Error} states that for all unit vectors $v$,
\begin{align}
\label{EqnmBd}
\frac{1}{n_1} \sum_{x_i' \in T'} |v^Tx_i'| \leq 2 \delta_2.
\end{align}

Let $T_2 \subseteq T_1$ be a set such that $|T_2| \geq (1 - \epsilon_2)|T_1|$.

 Since $|T_1| = n_1 \geq \frac{n}{2}$, we have
\begin{align}
\label{EqnSetSize}
\nonumber
|T_2 \cap S| &= |S| - | S\setminus T| - |T\setminus T_1 | - | T_1 \setminus  T_2| \geq n -  \epsilon_1n - c_1 \epsilon ' n - \epsilon_2 n_1  \geq (1 - \epsilon_1 - c_1 \epsilon' -  \epsilon_2)n \\
&\geq (1 -  c_3 \epsilon')n,
\end{align}
where $c_3$ is a constant, using the facts that $\epsilon_1 < \epsilon'$ and $\epsilon_2 = \Theta(\epsilon')$.

Now suppose $n = \Omega (p \sigma_{x,4}^4)$ and $\epsilon_0 = O\left(\frac{1}{\sigma_{x,4}^4}\right)$. By Lemma~\ref{LemLowL1Stab}, we know that, with probability at least $1 - \exp(- \Omega(n\epsilon_0))$,
we have
\begin{align}
\frac{1}{n} \sum_{i \in S'} |x_i^Tv| \geq \frac{1}{2 \sigma_{x,4}^2},
\label{EqLowBoundL1}
\end{align}
for any $S' \subseteq [n]$ such that $|S'| \geq (1 - \epsilon_0) n$ and any $v \in \cS^{p-1}$. Hence, if $c_3 \epsilon' \le \epsilon_0$, inequalities~\eqref{EqnSetSize} and~\eqref{EqLowBoundL1} together imply that
\begin{align}
\label{EqnMBd}
\frac{1}{|T_1|} \sum_{x_i' \in T_2} |v^Tx_i'| \geq \frac{1}{n} \sum_{x_i \in T_2 \cap S} |v^Tx_i| \geq \frac{1}{2 \sigma_{x,4}^2}. 
\end{align}

From inequalities~\eqref{EqnmBd} and~\eqref{EqnMBd}, we conclude that $T_1$ satisfies $\left(\epsilon_2, m= 2 \delta_2, M = \frac{1}{2 \sigma_{x,4}^2}, \ell_1\right)$-stability with the desired probability. Note that if we choose $n = \Omega(p \log p)$ large enough
and $\epsilon'$ to be a sufficiently small constant, we can guarantee that $c_2\epsilon_2 \le \epsilon_0$ and $\delta_2$ is sufficiently small, so $2m \le M$.

\section{Postprocessing}

In this appendix, we provide additional proof details for the results in Section~\ref{SecPP}.
We will use the following result from Diakonikolas et al.~\cite{DiaKP20}, which gives a result corresponding to Theorem~\ref{ThmStabHighProb} when the distribution only has a finite variance:
\begin{theorem}(Diakonikolas et al.~\cite{DiaKP20})
\label{ThmStabHighProbCovariance}
Let $S$ be a set of $n$ i.i.d.\ points from a distribution in $\R^p$ with mean $\mu$ and covariance $\Sigma \preceq \sigma^2 I$ for some $\sigma\geq 0$.
Let $\epsilon$ and $\tau$ be such that $\epsilon' = C\left(\epsilon + \frac{\log(1/\tau)}{n}\right) = O(1)$, for a large enough constant $C$. 
Then with probability at least $1 - \tau$,  there exists a subset $S' \subseteq S$ such that $|S'| \geq (1 - \epsilon')|S|$ and $S'$ is $(C_1\epsilon',\delta)$-stable with respect to $\mu$ and $\sigma^2$, where $C_1 > 2$ is any large constant and $\delta = O\left( \sqrt{\frac{p\log p}{n}} + \sqrt{\epsilon}+ \sqrt{\frac{\log(1/ \tau)}{n}}\right)$, with prefactor depending on $C_1$.
\end{theorem}

\subsection{Proof of Theorem~\ref{PropRobMeanMain}}
\label{AppPP}

Our approach differs from the proof of Theorem~\ref{PropPost} in that the vectors in the set 
\begin{equation*}
S_1 = \left\{\widehat{\beta}_1+ (y_i - x_i^T \widehat{\beta}_1)x_i: (x_i,y_i) \in S\right\}
\end{equation*}
may no longer be i.i.d.\ when we condition on the initial estimator $\widehat{\beta}_1$. Thus, we cannot directly apply Theorem~\ref{ThmStabSubGaussian} to obtain an error bound. On the other hand, recall from Remark~\ref{RemStableMean} that if we can show the existence of a sufficiently large stable subset of the set $S_1$, Theorem~\ref{ThmStability} implies a corresponding error bound.

For any fixed $v \in \real^p$, define the random variables
\begin{align*}
W^v_i  := v + (y_i- x_i^Tv)x_i = v + x_ix_i^T(\beta^* - v) + z_ix_i, \qquad \forall 1 \le i \le n,
\end{align*}
and define the multiset $S_v: = \{W^v_1,\dots,W^v_n\}$. Note that each set $S_v$ consists of $n$ i.i.d.\ data points, so that stability properties can be obtained easily; the additional challenge is that we need to show the existence of a stable subset for all $v \in \real^p$ simultaneously, so that we can apply the result when $v = \widehat{\beta}_1$. To this end, we will use a covering argument. Let $r = \Theta(\sigma)$ be such that $\|\widehat{\beta}_1 - \beta^*\|_2 \leq r$, and define the set $ \cT: = \{v: \|\beta^*- v\|_2 \leq r\}$.
We now define $\cC_\eta \subseteq \cT$ to be an $\eta$-cover of $\cT$, i.e., for every $v \in \cT$, there exists $v' \in \cC_ \eta$ such that $\|v - v'\|_2 \leq \eta$. Note that for $\eta \leq r$, we can choose $\cC_\eta$ such that $\log(|\cC_\eta|) \leq  p \log\left(\frac{3r}{\eta}\right)$ (cf. Corollary 4.2.13 of Vershynin~\cite{Ver18}).
	
For any $v \in \cC_\eta$, we have $\E W^v_i = \beta^*$.
Let $\Delta = v - \beta^*$.
As in inequality~\eqref{EqnCovW} in the proof of Theorem~\ref{PropPost}, we can argue that $\|\text{Cov}(W_v)\|_2 \le C_0 \sigma^2$.

Applying Theorem~\ref{ThmStabHighProbCovariance} with parameters $\tau' = \tau \exp(-C_1 p\log (pn)) $ and $\epsilon' = \Theta\left(\epsilon + \frac{\log(1/ \tau')}{n}\right) = \Theta\left(\epsilon + \frac{\log(1/ \tau)}{n} + \frac{p \log (pn)}{n} \right)$, for a large constant $C_1 > 0$ to be defined later, we see that with probability at least $1 - \tau'$, there exists a set $S_v' \subseteq S_v$ such that $|S_v'| \geq (1 - \epsilon')n$ and $S_v'$ is $(C\epsilon', \delta)$-stable with respect to $\beta^*$  
 and $\sigma_*^2 := C_0 \sigma^2$, where $\delta := \Theta\left( \sqrt{\frac{p\log (pn)}{n}} + \sqrt{\epsilon} + \sqrt{\frac{\log(1/ \tau)}{n}}\right)$.

Suppose a stable set exists for every element of $C_\eta$ (we will bound the error probability later). Now consider an arbitrary $v' \in \real^p$, and let $v \in C_\eta$ be such that $\|v' - v\|_2 \leq \eta$. We know that there exists a set $S_v' \subseteq S_v$ which is $(C \epsilon', \delta)$-stable with respect to $\beta^*$ and $\sigma_*^2$; we will show how to obtain a stable set $S_{v'}' \subseteq S_{v'}$ using $S_{v}'$. Note that $S_v'$ corresponds to a set of indices which we define as $T_v \subseteq [n]$, so $S'_v = \{W_i^{v}\}_{i \in T_v}$.

Define the set
\begin{align*}
S_2 := \left\{(x_i, y_i): \|x_i\|_2 \le \sqrt{\frac{p}{\epsilon'}} \text{ and } |y_i - x_i^T \beta^*| \le \frac{\sigma}{\sqrt{\epsilon'}}\right\}.
\end{align*}
By a Chernoff bound, we can argue that with probability at least $1 - \exp(- cn \epsilon') = 1 - O(\tau)$, we have $|S_2| \geq (1 - 4\epsilon')n$. Indeed, define the indicator variables $E_i = 1\{(x_i, y_i) \in S_2\}$. Then
\begin{align*}
\E(E_i) & = \P\left(\|x_i\|_2 \le \sqrt{\frac{p}{\epsilon'}} \text{ and } |z_i| \le \frac{\sigma}{\sqrt{\epsilon'}}\right) \ge 1 - \P\left(\|x_i\|_2^2 \ge \frac{p}{\epsilon'}\right) - \P\left(z_i^2 \ge \frac{\sigma^2}{\epsilon'}\right) \\
& \ge 1 - \frac{\E(\|x_i\|_2^2)}{d/\epsilon'} - \frac{\E(z_i^2)}{\sigma^2/\epsilon'} = 1 - 2\epsilon',
\end{align*}
using Markov's inequality. Applying the multiplicative Chernoff bound in Lemma~\ref{ThmChernoff} to the random variables $(1-E_i)$, we then obtain
\begin{equation*}
\P\left(|S_2| \ge (1-4\epsilon')n\right) \ge \P\left(\frac{1}{n}\sum_{i=1}^n (1-E_i) \le 4\epsilon'\right) \ge 1- \exp(-cn\epsilon'),
\end{equation*}
as claimed. We also define the set of indices $T_0 \subseteq [n]$ such that $S_2 = \{(x_i, y_i)\}_{i \in T_0}$.

Now let $T_{v'} := T_v \cap T_0$ and consider the set $S'_{v'} := \{W_i^{v'}\}_{i \in T_{v'}}$, which we will show is stable with high probability. Note that $|T_{v'}'| \geq (1 - 5\epsilon')n$.
We have the following lemma, proved in Appendix~\ref{AppClaimPP}:

\begin{lemma}
\label{ClaimPPCont}
Suppose $S'_v$ is $(C\epsilon', \delta)$-stable with respect to $\beta^*$ and $\sigma_*^2$ such that $|S_v'| \ge (1-\epsilon')n$, and suppose $|S_2| \ge (1-4\epsilon')n$. Suppose $\|v - v'\|_2 \le \eta$ and $\eta = \frac{r \sqrt{\epsilon'}}{f(d/\epsilon')}$, where $f$ is an appropriately defined second-degree polynomial. Then
$S_{v'}'$ is $(C \epsilon'/2, \delta')$-stable with respect to $\beta^*$ and $\sigma_*^2$, where $\delta' = \Theta\left( \sqrt{\frac{p\log (pn)}{n}} + \sqrt{\epsilon} + \sqrt{\frac{\log(1/ \tau)}{n}}\right)$.
\end{lemma}

Finally, we use a union bound to control the failure probability. Combining the error probability for the Chernoff bound for $S_2$ with the error probabilities for the elements of $C_\eta$, we see that the overall probability of error is bounded by
\begin{align*}
\exp(-cn\epsilon') + \tau' |\cC_ \eta| & \le \exp(-cn\epsilon') + \tau'\exp\left( p \log\left(\frac{3r }{ \eta}\right)\right) \\
& =  \exp(-cn\epsilon') + \tau \exp\left(-C_1 p\log(pn) + p \log\left(\frac{3f(d/\epsilon')}{\sqrt{\epsilon'}}\right) \right) \\
& \le \exp(-cn\epsilon') + \tau \exp\left(-C_1 p \log(pn) + c_1 p \log n + \frac{p}{2} \log \left(\frac{1}{\epsilon'}\right)\right) \\
& \le \exp(-cn\epsilon') + \tau \exp\left(-C_1 p \log(pn) + c_1 p \log n + c_2 p\log n\right),
\end{align*}
using the choice of $\eta$ in Lemma~\ref{ClaimPPCont} and the fact that $\epsilon' = \Omega\left(\frac{p}{n}\right)$ in the last two inequalities. The final expression can be made smaller than $2\tau$ for a sufficiently large choice of $C_1$, completing the proof.

\subsection{Proof of Lemma~\ref{ClaimPPCont}}
\label{AppClaimPP}

Consider any set $T' \subseteq T_{v'}$ such that $|T'| \ge \left(1- \frac{C\epsilon'}{2}\right) |T'_{v'}|$, and define $\Delta := \beta^* - v$ and $\Delta' := \beta^* - v '$, so $\Delta' - \Delta = v - v'$. Using the triangle inequality, we write
\begin{align}
\label{EqnMilk}
\left\|\frac{1}{|T'|} \sum_{i \in T'} W_i^{v'} - \beta^*\right\|_2 & = \left\|\frac{1}{|T'|} \sum_{i \in T'} v' +  x_ix_i^T (\beta^* - v')  + x_iz_i - \beta^* \right\|_2 \notag \\
& = \left\|\frac{1}{|T'|} \sum_{i \in T'} x_ix_i^T \Delta'  + x_iz_i - \Delta' \right\|_2 \notag \\
& \leq \left\|\frac{1}{|T'|} \sum_{i \in T'} x_ix_i^T\Delta + x_iz_i - \Delta \right\|_2 + \left\|\frac{1}{|T'|} \sum_{i \in T'}x_ix_i^T(\Delta' - \Delta)\right\|_2 + \|\Delta' - \Delta\|_2 \notag \\
& \le \left\|\frac{1}{|T'|} \sum_{i \in T'} x_ix_i^T\Delta + x_iz_i - \Delta \right\|_2 + \frac{d\eta}{\epsilon'} + \eta,
\end{align}
where we have used the facts that $\|x_i\|_2 \le \sqrt{\frac{p}{\epsilon'}}$ for $i \in T_0$ and $\|\Delta' - \Delta\|_2 \le \eta$ in the last line. Furthermore, note that the first term on the right-hand side of inequality~\eqref{EqnMilk}, which can be written as $\left\|\frac{1}{|T'|} \sum_{i \in T'} W_i^{v} - \beta^*\right\|_2$, can be upper-bounded by $\sigma_* \delta$ using the stability of the set $T_v$, since $T' \subseteq T_v$ and
\begin{equation*}
|T'| \ge \left(1 - \frac{C\epsilon'}{2}\right) |T'_{v'}| \ge \left(1 - \frac{C\epsilon'}{2}\right) (1-5\epsilon')n \ge (1-C\epsilon) |T_v|,
\end{equation*}
if $C \ge 10$. Thus, we conclude that
\begin{equation*}
\left\|\frac{1}{|T'|} \sum_{i \in T'} W_i^{v'} - \beta^*\right\|_2 \le 2 \sigma_*\delta,
\end{equation*}
by choosing $\eta \le \frac{\sigma_* \delta}{1 + d/\epsilon'}$. Note that since $r = \Theta(\sigma_*)$ and $\delta = \Omega(\sqrt{\epsilon'})$, this may be accomplished with the choice
\begin{equation}
\label{EqnEta1}
\eta = O\left(\frac{r \sqrt{\epsilon'}}{1+d/\epsilon'}\right).
\end{equation}

We also need to establish a spectral norm bound on the second moment matrix. Denoting
\begin{align*}
a_i & := x_ix_i^T \Delta + x_iz_i  - \Delta, \\
b_i & := x_ix_i^T (\Delta' - \Delta), \\
c & := \Delta- \Delta',
\end{align*}
we see that
\begin{align}
\label{EqnCookies}
& \left\|\frac{1}{|T'|} \sum_{i \in T'} \left(W_i^{v'} - \beta^*\right) \left(W_i^{v'} - \beta^*\right)^T - \sigma_*^2I\right\|_2 \notag \\
& = \left\|\frac{1}{|T'|} \sum_{i \in T'} (x_ix_i^T \Delta' + x_iz_i - \Delta')(x_ix_i^T \Delta' + x_iz_i - \Delta') - \sigma_*^2I \right\|_2 \notag \\
&=\left\|\frac{1}{|T'|} \sum_{i \in T'} (a_i + b_i + c)(a_i + b_i + c)^T - \sigma_*^2 I\right\|_2 \notag \\
&\leq \left\|\frac{1}{|T'|} \sum_{i \in T'} a_ia_i^T - \sigma_*^2I\right\|_2 +  \left\|\frac{1}{|T'|} \sum_{i \in T'} b_ib_i^T\right\|_2 +  \left\|\frac{1}{|T'|} \sum_{i \in T'} cc^T\right\|_2 \notag \\
& \qquad + 2\left\|\frac{1}{|T'|} \sum_{i \in T'} a_ib_i^T\right\|_2  + 2\left\|\frac{1}{|T'|} \sum_{i \in T'} a_ic^T\right\|_2 +  \left\|\frac{1}{|T'|} \sum_{i \in T'} b_ic^T\right\|_2.
\end{align}
By the stability of $T_v$, we have
\begin{align*}
\left\|\frac{1}{|T'|} \sum_{i \in T'} a_ia_i^T - \sigma_*^2 I\right\|_2 \leq \frac{\sigma_*^2 \delta^2}{C\epsilon'}.
\end{align*}
Further note that
\begin{align*}
\|a_i\|_2 & \le \frac{d \eta}{\epsilon'} + \sqrt{\frac{p}{\epsilon'}} \cdot \frac{\sigma}{\sqrt{\epsilon'}} + \eta, \\
\|b_i\|_2 & \le \frac{d\eta}{\epsilon'}, \\
\|c\|_2 & \le \eta.
\end{align*}
Thus, the right-hand expression in inequality~\eqref{EqnCookies} may be upper-bounded by
\begin{align*}
& \frac{\sigma_*^2 \delta^2}{C\epsilon'} + \frac{p^2 \eta^2}{(\epsilon')^2} + \eta^2 + 2\left(\frac{d\eta}{\epsilon'} + \eta\right) \left(\frac{d \eta}{\epsilon'} + \sqrt{\frac{p}{\epsilon'}} \cdot \frac{\sigma}{\sqrt{\epsilon'}} + \eta\right) + \frac{2d\eta^2}{\epsilon'} \\
& \le \frac{\sigma_*^2 \delta^2}{C\epsilon'} + \eta\left(\frac{p^2r}{(\epsilon')^2} + r + 2\left(\frac{p}{\epsilon'} + 1\right) \left(\frac{dr}{\epsilon'} + \frac{\sigma \sqrt{p}}{\epsilon'} + r\right) + \frac{2dr}{\epsilon'}\right) \\
& \le \frac{\sigma_{*}^2 \delta^2}{C\epsilon'/2},
\end{align*}
by choosing
\begin{equation}
\label{EqnEta2}
\eta = O\left(\frac{r}{p^2/(\epsilon')^2 + 1 + 2(d/\epsilon' + 1)(2d/\epsilon' + 1) + 2d/\epsilon'}\right),
\end{equation}
using the facts that $r = \Theta(\sigma_*)$ and $\delta = \Omega(\sqrt{\epsilon'})$.

Therefore, we see that defining $f$ appropriately and taking $\eta = \frac{r\sqrt{\epsilon'}}{f(d/\epsilon')}$ satisfies conditions~\eqref{EqnEta1} and~\eqref{EqnEta2} simultaneously, completing the proof.

\section{Additional simulations}
\label{AppSims}

We include additional experiment details in this section. Figure~\ref{fig:hub_app} shows how the choice of the tuning parameter $\gamma$ in the Huber loss affects the resulting error.
We note that Huber regression with filtering is quite robust to the choice of $\gamma$.  
\begin{figure}[!ht]
 \centering
    \begin{minipage}{0.8\textwidth}
        \centering
		\includegraphics[width=\textwidth]{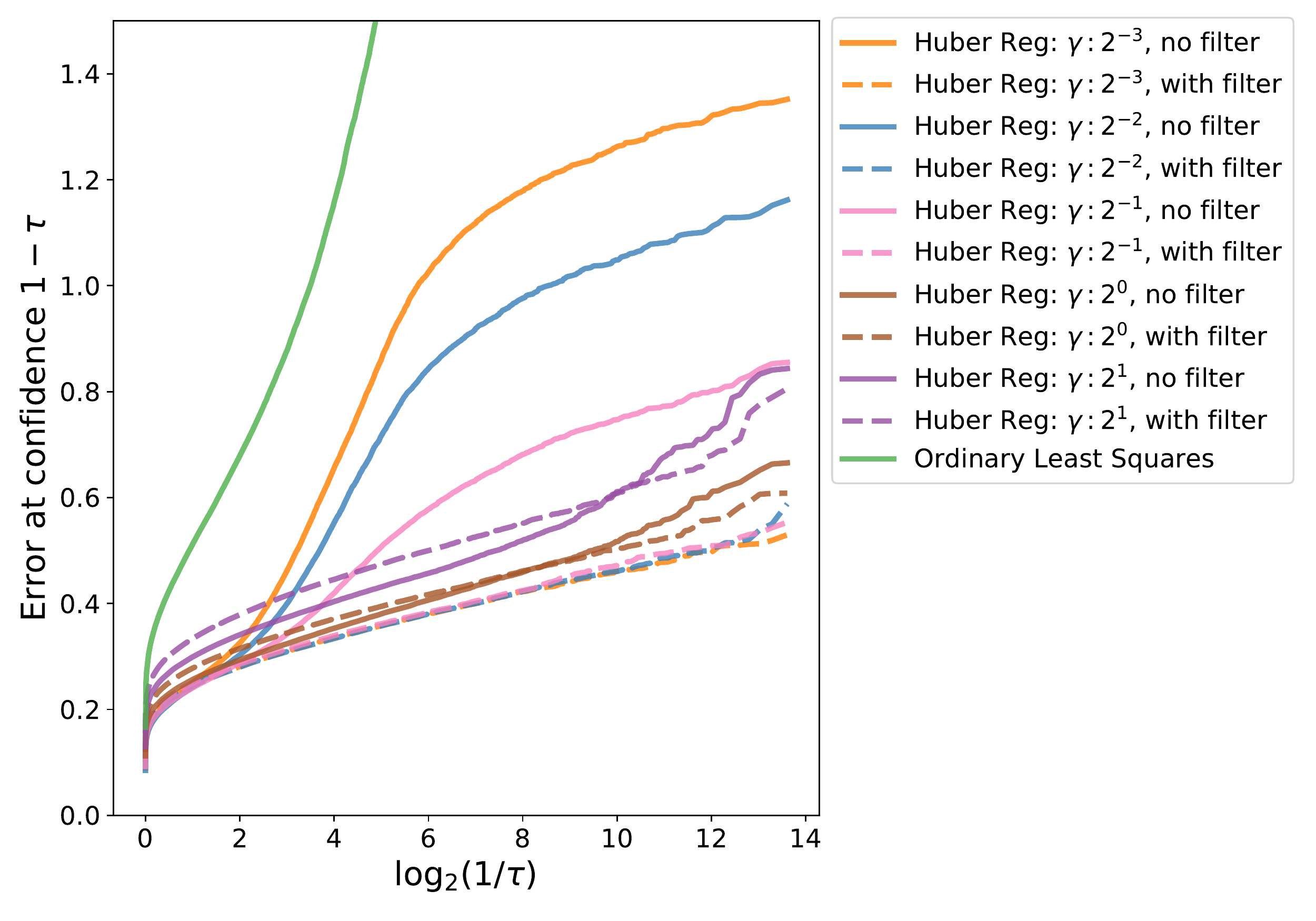}
    \end{minipage}%
\caption{Plot showing the effect of covariate filtering on Huber regression $(n=200, p = 40)$ for different values of $\gamma$.
The error is measured in terms of $\ell_2$-error, i.e., $\|\widehat{\beta} - \beta^*\|_2$.
Solid lines corresponds to ``vanilla'' version of the estimators (no filtering step), and dashed lines correspond to filtered versions, where the filtering step removes $10$ points out of $200$ points. We note that the performance of Huber regression with filtering is not greatly affected by the choice of $\gamma$.  
 }
\label{fig:hub_app}
 \end{figure}

Figure~\ref{fig:lts_app} shows how the choice of the thresholding parameter $m$ in LTS affects the resulting error.
We note that the LTS with filtering is also quite robust to the choice of $m$.  
\begin{figure}[!ht]
 \centering
    \begin{minipage}{0.8\textwidth}
        \centering
		\includegraphics[width=\textwidth]{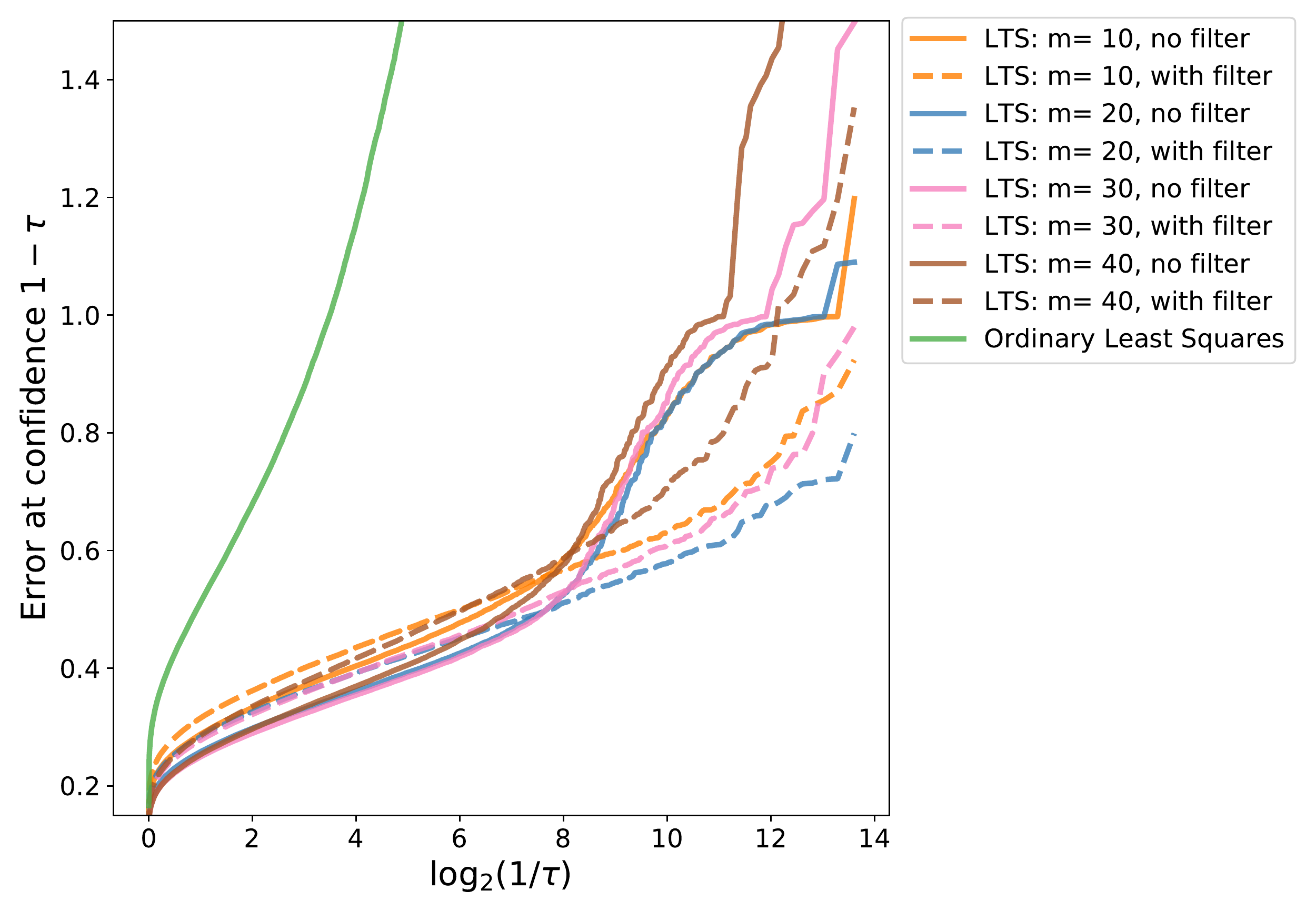}
    \end{minipage}%
\caption{Plot showing the effect of covariate filtering on LTS regression $(n=200, p = 40)$ for different values of $m$.
The error is measured in terms of $\ell_2$-error, i.e., $\|\widehat{\beta} - \beta^*\|_2$.
Solid lines corresponds to ``vanilla'' version of the estimators (no filtering step), and dashed lines correspond to filtered versions, where the filtering step removes $10$ points out of $200$ points. We note that the performance of LTS with filtering is not greatly affected by the choice of $m$.  
 }
\label{fig:lts_app}
 \end{figure}
\end{document}